\documentclass[11pt]{amsart}
\usepackage[top=1.5in, bottom=1.5in, left=1.4in, right=1.4in]{geometry} 
\usepackage{amssymb}
\usepackage{amsmath}
\usepackage{mathrsfs}
\input xy
\xyoption{all}
\usepackage{mathdots}
\usepackage{mathtools}

\usepackage{url}
\usepackage{enumerate}
\usepackage{csquotes}
\usepackage{array}
\usepackage{hyperref}
\usepackage{afterpage}

\usepackage[all]{xy}
\usepackage{graphicx}  
\usepackage{tikz}
\usetikzlibrary{decorations.pathreplacing}
\usepackage{tkz-graph}
\usetikzlibrary{arrows}
\usetikzlibrary{decorations.markings}

\newtheorem{thm}{Theorem}[section]
\newtheorem{lemma}[thm]{Lemma}

\newtheorem{cor}[thm]{Corollary}
\newtheorem{prop}[thm]{Proposition}

\newtheorem{Definition}[thm]{Definition}
\newenvironment{definition}
  {\begin{Definition}\rm}{\end{Definition}}
\newtheorem{Example}[thm]{Example}
\newenvironment{example}
  {\pushQED{\qed}\begin{Example} \rm}{\popQED \end{Example}}
\newtheorem{Remark}[thm]{Remark}
\newenvironment{remark}
  {\pushQED{\qed}\begin{Remark}\rm}{\popQED \end{Remark}}
  
\newtheorem{innercustomthm}{Theorem}
\newenvironment{customthm}[1]
  {\renewcommand\theinnercustomthm{#1}\innercustomthm}
  {\endinnercustomthm}

\newtheorem{innercustomlem}{Lemma}
\newenvironment{customlem}[1]
  {\renewcommand\theinnercustomlem{#1}\innercustomlem}
  {\endinnercustomlem}
  
\numberwithin{equation}{section}

\allowdisplaybreaks

\usepackage{etoolbox}
\apptocmd{\sloppy}{\hbadness 10000\relax}{}{}


\pgfarrowsdeclare{mytip}{mytip}
 {
  \pgfarrowsleftextend{-5.59\pgflinewidth}
  \pgfarrowsrightextend{4.5\pgflinewidth}}
 {
  \pgfpathmoveto{\pgfpoint{4.5\pgflinewidth}{0pt}} 
  \pgfpathlineto{\pgfpoint{-0.63\pgflinewidth}{1.09\pgflinewidth}} 
  \pgfpathlineto{\pgfpoint{-5.48\pgflinewidth}{3.06\pgflinewidth}} 
  \pgfpathlineto{\pgfpoint{-5.59\pgflinewidth}{3.0\pgflinewidth}}
  \pgfpathlineto{\pgfpoint{-3.95\pgflinewidth}{0pt}} 
  \pgfpathlineto{\pgfpoint{-5.59\pgflinewidth}{-3.0\pgflinewidth}}
  \pgfpathlineto{\pgfpoint{-5.48\pgflinewidth}{-3.06\pgflinewidth}}
  \pgfpathlineto{\pgfpoint{-0.63\pgflinewidth}{-1.09\pgflinewidth}}
  \pgfusepathqfill
 }
 
 \pgfarrowsdeclare{mytip2}{mytip2}{%
  \setlength{\arrowsize}{1pt}
  \addtolength{\arrowsize}{.5\pgflinewidth}
  \pgfarrowsrightextend{0}
  \pgfarrowsleftextend{-5\arrowsize}
}{%
  \setlength{\arrowsize}{1pt}
  \addtolength{\arrowsize}{.5\pgflinewidth}
  \pgfpathmoveto{\pgfpoint{-5\arrowsize}{1.5\arrowsize}}
  \pgfpathlineto{\pgfpointorigin}
  \pgfpathlineto{\pgfpoint{-5\arrowsize}{-1.5\arrowsize}}
  \pgfusepathqstroke
}

\def\SetBasicGraph { 	
	\SetVertexMath
	\GraphInit[vstyle=Classic]
	\SetUpVertex[MinSize=0pt]
	\SetVertexLabel
	\tikzset{VertexStyle/.style = {shape = circle,fill = black,minimum size = 0pt,inner sep=0.75pt}}
	\SetUpEdge[color=black]
	\tikzset{->-/.style={decoration={ markings, mark=at position 0.8 with {\arrow{>}}},postaction={decorate}}}
}

\def\SetFancyGraph {
	\SetVertexMath
	\GraphInit[vstyle=Art]
	\SetUpVertex[MinSize=2pt]
	\SetVertexLabel
	\tikzset{VertexStyle/.style = {shape = circle,shading = ball,ball color = black,inner sep = 1.5pt}}
	\SetUpEdge[color=black]
	\tikzset{->-/.style={decoration={ markings, mark=at position 0.8 with {\arrow{>}}},postaction={decorate}}}
	\tikzset{->--/.style={decoration={ markings, mark=at position 0.55 with {\arrow{>}}},postaction={decorate}}}
}

\def\Triangle[#1,#2,#3,#4,#5] {
	\SetBasicGraph
	\Vertex[NoLabel,x=#1+-0.2,y=#2+-0.25]{v_1}
	\Vertex[NoLabel,x=#1+0.15,y=#2+0.25]{v_2}
	\Vertex[NoLabel,x=#1+0.5,y=#2+-0.25]{v_3}
	\ifthenelse{#3=0}{\Edges[style={thick}](v_1,v_2)}{}
	\ifthenelse{#3=1}{\Edges[style={->-,>=mytip2,thick}](v_1,v_2)}{}
	\ifthenelse{#3=-1}{\Edges[style={->-,>=mytip2,thick}](v_2,v_1)}{}
	\ifthenelse{#3=2}{\Edges[style={->,>=mytip2,thick}](v_1,v_2) \Edges[style={->,>=mytip2,thick}](v_2,v_1)}{}
	\ifthenelse{#4=0}{\Edges[style={thick}](v_1,v_3)}{}
	\ifthenelse{#4=1}{\Edges[style={->-,>=mytip2,thick}](v_3,v_1)}{}
	\ifthenelse{#4=-1}{\Edges[style={->-,>=mytip2,thick}](v_1,v_3)}{}
	\ifthenelse{#4=2}{\Edges[style={->,>=mytip2,thick}](v_1,v_3) \Edges[style={->,>=mytip2,thick}](v_3,v_1)}{}
	\ifthenelse{#5=0}{\Edges[style={thick}](v_2,v_3)}{}
	\ifthenelse{#5=1}{\Edges[style={->-,>=mytip2,thick}](v_3,v_2)}{}
	\ifthenelse{#5=-1}{\Edges[style={->-,>=mytip2,thick}](v_2,v_3)}{}
	\ifthenelse{#5=2}{\Edges[style={->,>=mytip2,thick}](v_2,v_3) \Edges[style={->,>=mytip2,thick}](v_3,v_2)}{}
}

\title{Fourientation activities and the Tutte polynomial}
\author{Spencer Backman}
\author{Sam Hopkins}
\author{Lorenzo Traldi}

\begin{document}

\begin{abstract}
A fourientation of a graph $G$ is a choice for each edge of the graph whether to orient that edge in either direction, leave it unoriented, or biorient it.  We may naturally view fourientations as a mixture of subgraphs and graph orientations where unoriented and bioriented edges play the role of absent and present subgraph edges, respectively.  Building on work of Backman and Hopkins~(2015), we show that given a linear order and a reference orientation of the edge set, one can define activities for fourientations of $G$ which allow for a new~$12$ variable expansion of the Tutte polynomial $T_G$.  Our formula specializes to both an orientation activities expansion of $T_G$ due to Las~Vergnas~(1984) and a generalized activities expansion of~$T_G$ due to Gordon and Traldi~(1990).
\end{abstract}

\maketitle

\section{Introduction}

This paper concerns the Tutte polynomial and its relationship to graph orientations. The Tutte polynomial~$T_G(x,y)$ of a graph~$G$ is among the most well-studied graph polynomials; see~\cite{welsh1999tutte}~\cite{welsh2000potts}. The connection between orientations and the Tutte polynomial goes back at least to the seminal work of Stanley~\cite{stanley1973acyclic} who showed that the number of acyclic orientations of~$G$ is~$T_G(2,0)$. For more on the history of orientations and the Tutte polynomial, including the work of many authors who showed that various classes of orientations are enumerated by the Tutte polynomial, see~\cite[\S1.1]{backman2015fourientations}. 

Backman and Hopkins~\cite{backman2015fourientations} discussed the connection between the Tutte polynomial and the~\emph{fourientations} of a graph. A fourientation of a graph is a kind of generalized orientation, given by specifying a set of orientations for each edge; the set may include neither orientation, either one, or both. The work of Backman and Hopkins was preceded by results of Gessel and Sagan~\cite{gessel1996tutte}, as well as Hopkins and Perkinson~\cite{hopkins2016bigraphical}, and Backman~\cite{backman2014partial}, showing that several classes of \emph{partial orientations} are also enumerated by the Tutte polynomial. Backman and Hopkins~\cite{backman2015fourientations} used fourientations to put these results about partial orientations, as well as the classical Tutte polynomial evaluations for orientations, in a unified framework.

The purpose of the present paper is to define activities for fourientations. Here ``activities'' refers to a pair of mappings defined by Tutte~\cite{tutte1954contribution}, which use a total order on $E(G)$ to associate two sets of edges to each spanning tree of a connected graph $G$ (\emph{internally active} edges are defined using cuts, and \emph{externally active} edges are defined using cycles). In fact, Tutte's original definition of his polynomial was in terms of internal and external activities of spanning trees. Las~Vergnas~\cite{las1984tutte} defined notions of cut and cycle activity for orientations, and found an expansion of the Tutte polynomial in terms of orientation activities that recaptures Stanley's result. In his final article, Las Vergnas refined this formula~\cite{las2012tutte} in a way which allowed him to recover additional evaluations such as Greene and Zaslavsky's interpretation of $T(1,0)$~\cite{greene1983interpretation}.  Las Vergnas's refined orientation expansion of the Tutte polynomial is very analogous to a generalized activities formula of Gordon-Traldi~\cite{gordon1990generalized}, which expresses the Tutte polynomial as a sum over spanning subgraphs. Gioan and Las Vergnas~\cite{gioan2009active}~\cite{gioanpreprint}~\cite{gioan2005activity} produced a bijection between certain classes of orientations and subgraphs, which is canonical up to edge order and allows for the orientation expansion of the Tutte polynomial to derived directly from the subgraph expansion. In this paper we offer a fourientation activities formula that simultaneously generalizes both the Las~Vergnas and Gordon-Traldi formulas and in addition recovers the main result of~\cite{backman2015fourientations}. 

Before stating our formula, we need to specify some notation and terminology. Let~$G$ be an undirected graph which may have multiple edges and/or loops. We use~$V(G)$ to denote the vertex set of $G$, and $E(G)$ the edge set of $G$. Throughout we will use~$n := \left\vert V(G) \right\vert$ for the number of vertices of $G$ and~\mbox{$g := \left\vert E(G) \right\vert - \left\vert V(G) \right\vert + \kappa$} for the cyclomatic number, where $\kappa := \kappa(G)$ is the number of connected components of~$G$. For basic background on and terminology for graphs, including such concepts as \emph{cycles}, \emph{cuts}, \emph{deletion-contraction}, and the \emph{Tutte polynomial}, see~\cite[\S2.1]{backman2015fourientations} and~\cite{welsh1999tutte}~\cite{welsh2000potts}. Recall that a \emph{spanning subgraph} of $G$ is a subgraph $H=(V(G),E(H))$ with $E(H) \subseteq E(G)$; i.e., it is a subgraph that includes all the vertices of $G$ and some edges of $G$. We identify such a subgraph $H$ with its subset of edges~$S := E(H)$. Therefore we let~$\mathcal{S}(G) := 2^{E(G)}$ denote the set of spanning subgraphs of $G$.

In order to talk about orientations of $G$ it is helpful to have a fixed reference orientation~$O_{\mathrm{ref}}$. The reference orientation~$O_{\mathrm{ref}}$ is a choice for each edge $e = \{u,v\} \in E(G)$ of a positive direction~$e^{+} = (u,v)$ and therefore also a negative direction $e^{-} = (v,u)$. With respect to~$O_{\mathrm{ref}}$ an orientation of $G$ is then just a subset $O \subseteq \mathbb{E}(G)$ of the set~$\mathbb{E}(G):=\{e^+,e^-\colon e \in E(G)\}$, which satisfies~$\left\vert \{e^+,e^-\} \cap O \right\vert  = 1$ for all $e \in E(G)$. Here we identify an orientation~$O$ (which we have defined to be just a set of formal symbols) with the set of directed edges~$\{(u,v)\colon e^{\delta} = (u,v) \in O, \delta \in \{+,-\}\}$ and this identification depends implicitly on $O_{\mathrm{ref}}$. We use $\mathcal{O}(G)$ to denote the set of orientations of $G$. Note that $\left\vert \mathcal{O}(G) \right\vert  = 2^{\left\vert E(G) \right\vert }$ even when $G$ has loops or multiple edges. 

Recall from Backman-Hopkins~\cite{backman2015fourientations} that a \emph{fourientation} of a graph $G$ with respect to some fixed reference orientation $O_{\mathrm{ref}}$ is just an arbitrary subset of~$\mathbb{E}(G)$. We use~$\mathcal{O}^{4}(G)$ to denote the set of fourientations of $G$. (The superscript is intended not as an exponent but merely as a reference to the word ``four''.) Let~$O \in \mathcal{O}^{4}(G)$. We say that~$e \in E(G)$ is \emph{unoriented} in $O$ if $\{e^{+},e^{-}\} \cap O = \varnothing$ and we say that $e$ is \emph{bioriented} in~$O$ if $\{e^{+},e^{-}\} \subseteq O$. We say $e$ is \emph{simply oriented}, or just \emph{oriented}, in $O$ if it is neither unoriented nor bioriented. Let $O^o \in \mathcal{S}(G)$ denote the set of oriented edges of~$O$, $O^u \in \mathcal{S}(G)$ the set of unoriented edges, and $O^b \in \mathcal{S}(G)$ the set of bioriented edges. Let $O^+$ denote the set of oriented edges of~$O$ oriented in agreement with $O_{\mathrm{ref}}$ and $O^{-}$ the set of oriented edges oriented in disagreement with $O_{\mathrm{ref}}$, so that $O^{o} = O^{+} \sqcup O^{-}$.

\begin{thm} \label{thm:main}
There exist fourientation cut activities $I\colon \mathcal{O}^{4}(G)\to \mathcal{S}(G)$ and cycle activities~\mbox{$L\colon \mathcal{O}^{4}(G) \to \mathcal{S}(G)$}, as well as a map $\varphi\colon \mathcal{O}^{4}(G) \to \mathcal{S}(G)$, such that
{\footnotesize \begin{gather*}
(k_1+m)^{n-\kappa}(k_2+l)^{g}T_G\left(\frac{k_1x+k_2w+m\hat{x}+l\hat{w}}{k_1+m},\frac{k_2y+k_1z+l\hat{y}+m\hat{z}}{k_2+l}\right) =  \\
\sum_{O \in \mathcal{O}^{4}(G)} \hspace{-0.3cm} k_1^{\left\vert O^{o} \cap \varphi(O) \right\vert}k_2^{\left\vert O^{o} \setminus \varphi(O) \right\vert}l^{\left\vert O^{u} \right\vert}m^{\left\vert O^{b}\right\vert}x^{\left\vert I^{+}(O) \right\vert} w^{\left\vert I^{-}(O) \right\vert} {\hat{x}}^{\left\vert I^{b}(O)\right\vert} {\hat{w}}^{\left\vert I^{u}(O) \right\vert} y^{\left\vert L^{+}(O) \right\vert} z^{\left\vert L^{-}(O) \right\vert}{\hat{y}}^{\left\vert L^{u}(O) \right\vert} {\hat{z}}^{\left\vert L^{b}(O) \right\vert},
\end{gather*}}
where we use the notation $I^{o}(O) := I(O) \cap O^{o}$, $L^{b}(O) := L(O) \cap O^{b}$ and so on.
\end{thm}

All parameters in Theorem~\ref{thm:main} are described later in the paper; see Figure~\ref{fig:parameters}.

Now let us briefly outline the structure of the paper. In Section~\ref{sec:setup} we recall the formulas of Gordon-Traldi~\cite{gordon1990generalized} and Las Vergnas~\cite{las1984tutte}. In Section~\ref{sec:main} we provide more details about the significance of Theorem~\ref{thm:main}, and we connect it to a number of known results; in particular, the formulas of Gordon-Traldi and Las Vergnas are both evaluations of the formula of Theorem~\ref{thm:main}. We also outline the proof of Theorem~\ref{thm:main}, which is modeled on Las~Vergnas's justification of his orientation activity formula~\cite{las1984tutte}. (In contrast with Las~Vergnas's discussion, though, our proof is strictly graph-theoretic.) The centerpiece of the proof is the construction of the surjection~$\varphi$ from fourientations to subgraphs. The activity mappings $I$ and $L$ are derived from $\varphi$ using generalized activities, and this derivation gives rise to the connection between Theorem~\ref{thm:main} and the formula of Gordon-Traldi~\cite{gordon1990generalized}. The connection between Theorem~\ref{thm:main} and the formula of Las Vergnas~\cite{las1984tutte} arises from the fact that although the precise significance of $I$ and $L$ is obscured by the recursive definition of $\varphi$, the sets $I^o$, $L^o$, $I^u$ and $L^b$ have separate descriptions which do not require recursion. Our construction of~$\varphi$ is conditional on a certain key lemma, Lemma~\ref{lem:key}, whose long and technical proof we include in an appendix. In Section~\ref{sec:ex} we give an example of the surjection~$\varphi$ and discuss how~$\varphi$ can be rather subtle. In Section~\ref{sec:future} we discuss some possible extensions of our results and future directions. In Appendix~\ref{sec:genact} we review the proofs of some results related to the generalized activities of Gordon-Traldi in order to make this paper self-contained. And in Appendix~\ref{sec:proofkeylem} we finally prove the key lemma.

\noindent {\bf Acknowledgements}: We are grateful to Emeric Gioan for illuminating conversations about activity-preserving bijections, and for sharing~\cite{gioanpreprint} with us. We are also grateful to two anonymous readers, whose advice improved the readability of the paper.  The second author was supported by NSF grant~$\#1122374$. An extended abstract of this paper appeared in the proceedings of the 28th International Conference on Formal Power Series and Algebraic Combinatorics (FPSAC '16)~\cite{backman2016abstract}. 

\section{Background results} \label{sec:setup}

In order to fix a convention, let us take as definition of the Tutte polynomial~$T_G(x,y)$ of~$G$ the following:
\[T_G(x,y) := \sum_{S \in \mathcal{S}(G)} (x-1)^{\kappa(S)-\kappa} (y-1)^{\left\vert S \right\vert + \kappa(S) - n},\]
where $\kappa(S) := \kappa(V(G),S)$ is the number of connected components of the spanning subgraph~$(V(G),S)$. This is known as the \emph{corank-nullity} generating function expansion for the Tutte polynomial (see for instance~\cite[(2.5)]{welsh1999tutte}~\cite[(8)]{welsh2000potts}). When written in this form, the closely related polynomial~$T_G(1+x,1+y)$ is often called the \emph{Whitney rank generating function} of~$G$ in recognition of the fundamental contributions of Whitney~\cite{whitney1935abstract}.

\subsection{Generalized activities}

Recall that a cut $Cu$ of our graph~$G$ is given by a subset $E(Cu) \subseteq E(G)$, which is minimal among nonempty subsets $S \in \mathcal{S}(G)$ such that $\kappa(E(G) \setminus S)> \kappa(G)$. A cycle $Cy$ is given by a subset $E(Cy) \subseteq E(G)$, which is minimal among nonempty subsets $S \in \mathcal{S}(G)$ such that no cut of $G$ includes precisely one element of $S$. If~$S \in \mathcal{S}(G)$ then abusing notation, we say that a cut $Cu$ of $G$ is a \emph{cut of~$S$} if~$E(Cu) \cap S = \varnothing$. Similarly we say that a cycle~$Cy$ of $G$ is a \emph{cycle of~$S$} if $E(Cy) \subseteq S$. Let $<$ be a total order on~$E(G)$. Gordon-Traldi~\cite{gordon1990generalized} define generalized cut and cycle activities for arbitrary spanning subgraphs of $G$ as follows. An edge~$e \in E(G)$ is \emph{cut active} with respect to $S$ if it is the min edge (where ``min'' is the minimum according to~$<$) in~$E(Cu)$ for $Cu$ some cut of~$S \setminus \{e\}$. We use~$\hat{I}(S)$ (where $I$ is for \emph{isthmus}) to denote the set of cut active edges of~$S$. Similarly,~$e$ is \emph{cycle active} with respect to $S$ if it is the min edge in $E(Cy)$ for $Cy$ some cycle of~$S \cup \{e\}$. We use~$\hat{L}(S)$ (where $L$ is for \emph{loop}) to denote the set of cycle active edges of $S$. (The notation~$I(S)$, $L(S)$ is used in~\cite{gordon1990generalized}, but we use $\hat{I}(S)$, $\hat{L}(S)$ here in order to distinguish between generalized activities and other notions of activity defined below, using orientations and fourientations.) The term~\emph{generalized activities} is used because these definitions directly generalize Tutte's original definition of activity~\cite{tutte1954contribution} from spanning trees to spanning subgraphs. Observe that the maps~$\hat{I},\hat{L}\colon \mathcal{S}(G) \to \mathcal{S}(G)$ depend on the fixed edge order~$<$ but we leave this dependence implicit. One easy consequence of the definitions is that~$\hat{I}(S) \cap \hat{L}(S) = \varnothing$ for all $S \in \mathcal{S}(G)$. We now review a few other basic results about the generalized activities~$\hat{I}$ and~$\hat{L}$. For the sake of completeness, we present proofs of these results in Appendix~\ref{sec:genact}. 

\begin{lemma}[{Gordon-Traldi~\cite[Theorem 2]{gordon1990generalized}}] \label{lem:intervals}
Let $S, T \in \mathcal{S}(G)$ with~$S \setminus (\hat{I}(S) \cup \hat{L}(S)) \subseteq T \subseteq S \cup \hat{I}(S) \cup \hat{L}(S)$. Then~$\hat{I}(S) = \hat{I}(T)$ and~$\hat{L}(S) = \hat{L}(T)$.
\end{lemma}

\begin{definition} \label{def:crapointervals}
For $S,T \in \mathcal{S}(G)$ we use $S\sim T$ to denote
\[ S \setminus (\hat{I}(S) \cup \hat{L}(S)) \subseteq T \subseteq S \cup \hat{I}(S) \cup \hat{L}(S).\]
Lemma~\ref{lem:intervals} guarantees that $\sim$ is an equivalence relation on $\mathcal{S}(G)$. Observe that the equivalence class to which $S \in \mathcal{S}(G)$ belongs is $[ S \setminus (\hat{I}(S) \cup \hat{L}(S)), S \cup \hat{I}(S) \cup \hat{L}(S)]$, an interval in the poset of spanning subgraphs~$\mathcal{S}(G)$ ordered by containment. This partition of~$\mathcal{S}(G)$ into intervals was first demonstrated by Crapo~\cite{crapo1969tutte}, so we refer to the equivalence classes of subgraphs under~$\sim$ as~\emph{Crapo intervals}.
\end{definition}

\begin{lemma} [{Gordon-Traldi~\cite[Theorem 1]{gordon1990generalized}}]\label{lem:rankact}
We have $\left\vert\hat{I}(S) \setminus S\right\vert = \kappa(S) - \kappa$ and $\left\vert \hat{L}(S) \cap S \right\vert = \kappa(S) + \left\vert S \right\vert-n$ for all~$S \in \mathcal{S}(G)$.
\end{lemma}

The main result about the subgraph activities~$\hat{I}$ and~$\hat{L}$ is~\cite[Theorem 3]{gordon1990generalized}, which provides the following expansion of the Tutte polynomial~$T_G(x,y)$ of~$G$. 

\begin{thm}[{Gordon-Traldi~\cite[Theorem 3]{gordon1990generalized}; see also~\cite[Theorem 3.5]{las2013tutte}}] \label{thm:gordontraldi}
For any graph~$G$ and any total order $<$ on $E(G)$ we have
\[T_G(\hat{x}+\hat{w},\hat{y}+\hat{z}) = \sum_{S \in \mathcal{S}(G)} {\hat{x}}^{\left\vert \hat{I}(S)\cap S\right\vert} {\hat{w}}^{\left\vert \hat{I}(S)\setminus S\right\vert} {\hat{y}}^{\left\vert \hat{L}(S)\setminus S\right\vert} {\hat{z}}^{\left\vert \hat{L}(S)\cap S\right\vert}.\]
\end{thm}

The expansion given in the theorem recovers many other expansions of the Tutte polynomial. For instance, when $G$ is connected, taking $\hat{w} :=0$ and $\hat{z} := 0$ recovers Tutte's spanning tree activity expansion for $T_G$~\cite[(13)]{tutte1954contribution}:
\[T_G(\hat{x},\hat{y}) = \sum_{\substack{\textrm{$T$ spanning tree}\\ \textrm{of $G$} }} {\hat{x}}^{\left\vert \hat{I}(T)\right\vert } {\hat{y}}^{\left\vert \hat{L}(T)\right\vert }.\]
When $G$ is not connected there is an entirely analogous expansion, indexed by maximal forests. Also, taking $\hat{x} := 1$ and $\hat{y} := 1$ in Theorem~\ref{thm:gordontraldi} recovers the corank-nullity generating function expansion for the Tutte polynomial:
\[T_G(1+\hat{w},1+\hat{z}) = \sum_{S \in \mathcal{S}(G)} {\hat{w}}^{\left\vert \hat{I}\setminus S \right\vert} {\hat{z}}^{\left\vert \hat{L}(S) \cap S\right\vert} = \sum_{S \in \mathcal{S}(G)} {\hat{w}}^{\kappa(S)-\kappa} {\hat{z}}^{\kappa(S)+\left\vert S \right\vert -n}.\]
(The second equality follows from Lemma~\ref{lem:rankact}.) A two-variable formula of Gordon-Traldi~\cite{gordon1990generalized} is obtained from Theorem~\ref{thm:gordontraldi} by taking $\hat{x}:=\hat{w}:= \hat{u}/2$ and $\hat{y}:=\hat{z} := \hat{v}/2$: 
\[T_G(\hat{u},\hat{v}) = \sum_{S \in \mathcal{S}(G)} \left(  \frac{\hat{u}}{2}\right)^{\left\vert \hat{I}(S) \right\vert} \left(  \frac{\hat{v}}{2}\right)^{\left\vert  \hat{L}(S) \right\vert}.\] Each Crapo interval contains precisely one spanning tree (or maximal forest, if $G$ is not connected), so this formula yields the spanning tree activity formula directly, through term collection. Other expansions related to Theorem~\ref{thm:gordontraldi} are given in~\cite{gordon1990generalized} and~\cite{las2013tutte}. Notice that the corank-nullity formula implies that the number of spanning subgraphs of $G$ is~$T_G(2,2)$, the number of spanning forests is~$T_G(2,1)$, the number of spanning subgraphs with $\kappa(S)=\kappa$ is $T_G(1,2)$, and the number of maximal forests (spanning trees if $G$ is connected) is $T_G(1,1)$. More generally, by taking the variables~$\hat{x},\hat{y},\hat{w},\hat{z} \in \{0,1\}$, Theorem~\ref{thm:gordontraldi} gives combinatorial interpretations for all evaluations $T_G(a,b)$ with integer~$0 \leq a,b \leq 2$ in terms of spanning subgraphs and activities.

\subsection{Orientation activities}

There is a very analogous story for orientation activities due to Las~Vergnas~\cite{las1984tutte}.\footnote{The idea of defining activities for orientations was first introduced by Berman~\cite{berman1977dichromate}, but his account of their connection with the Tutte polynomial was not correct as he counted active (co)circuits rather than active elements. See the footnote on page~370 of~\cite{las1984tutte} for details.} If $C$ is a cut of $G$, then there is a subset $X \subset V(G)$ such that $C$ includes the edges of $G$ which connect a vertex from $X$ to a vertex outside $X$. There are two ways to consistently orient $C$, into $X$ and out from $X$. Similarly, if $C$ is a cycle of $G$ then there are two ways to consistently orient $C$. A~\emph{directed cut (cycle) of }$G$ is obtained by choosing one of the two orientations of a cut (cycle) of $G$. We use $\mathbb{E}(C) \subseteq \mathbb{E}(G)$ to denote the set of directions for the edges of~$C$ which respect the chosen orientation. If $O$ is an orientation of $G$ then we say that $C$ is a~\emph{directed cut (cycle) of }$O$ if $\mathbb{E}(C) \subseteq O$. In other words, a directed cut (cycle) of $O$ is a cut (cycle) of $G$ whose edges are oriented consistently in $O$.

Las Vergnas defines his orientation activities as follows. Again we must fix a total order~$<$ of $E(G)$. We say that $e \in E(G)$ is \emph{cut (cycle) active} in the orientation $O$ if it is the min edge in $E(C)$ for $C$ some directed cut (cycle) of~$O$. We let $I(O)$ denote the set of cut active edges of $O$ and $L(O)$ the set of cycle active edges. In order to state the orientation analog of Theorem~\ref{thm:gordontraldi} we need one more piece of notation. For $O \in \mathcal{O}(G)$ and $\delta \in \{+,-\}$ set $O^{\delta} := \{e \in E(G)\colon e^{\delta} \in O\}$ so that~$E(G) = O^{+} \sqcup O^{-}$. To simplify notation we write $I^{+}(O) := I(O) \cap O^+$ and so on.

\begin{thm}[{Las Vergnas~\cite{gioanpreprint}~\cite{las1984tutte}~\cite{las2012tutte}}] \label{thm:lasvergnas}
For any graph $G$, any reference orientation~$O_{\mathrm{ref}}$, and any total order $<$ on $E(G)$ we have
\[T_G(x+w,y+z) = \sum_{O \in \mathcal{O}(G)} {x}^{\left\vert I^{+}(O) \right\vert} {w}^{\left\vert I^{-}(O) \right\vert} {y}^{\left\vert  L^{+}(O) \right\vert} {z}^{\left\vert L^{-}(O) \right\vert}.\]
\end{thm}

Again, Theorem~\ref{thm:lasvergnas} is a very general (and elegant) expansion for the Tutte polynomial with many consequences. For instance, a two-variable formula of Las Vergnas~\cite{las1984tutte} is obtained by taking $x:= w:= u/2$ and $y:=z := v/2$: 
\[T_G(u,v) = \sum_{O \in \mathcal{O}(G)} \left(  \frac{u}{2}\right)^{\left\vert I(O) \right\vert} \left(  \frac{v}{2}\right)^{\left\vert  L(O) \right\vert}.\] 
Also, taking $x:= w:= 1$, and $y:=z := 0$ we recover Stanley's celebrated result that the number of acyclic orientations of $G$ is~$T_G(2,0)$~\cite{stanley1973acyclic}. Dually, taking $x:= w:= 0$, and $y:=z := 1$ yields Las Vergnas's own result that the number of strongly connected orientations of $G$ is~$T_G(0,2)$~\cite{las1980convexity}. And more generally by taking $x,y,w,z \in \{0,1\}$ Theorem~\ref{thm:lasvergnas} gives combinatorial interpretations for all $T_G(a,b)$ with $0 \leq a,b \leq 2$ in terms of orientations. This $3 \times 3$ square of orientation classes has been explored in the works of Gioan~\cite{gioan2007enumerating}~\cite{gioan2008circuit} and Bernardi~\cite{bernardi2008tutte}. Gioan and Bernardi unified the results of many authors who, as mentioned in the introduction, showed that various classes of orientations are enumerated by the Tutte polynomial. We remark that Bernardi worked with a ribbon graph structure, i.e. a combinatorial map, rather than a total order on the edge and a reference orientation. For more on the history of the Tutte polynomial and its relation to orientation classes, see~\cite[\S1.1]{backman2015fourientations}

Of course, when the variables with hats are set equal to those without, Theorems~\ref{thm:gordontraldi} and~\ref{thm:lasvergnas} offer two different combinatorial expansions for the same polynomial. Proving bijectively that these expressions are indeed equal by matching terms on either side is one aim of the so-called ``active bijection'' of Gioan-Las Vergnas~\cite{gioan2009active}~\cite{gioanpreprint}~\cite{gioan2005activity}. Here we connect Theorems~\ref{thm:gordontraldi} and~\ref{thm:lasvergnas} in a different way: we offer an expansion of the Tutte polynomial in terms of fourientation activities that simultaneously generalizes both the Gordon-Traldi and Las Vergnas formulas. Fourientations can therefore be seen as a hybrid of subgraphs and orientations.

\section{Proving Theorem~\ref{thm:main}} \label{sec:main}

As discussed above, our purpose is to extend the work of Backman-Hopkins on fourientations \cite{backman2015fourientations} to provide a Tutte polynomial formula that specializes to both the Las Vergnas and the Gordon-Traldi formulas discussed in Section~\ref{sec:setup}. The key definition in~\cite{backman2015fourientations} is that of a potential cut (cycle) of a fourientation. Backman and Hopkins enumerated many classes of fourientations defined in terms of potential cuts and cycles.

\begin{definition} (\cite{backman2015fourientations})
Let~$O \in \mathcal{O}^{4}(G)$. We say a directed cut $Cu$ of $G$ is a \emph{potential cut of $O$} if for $\delta \in \{+,-\}$, ~$e^{\delta} \in \mathbb{E}(Cu) \Rightarrow e^{-\delta} \notin O$. Similarly, a directed cycle~$Cy$ of~$G$ is a \emph{potential cycle of $O$} if for $\delta \in \{+,-\}$, $e^{\delta} \in \mathbb{E}(Cy) \Rightarrow e^{\delta} \in O$. In other words, a potential cut of a fourientation is the same as a directed cut of an orientation except that some edges of the cut are allowed to be unoriented; and a potential cycle of a fourientation is the same as a directed cycle of an orientation except that some edges of the cycle are allowed to be bioriented.
\end{definition}

Theorem~\ref{thm:main} involves fourientation activities $I,L\colon \mathcal{O}^{4}(G) \to \mathcal{S}(G)$. As noted in the introduction, we use the notation $I^{u}(O) := I(O) \cap O^{u}$, $L^{+}(O) := L(O) \cap O^{+}$ etc. Some of these sets can be defined individually:

\begin{definition} \label{def:fouractivity}
Let $G$ be a graph, $O_{\mathrm{ref}}$ a reference orientation, and $<$ a total order on~$E(G)$. Also let $\sigma_u,\sigma_b\colon E(G) \to \{+,-\}$ be arbitrary sign labels of the edges of~$G$.\footnote{These sign labels $\sigma_u$ and $\sigma_b$ offer a certain degree of ``independence'' between the edges of $G$ in our main theorem, Theorem~\ref{thm:main}. As we will see in a moment, the sign labels show that the map~$\varphi$ of Theorem~\ref{thm:main} is highly non-unique. Note that even for a fixed choice of $\sigma_u$ and $\sigma_b$, the map~$\varphi$ is still not unique, as discussed later in Remark~\ref{rem:k4}.} Then for each~$O\in\mathcal{O}^{4}(G)$, we set
\begin{enumerate}
\item $I^{o}(O) := \{e\in O^{o} \colon \textrm{$e$ is the min edge in some potential cut of $O$}\}$;
\item $L^{o}(O) := \{e\in O^{o} \colon \textrm{$e$ is the min edge in some potential cycle of $O$}\}$;
\item $I^{u}(O) := \{e\in O^{u} \colon \textrm{$e$ is the min edge in some potential cut of $O \cup \{e^{\sigma_u(e)}\}$}\}$;
\item $L^{b}(O) := \{e\in O^{b} \colon \textrm{$e$ is the min edge in some potential cycle of $O \setminus \{e^{\sigma_b(e)}\}$}\}$.
\end{enumerate}
\end{definition}

From now on we fix $G$, $O_{\mathrm{ref}}$, $<$, $\sigma_u$, and~$\sigma_b$ as in Definition~\ref{def:fouractivity}. We will define the map~$\varphi$ of Theorem~\ref{thm:main} recursively using deletion-contraction. We will always delete and contract the maximum edge $e_{\mathrm{max}}$ of $G$ according to $<$, and we will always assume that the deletion $G-e_{\mathrm{max}}$ and the contraction $G/e_{\mathrm{max}}$ inherit from $G$ the several auxiliary features that appear in the definitions above (the edge order~$<$, the reference orientation~$O_{\mathrm{ref}}$, and the sign labels~$\sigma_u$ and~$\sigma_b$). For $O \in \mathcal{O}^{4}(G)$ and $e \in E(G)$, we use the notation $O-e$ to denote the fourientation of $G-e$ obtained from $O$ by restricting to $\mathbb{E}(G-e)$. The notation $O/e$ is used similarly (and in fact as sets $O - e = O/e$). One more piece of notation: for $e \in E(G)$ and $O \in \mathcal{O}^{4}(G)$ set $^{e}O := O \Delta \{e^+,e^-\}$ where $\Delta$ denotes set-theoretic symmetric difference. Explicitly, if $e$ is unoriented in $O$ then $e$ is bioriented in $^{e}O$ and vice-versa; and if~$e$ is oriented one way in $O$ then it is oriented oppositely in~$^{e}O$. In the generic case where~$e_{\mathrm{max}}$ is neither an isthmus nor a loop, we apply the following key lemma (which should be compared to~\cite[Proposition 1.1]{gioan2002correspondance}~\cite[Lemma 3.2]{las1984tutte}~\cite[Lemma 3.4]{las2012tutte}):

\begin{lemma}\label{lem:key}
Suppose the maximum edge $e := e_{\mathrm{max}}$ of $G$ is neither an isthmus nor a loop. Let $O \in \mathcal{O}^{4}(G)$. Then the sets defined in Definition~\ref{def:fouractivity} satisfy at least one of the following sets of conditions:
\begin{enumerate}
\item \begin{itemize}
\item $I^{o}(O) = I^{o}(O/e)$, $I^{u}(O) = I^{u}(O/e)$,
\item $L^{o}(O) = L^{o}(O/e)$, $L^{b}(O) = L^{b}(O/e)$,
\item $I^{o}(^{e}O) = I^{o}(O-e)$, $I^{u}(^{e}O) = I^{u}(O-e)$, 
\item $L^{o}(^{e}O) = L^{o}(O-e)$, $L^{b}(^{e}O) = L^{b}(O-e)$;\end{itemize}
\item \begin{itemize}
\item $I^{o}(^{e}O) = I^{o}(O/e)$, $I^{u}(^{e}O) = I^{u}(O/e)$,
\item $L^{o}(^{e}O) = L^{o}(O/e)$, $L^{b}(^{e}O) = L^{b}(O/e)$,
\item $I^{o}(O) = I^{o}(O-e)$, $I^{u}(O) = I^{u}(O-e)$,
\item $L^{o}(O) = L^{o}(O-e)$, $L^{b}(O) = L^{b}(O-e)$.\end{itemize}
\end{enumerate}
(Note that (2) merely asserts that (1) holds for $^{e}O$.) Moreover, if $e$ is bioriented in $O$ then certainly (1) holds and if $e$ is unoriented in $O$ then certainly (2) holds.
\end{lemma}

In the case where $e_{\mathrm{max}}$ is either an isthmus or a loop we apply one of the following two simpler lemmas instead. (The proofs are straightforward, and are left to the reader.)

\begin{lemma} \label{lem:isthmus}
Suppose the maximum edge $e := e_{\mathrm{max}}$ of $G$ is an isthmus. For any fourientation~\mbox{$O \in \mathcal{O}^{4}(G)$}, the sets defined in Definition~\ref{def:fouractivity} satisfy the following:
\begin{itemize}
\item if $e\in O^{o}$ then $I^{o}(O)=\{e\}\cup I^{o}(O-e)$ and $I^{u}(O) = I^{u}(O-e)$;
\item if $e\in O^{u}$ then $I^{o}(O) = I^{o}(O-e)$ and $I^{u}(O)=\{e\}\cup I^{u}(O-e)$;
\item $L^{o}(O) = L^{o}(O-e)$;
\item $L^{b}(O) = L^{b}(O-e)$.
\end{itemize}
\end{lemma}

\begin{lemma} \label{lem:loop}
Suppose the maximum edge $e := e_{\mathrm{max}}$ of $G$ is a loop. For any fourientation~\mbox{$O \in \mathcal{O}^{4}(G)$}, the sets defined in Definition~\ref{def:fouractivity} satisfy the following:
\begin{itemize}
\item $I^{o}(O)= I^{o}(O/e)$;
\item $I^{u}(O)  = I^{u}(O/e)$;
\item if $e\in O^{o}$ then $L^{o}(O)=\{e\}\cup L^{o}(O-e)$ and $L^{b}(O) = L^{b}(O-e)$;
\item if $e\in O^{b}$ then $L^{o}(O)= L^{o}(O-e)$ and $L^{b}(O) = \{e\}\cup L^{b}(O-e)$.
\end{itemize}
\end{lemma}

Given the above lemmas it is straightforward to recursively define a surjection $\varphi$ of the type mentioned in Theorem~\ref{thm:main}. We can also specify several additional properties, which are needed to justify the formula of Theorem~\ref{thm:main}.

\begin{thm} \label{thm:fouracttutteeval}
There exists a $2^{\left\vert E(G)\right\vert}$-to-one surjection $\varphi\colon \mathcal{O}^{4}(G) \to \mathcal{S}(G)$ such that
\begin{itemize}
\item the map $O \mapsto (\varphi(O),O^{o})$ is a bijection between $\mathcal{O}^{4}(G)$ and $\mathcal{S}(G) \times \mathcal{S}(G)$;
\item we have $O^{u} \cap \varphi(O) = \varnothing$ and $O^{b} \subseteq \varphi(O)$ for all $O \in \mathcal{O}^{4}(G)$;
\item setting $I(O) := \hat{I}(\varphi(O))$ and $L(O) := \hat{L}(\varphi(O))$ (where $\hat{I},\hat{L}$ are the Gordon-Traldi generalized activities with respect to the same edge order $<$), the resulting maps~$I,L\colon \mathcal{O}^{4}(G) \to \mathcal{S}(G)$ are compatible with Definition~\ref{def:fouractivity} in the sense that for every~$O\in \mathcal{O}^{4}(G)$, $I(O) \cap O^{o}=I^{o}(O)$, $I(O) \cap O^{u}=I^{u}(O)$, $L(O) \cap O^{o}=L^{o}(O)$ and~$L(O) \cap O^{b}=L^{b}(O)$;
\item the mappings $I,L$ have $I^{+}(O) \subseteq \varphi(O)$, $I^{-}(O) \cap \varphi(O) = \varnothing$, $L^{-}(O) \subseteq \varphi(O)$, and $L^{+}(O) \cap \varphi(O) = \varnothing$ for all $O \in \mathcal{O}^{4}(G)$;
\item if $O^{o}=\varnothing$ then $I(O)=\hat{I}(O^b)$ and $L(O)=\hat{L}(O^b)$.
\end{itemize}
\end{thm}

\begin{proof}[Proof of Theorem~\ref{thm:fouracttutteeval} from Lemmas~\ref{lem:key},~\ref{lem:isthmus}, and~\ref{lem:loop}]
If $E(G)=\varnothing$ then there is nothing to prove. We proceed to define $\varphi$ by induction on~$\left\vert E(G)\right\vert >0$. Let~$e := e_{\mathrm{max}}$ be the maximum edge of $G$ and let~$\varphi_{c}:\mathcal{O}^{4}(G/e)\rightarrow\mathcal{S}(G/e)$ and~$\varphi_{d}:\mathcal{O}^{4}(G-e)\rightarrow\mathcal{S}(G-e)$ be surjections provided by the inductive hypothesis. We use $\varphi_{c}$ and $\varphi_{d}$ to define the surjection~$\varphi:\mathcal{O}^{4}(G)\rightarrow\mathcal{S}(G)$ as follows:

\begin{enumerate}
\item If $e$ is an isthmus and $O\in\mathcal{O}^{4}(G)$ then define $\varphi(O):=\{e\}\cup\varphi_{c}(O/e)$ if~$e^{+}\in O$, and $\varphi(O):=\varphi_{d}(O-e)$ if~$e^{+}\notin O$.

\item If $e$ is a loop and $O\in\mathcal{O}^{4}(G)$, then define $\varphi(O):=\{e\}\cup\varphi_{c}(O/e)$ if~$e^{-}\in O$, and $\varphi(O):=\varphi_{d}(O-e)$ if $e^{-}\notin O$.

\item If $e$ is neither an isthmus nor a loop and condition~(2) of Lemma~\ref{lem:key} fails for~$O\in\mathcal{O}^{4}(G)$, then set~$\varphi(O):=\{e\}\cup\varphi_{c}(O/e)$ and $\varphi(^{e}O):=\varphi_{d}(O-e)$. Note that Lemma~\ref{lem:key} tells us that condition~(2) of Lemma~\ref{lem:key} holds for~$^{e}O$, so there is no conflict if it should happen that $^{e}O$ is considered rather than $O$.

\item If $e$ is neither an isthmus nor a loop and condition~(1) of Lemma~\ref{lem:key} fails for~$O \in \mathcal{O}^{4}(G)$, set~$\varphi(O):=\varphi_{d}(O-e)$ and~$\varphi(^{e}O):=\{e\}\cup\varphi_{c}(O/e)$. Lemma~\ref{lem:key}
tells us that $^{e}O$ satisfies condition~(1) of Lemma~\ref{lem:key}, so there is no conflict if it should happen that~$^{e}O$ is considered rather than~$O$.

\item If $O\in\mathcal{O}^{4}(G)$ and conditions~(1) and~(2) of Lemma~\ref{lem:key} both hold for $O$, then conditions~(1) and~(2) of Lemma~\ref{lem:key} both hold for $^{e}O$ too. If $e$ is bioriented, set~$\varphi(O):=\{e\}\cup\varphi_{c}(O/e)$ and $\varphi(^{e}O):=\varphi_{d}(O-e)$. If $e$ is unoriented, set~$\varphi(^{e}O):=\{e\}\cup\varphi_{c}(O/e)$ and~$\varphi(O):=\varphi_{d}(O-e)$. If $e$ is oriented and $e^{+}\in O$, set~$\varphi(O):=\{e\}\cup\varphi_{c}(O/e)$ and~$\varphi(^{e}O):=\varphi_{d}(O-e)$. And finally, if $e$ is oriented and $e^{-}\in O$, set~$\varphi(^{e}O):=\{e\}\cup\varphi_{c}(O/e)$ and~$\varphi(O):=\varphi_{d}(O-e)$. Toggling the status of $e$ in $\varphi(O)$ guarantees that there is no conflict if it should happen that~$^{e}O$ is considered rather than~$O$.

\end{enumerate}

Lemma~\ref{lem:key} asserts that if $e$ is bioriented then condition~(1) of that lemma certainly holds and if $e$ is unoriented then condition~(2) certainly holds, so~for all $O \in \mathcal{O}^{4}(G)$ we have~$O^b\subseteq \varphi(O)$ and~$O^u \cap \varphi(O) = \varnothing$. Also, the definition of~$\varphi$ implies that~$e\in\varphi(O)$ if and only if $e\notin\varphi(^{e}O)$ which can be used to see inductively that $O \mapsto (\varphi(O),O^o)$ is a bijection from $\mathcal{O}^{4}(O)$ to~$\mathcal{S}(G) \times \mathcal{S}(G)$. That~$I(O) := \hat{I}(\varphi(O))$ and~$L(O) := \hat{L}(\varphi(O))$ are compatible with Definition~\ref{def:fouractivity} is also verified inductively, using Lemmas~\ref{lem:key}, \ref{lem:isthmus}, and~\ref{lem:loop} together with the following simple observations about~$\hat{I}$ and~$\hat{L}$ (where~$\hat{I}_{\Gamma}$ and~$\hat{L}_{\Gamma}$ denote these generalized activities for the graph~$\Gamma$):
\begin{itemize}
\item if $e$ is neither an isthmus nor a loop, then
\begin{align*}
e & \in S\Rightarrow \hat{I}_{G}(S)=\hat{I}_{G/e}(S\setminus\{e\}) \textrm{ and } \hat{L}_{G}(S)=\hat{L}_{G/e}(S\setminus\{e\}),\\
e & \not \in S\Rightarrow \hat{I}_{G}(S)=\hat{I}_{G-e}(S) \textrm{ and } \hat{L}_{G}(S)=\hat{L}_{G-e}(S);
\end{align*}
\item is $e$ is an isthmus, then
\begin{align*}
e & \in S\Rightarrow \hat{I}_{G}(S)=\hat{I}_{G/e}(S\setminus\{e\}) \textrm{ and } \hat{L}_{G}(S)=\{e\}\cup \hat{L}_{G/e}(S\setminus\{e\}),\\
e & \not \in S\Rightarrow \hat{I}_{G}(S)=\hat{I}_{G-e}(S) \textrm{ and } \hat{L}_{G}(S)=\{e\}\cup \hat{L}_{G-e}(S);
\end{align*}
\item is $e$ is a loop, then
\begin{align*}
e & \in S\Rightarrow \hat{I}_{G}(S)=\{e\}\cup \hat{I}_{G/e}(S\setminus
\{e\}) \textrm{ and } \hat{L}_{G}(S)=\hat{L}_{G/e}(S\setminus\{e\}),\\
e  & \not \in S\Rightarrow \hat{I}_{G}(S)=\{e\}\cup \hat{I}_{G-e}(S) \textrm{ and }\hat{L}_{G}(S)=\hat{L}_{G-e}(S).
\end{align*}
\end{itemize}
That~$I(O) := \hat{I}(\varphi(O))$ and~$L(O) := \hat{L}(\varphi(O))$ satisfy~$I(O) = \hat{I}(O^b)$ and~$L(O) = \hat{L}(O^b)$ when $O^o = \varnothing$ is clear because in this case we have $\varphi(O) = O^b$ by construction. \end{proof}

To complete the proof of Theorem~\ref{thm:fouracttutteeval} we must prove the key lemma, Lemma~\ref{lem:key}. The proof of Lemma~\ref{lem:key} is long and rather technical, so it is included in Appendix~\ref{sec:proofkeylem}. 

Theorem~\ref{thm:main} follows readily from  Theorem~\ref{thm:fouracttutteeval}:

\begin{proof}[Proof of Theorem~\ref{thm:main} from Theorem~\ref{thm:fouracttutteeval}]
By Theorem~\ref{thm:gordontraldi}, we have the following.
{\footnotesize
\begin{gather*}
(k_1+m)^{n-\kappa}(k_2+l)^{g}T_G\left(\frac{k_1x+k_2w+m\hat{x}+l\hat{w}}{k_1+m},\frac{k_2y+k_1z+l\hat{y}+m\hat{z}}{k_1+l}\right) = \\
 \\
(k_1+m)^{n-\kappa}(k_2+l)^{g}\sum_{S \in \mathcal{S}(G)} \left\{
\begin{array}[c]{c}
\left(\frac{k_1x+m\hat{x}}{k_1+m}\right)^{\left\vert \hat{I}(S)\cap S \right\vert} \left(\frac{k_2w+l\hat{w}}{k_1+m}\right)^{\left\vert \hat{I}(S)\setminus S\right\vert} \\ 
\cdot \left(\frac{k_2y+l\hat{y}}{k_2+l}\right)^{\left\vert \hat{L}(S) \setminus S\right\vert} \left(\frac{k_1z+m\hat{z}}{k_2+l}\right)^{\left\vert \hat{L}(S)\cap S \right\vert}
\end{array} \right. = \\
 \\
\sum_{S \in \mathcal{S}(G)} \left\{ 
\begin{array}[c]{c}
(k_1+m)^{n-\kappa-\left\vert \hat{I}(S) \right\vert}(k_2+l)^{g-\left\vert \hat{L}(S) \right\vert} \\ 
\cdot (k_1x+m\hat{x})^{\left\vert \hat{I}(S)\cap S \right\vert} (k_2w+l\hat{w})^{\left\vert \hat{I}(S)\setminus S \right\vert}(k_2y+l\hat{y})^{\left\vert \hat{L}(S) \setminus S \right\vert} (k_1z+m\hat{z})^{\left\vert \hat{L}(S)\cap S\right\vert}
\end{array} \right. = \\
 \\
\sum_{S \in \mathcal{S}(G)} \left\{ 
\begin{array}[c]{c}
(k_1+m)^{\left\vert S \setminus (\hat{I}(S) \cup \hat{L}(S)) \right\vert}(k_2+l)^{\left\vert E(G) \setminus (S \cup \hat{I}(S) \cup \hat{L}(S)) \right\vert} \\ 
\cdot (k_1x+m\hat{x})^{\left\vert \hat{I}(S)\cap S \right\vert} (k_2w+l\hat{w})^{\left\vert \hat{I}(S)\setminus S \right\vert}(k_2y+l\hat{y})^{\left\vert \hat{L}(S) \setminus S \right\vert} (k_1z+m\hat{z})^{\left\vert \hat{L}(S)\cap S \right\vert}
\end{array} \right. \\
\end{gather*}}
Above we have used the fact that for all $S \in \mathcal{S}(G)$ we have the equalities
\begin{align*}
n-\kappa-\left\vert \hat{I}(S)\right\vert &= \left\vert S\setminus(\hat{I}(S)\cup \hat{L}(S))\right\vert \text{ and} \\
g-\left\vert \hat{L}(S)\right\vert &= \left\vert E(G)\setminus (S\cup \hat{I}(S) \cup \hat{L}(S))\right\vert,
\end{align*}
which follow from Lemma~\ref{lem:rankact}. Now we expand all the products to transform the above expression into
\[\sum_{S \in \mathcal{S}(G)} \sum_{S' \in \mathcal{S}(G)} \left\{ 
\begin{array}[c]{c}
k_1^{\left\vert S \cap S'\right\vert} k_2^{\left\vert S' \setminus S\right\vert} l^{\left \vert E(G) \setminus (S' \cup S) \right \vert} m^{\left \vert S \setminus S' \right \vert} \\
\cdot x^{\left \vert \hat{I}(S)\cap (S\cap S') \right \vert}w^{\left \vert \hat{I}(S)\cap (S' \setminus S) \right \vert} \hat{x}^{\left \vert \hat{I}(S) \cap (S \setminus S')) \right \vert} \hat{w}^{\left \vert \hat{I}(S) \cap (E(G) \setminus (S \cup S')) \right \vert} \\
\cdot y^{\left \vert \hat{L}(S)\cap (S' \setminus S) \right \vert} z^{\left \vert \hat{L}(S)\cap (S \cap S') \right \vert} \hat{y}^{\left \vert \hat{L}(S) \cap (E(G) \setminus (S' \cup S)) \right \vert} \hat{z}^{\left \vert \hat{L}(S) \cap (S \setminus S') \right \vert}
\end{array} \right.\]
Finally, we use Theorem~\ref{thm:fouracttutteeval} to rewrite exponents as functions of fourientations rather than subgraphs. The resulting formula (where $S$ becomes $\varphi(O)$ and $S'$ becomes $O^o$) is 
{\footnotesize
\begin{gather*}
\sum_{S \in \mathcal{S}(G)} \sum_{O \in \varphi^{-1}(S)} \left\{
\begin{array}[c]{c}
k_1^{\left\vert O^o \cap \varphi(O)) \right\vert}k_2^{\left\vert O^o \setminus \varphi(O) \right\vert}l^{\left\vert O^{u} \right\vert}m^{\left\vert O^{b} \right\vert} \\ 
\cdot x^{\left\vert I^{+}(O) \right\vert} w^{\left\vert I^{-}(O) \right\vert} {\hat{x}}^{\left\vert I^{b}(O) \right\vert} {\hat{w}}^{\left\vert I^{u}(O) \right\vert} y^{\left\vert L^{+}(O) \right\vert} z^{\left\vert L^{-}(O) \right\vert}{\hat{y}}^{\left\vert L^{u}(O)\right\vert} {\hat{z}}^{\left\vert L^{b}(O)\right\vert}
\end{array} \right. = \\
 \\
\sum_{O \in \mathcal{O}^{4}(G)} \hspace{-0.3cm} k_1^{\left\vert O^o \cap \varphi(O) \right\vert}k_2^{\left\vert O^o \setminus \varphi(O) \right\vert}l^{\left\vert O^{u}\right\vert}m^{\left\vert O^{b}\right\vert}x^{\left\vert I^{+}(O)\right\vert} w^{\left\vert I^{-}(O)\right\vert} {\hat{x}}^{\left\vert I^{b}(O)\right\vert} {\hat{w}}^{\left\vert I^{u}(O)\right\vert} y^{\left\vert L^{+}( O)\right\vert} z^{\left\vert L^{-}(O)\right\vert}{\hat{y}}^{\left\vert L^{u}(O)\right\vert} {\hat{z}}^{\left\vert L^{b}(O)\right\vert}.
\end{gather*}} 
\end{proof}

Observe that the compatibility provisions of Theorem~\ref{thm:fouracttutteeval} imply that Theorem~\ref{thm:main} recovers Theorem~\ref{thm:gordontraldi} by taking~$(k_1,k_2,l,m) := (0,0,1,1)$, and Theorem~\ref{thm:lasvergnas} by taking~$(k_1,k_2,l,m) := (1,1,0,0)$. Moreover, Theorem~\ref{thm:main} offers a nontrivial interpolation of the Gordon-Traldi and Las~Vergnas formulas. Taking $k_1 := k_2 := k$, $\hat{x} := \hat{y} := 1$, and~$x,y,w,z,\hat{w},\hat{z}\in\{0,1\}$ we recover all the enumerations of (in fact, generating functions for) ``min-edge classes'' of fourientations obtained by Backman-Hopkins~\cite{backman2015fourientations}. (These min-edge classes form an intersection lattice of~$64$ sets of fourientations defined in terms of restrictions on potential cuts and cycles.) We remark that these min-edge classes of fourientations themselves already unified earlier results of Gessel-Sagan~\cite{gessel1996tutte}, Hopkins-Perkinson~\cite{hopkins2016bigraphical}, and Backman~\cite{backman2014partial}.  For instance, by taking~$(k_1,k_2,l,m) := (1,1,1,0)$, $y := z := 0$ and~$x := w := \hat{x} := \hat{w} := \hat{y} := \hat{z} := 1$, the formula says that the number of acyclic partial orientations of $G$ is $2^{g}T_G(3,1/2)$, a result originally obtained by Gessel-Sagan~\cite{gessel1996tutte}.\footnote{See Backman-Hopkins~\cite{backman2015fourientations} for precise definitions of \emph{partial orientation}, \emph{acyclic partial orientation}, \emph{$q$-connected partial orientation}, \emph{min-edge class}, et cetera.} Furthermore, Backman-Hopkins~\cite[\S4.4]{backman2015fourientations} showed that, when $G$ is connected, for any choice of root $q \in V(G)$ there is a choice of reference orientation~$O_{\mathrm{ref}}$, total edge order $<$, and edge labels $\sigma_u,\sigma_b$ so that the $q$-connected fourientations of~$G$ are precisely those fourientations~$O \in \mathcal{O}^{4}(G)$ which satisfy~$I^{u}(O) = I^{-}(O) = \varnothing$ (where~$I$ is as in Theorem~\ref{thm:fouracttutteeval}). Thus Theorem~\ref{thm:fouracttutteeval} also implies that (for connected graphs~$G$) the number of~$q$-connected fourientations is $2^{\left\vert E(G) \right\vert}T_G(1,2)$ and the number of acyclic, $q$-connected partial orientations is~$2^{g}T_G(1,1/2)$, two other evaluations first obtained by Gessel-Sagan~\cite{gessel1996tutte}.  In sum, we see that Theorem~\ref{thm:main} encompasses many results that have appeared in the literature.

Note that~$\varphi$ restricts to an activity-preserving bijection $\mathcal{O}(G) \stackrel{\sim}\longrightarrow \mathcal{S}(G)$, and any such bijection provides an equivalence between Theorems~\ref{thm:gordontraldi} and~\ref{thm:lasvergnas}.  As discussed in~\cite{gioanpreprint}, in general there are many such activity-preserving bijections, with various properties. In particular, the active bijection of Gioan-Las Vergnas~\cite{gioan2009active}~\cite{gioanpreprint}~\cite{gioan2005activity} has some attractive properties not shared by $\varphi$. One such property is that if $O \in \mathcal{O}(G)$ and~$O^{\mathrm{rev}}$ is the orientation obtained from~$O$ by reversing all edge orientations, then the images of~$O$ and~$O^{\mathrm{rev}}$ under the active bijection belong to the same Crapo interval. As discussed in Remark~\ref{rem:k4} below, the images of~$O$ and~$O^{\mathrm{rev}}$ under $\varphi$ need not belong to the same Crapo interval.

For convenience we have included a table, Figure~\ref{fig:parameters}, which describes each parameter in Theorem~\ref{thm:main} and either refers to the direct definition of the quantity this parameter records or refers to a remark where we discuss why such a direct interpretation is difficult to find.

\begin{figure} 
\renewcommand*{\arraystretch}{1.4}
\begin{tabular}{c | c | c}
Parameter in Theorem~\ref{thm:main} & Quantity it records & Description of this quantity \\ \hline
$k_1$ & $|O^o \cap \varphi(O)|$ & \parbox{2.3in}{\vspace{.5\baselineskip} $|O^o|$ is the number of oriented edges in $O$; no direct interpretation of $|O^o \cap \varphi(O)|$ alone, see Remark~\ref{rem:k1andk2} \vspace{.5\baselineskip} } \\ \hline
$k_2$ & $|O^o \setminus \varphi(O)|$ & \parbox{2.3in}{\vspace{.5\baselineskip} $|O^o|$ is the number of oriented edges in $O$; no direct interpretation of $|O^o \setminus \varphi(O)|$ alone, see Remark~\ref{rem:k1andk2} \vspace{.5\baselineskip}} \\ \hline
$l$ & $|O^u|$ & \parbox{2.3in}{\vspace{.5\baselineskip} number of unoriented edges in $O$ \vspace{.5\baselineskip}} \\ \hline
$m$ & $|O^b|$ & \parbox{2.3in}{\vspace{.5\baselineskip} number of bioriented edges in $O$ \vspace{.5\baselineskip}} \\ \hline
$x$ & $|I^+(O)|$ & \parbox{2.3in}{\vspace{.5\baselineskip} number of cut active edges in $O$ oriented in agreement with $O_{\mathrm{ref}}$; see Definition~\ref{def:fouractivity} \vspace{.5\baselineskip}} \\ \hline
$w$ & $|I^-(O)|$ & \parbox{2.3in}{\vspace{.5\baselineskip} number of cut active edges in $O$ oriented in diagreement with $O_{\mathrm{ref}}$; see Definition~\ref{def:fouractivity} \vspace{.5\baselineskip}} \\ \hline
$\hat{x}$ & $|I^b(O)|$ & \parbox{2.3in}{\vspace{.5\baselineskip} number of cut active bioriented edges in $O$; no direct intepretation, see Remark~\ref{rem:mysteryact} \vspace{.5\baselineskip}} \\ \hline
$\hat{w}$ & $|I^u(O)|$ & \parbox{2.3in}{\vspace{.5\baselineskip} number of cut active unoriented edges in $O$; see Definition~\ref{def:fouractivity} \vspace{.5\baselineskip}} \\ \hline
$y$ & $|L^+(O)|$ & \parbox{2.3in}{\vspace{.5\baselineskip} number of cycle active edges in $O$ oriented in agreement with $O_{\mathrm{ref}}$; see Definition~\ref{def:fouractivity} \vspace{.5\baselineskip}} \\ \hline
$z$ & $|L^-(O)|$ & \parbox{2.3in}{\vspace{.5\baselineskip} number of cycle active edges in $O$ oriented in diagreement with $O_{\mathrm{ref}}$; see Definition~\ref{def:fouractivity} \vspace{.5\baselineskip}} \\ \hline
$\hat{y}$ & $|L^u(O)|$ & \parbox{2.3in}{\vspace{.5\baselineskip} number of cycle active unoriented edges in $O$; no direct intepretation, see Remark~\ref{rem:mysteryact} \vspace{.5\baselineskip}} \\ \hline
$\hat{z}$ & $|L^b(O)|$ & \parbox{2.3in}{\vspace{.5\baselineskip} number of cycle active bioriented edges in $O$; see Definition~\ref{def:fouractivity} \vspace{.5\baselineskip}} \\ \hline
\end{tabular} 
\caption{Table of parameters in Theorem~\ref{thm:main}.} \label{fig:parameters}
\end{figure}

\section{Examples and discussion of~\texorpdfstring{$\varphi$}{phi}} \label{sec:ex}

In this section we will give an example of what the surjection $\varphi$ is in one of the simplest nontrivial cases: the triangle graph. Then we make some remarks about how the map $\varphi$ can be rather subtle.

\begin{example}\label{ex:main}
Let $G$ and $O_{\mathrm{ref}}$ be as below:
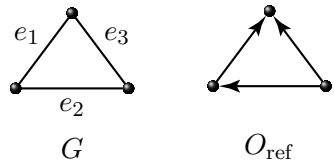
\begin{center}
\begin{tikzpicture}[scale=1]
	\SetFancyGraph
	\Vertex[NoLabel,x=0.75,y=0]{v_3}
	\Vertex[NoLabel,x=0,y=1]{v_2}
	\Vertex[NoLabel,x=-0.75,y=0]{v_1}
	\Edges[style={thick}](v_1,v_2)
	\Edges[style={thick}](v_1,v_3)
	\Edges[style={thick}](v_2,v_3)
	\node at (0.6,0.7) {$e_3$};
	\node at (0,-0.25) {$e_2$};
	\node at (-0.6,0.7) {$e_1$};
	\node at (0,-0.8){$G$};
\end{tikzpicture} \qquad \begin{tikzpicture}[scale=1]
	\SetFancyGraph
	\SetFancyGraph
	\Vertex[NoLabel,x=0.75,y=0]{v_3}
	\Vertex[NoLabel,x=0,y=1]{v_2}
	\Vertex[NoLabel,x=-0.75,y=0]{v_1}
	\Edges[style={thick,->,>=mytip}](v_1,v_2)
	\Edges[style={thick,->,>=mytip}](v_3,v_1)
	\Edges[style={thick,->,>=mytip}](v_3,v_2)
	\node at (0,-0.8){$O_{\mathrm{ref}}$};
\end{tikzpicture}
\end{center}
We take the total edge order $e_1 < e_2 < e_3$. We set $\sigma_u(e) := -$ and $\sigma_b(e) := +$ for all edges~$e \in E(G)$. Then Figures~\ref{fig:expart1} and~\ref{fig:expart2} show the fibers~$\varphi^{-1}(S)$ for all~$S \in \mathcal{S}(G)$. In our depiction of a subgraph~$S \in \mathcal{S}(G)$, edges that belong to~$S$ are solid and edges that do not belong to~$S$ are dashed. And in our depiction of a fourientation~$O \in \mathcal{O}^{4}(G)$ oriented edges $e^{\delta}=(u,v)$ are depicted by an arrow from~$u$ to~$v$, while unoriented edges are solid lines with no arrows, and bioriented edges are drawn with two arrows, one in each direction. Note that in this example there is a unique $\varphi$ satisfying the constraints of Theorem~\ref{thm:fouracttutteeval}. \end{example}

\afterpage{
\begin{figure}
\begin{center}
\begin{tabular}{ >{\centering\arraybackslash}m{3cm} | >{\centering\arraybackslash}m{2cm} | >{\centering\arraybackslash}m{2cm} | >{\centering\arraybackslash}m{3cm}}
$S \in \mathcal{S}(G)$ & $I(S)$ & $L(S)$ & $\varphi^{-1}(S) $\\
\hline
\begin{tikzpicture}
	\SetFancyGraph
	\Vertex[NoLabel,x=-0.75,y=0]{v_1}
	\Vertex[NoLabel,x=0,y=1]{v_2}
	\Vertex[NoLabel,x=0.75,y=0]{v_3}
	\Edges[style={dashed}](v_1,v_2)
	\Edges[style={dashed}](v_1,v_3)
	\Edges[style={dashed}](v_2,v_3)
\end{tikzpicture}& $\{e_1,e_2\}$ & $\varnothing$ & \begin{tikzpicture}
	\node at (0,0.25) {};
	\Triangle[1,0,-1,-1,-1]
	\Triangle[0,-1,0,-1,-1]
	\Triangle[1,-1,-1,0,-1]
	\Triangle[2,-1,-1,-1,0]
	\Triangle[0,-2,0,0,-1]
	\Triangle[1,-2,0,-1,0]
	\Triangle[2,-2,-1,0,0]
	\Triangle[1,-3,0,0,0]
\end{tikzpicture} \\ \hline
\begin{tikzpicture}
	\SetFancyGraph
	\Vertex[NoLabel,x=-0.75,y=0]{v_1}
	\Vertex[NoLabel,x=0,y=1]{v_2}
	\Vertex[NoLabel,x=0.75,y=0]{v_3}
	\Edges[style={dashed}](v_1,v_2)
	\Edges[style={ultra thick}](v_1,v_3)
	\Edges[style={dashed}](v_2,v_3)
\end{tikzpicture}& $\{e_1,e_2\}$ & $\varnothing$ & \begin{tikzpicture}
	\node at (0,0.25) {};
	\Triangle[1,0,-1,1,1]
	\Triangle[0,-1,0,1,1]
	\Triangle[1,-1,-1,2,-1]
	\Triangle[2,-1,-1,1,0]
	\Triangle[0,-2,0,2,-1]
	\Triangle[1,-2,0,1,0]
	\Triangle[2,-2,-1,2,0]
	\Triangle[1,-3,0,2,0]
\end{tikzpicture} \\ \hline
\begin{tikzpicture}
	\SetFancyGraph
	\Vertex[NoLabel,x=-0.75,y=0]{v_1}
	\Vertex[NoLabel,x=0,y=1]{v_2}
	\Vertex[NoLabel,x=0.75,y=0]{v_3}
	\Edges[style={ultra thick}](v_1,v_2)
	\Edges[style={dashed}](v_1,v_3)
	\Edges[style={dashed}](v_2,v_3)
\end{tikzpicture}& $\{e_1,e_2\}$ & $\varnothing$ & \begin{tikzpicture}
	\node at (0,0.25) {};
	\Triangle[1,0,1,-1,-1]
	\Triangle[0,-1,2,-1,-1]
	\Triangle[1,-1,1,0,-1]
	\Triangle[2,-1,1,-1,0]
	\Triangle[0,-2,2,0,-1]
	\Triangle[1,-2,2,-1,0]
	\Triangle[2,-2,1,0,0]
	\Triangle[1,-3,2,0,0]
\end{tikzpicture} \\ \hline
\begin{tikzpicture}
	\SetFancyGraph
	\Vertex[NoLabel,x=-0.75,y=0]{v_1}
	\Vertex[NoLabel,x=0,y=1]{v_2}
	\Vertex[NoLabel,x=0.75,y=0]{v_3}
	\Edges[style={ultra thick}](v_1,v_2)
	\Edges[style={ultra thick}](v_1,v_3)
	\Edges[style={dashed}](v_2,v_3)
\end{tikzpicture}& $\{e_1,e_2\}$ & $\varnothing$ & \begin{tikzpicture}
	\node at (0,0.25) {};
	\Triangle[1,0,1,1,1]
	\Triangle[0,-1,2,1,1]
	\Triangle[1,-1,1,2,1]
	\Triangle[2,-1,1,1,0]
	\Triangle[0,-2,2,2,-1]
	\Triangle[1,-2,2,1,0]
	\Triangle[2,-2,1,2,0]
	\Triangle[1,-3,2,2,0]
\end{tikzpicture} \\ \hline
\end{tabular}
\end{center}
\caption{The first half of the $S \in \mathcal{S}(G)$ for Example~\ref{ex:main}, together with their activities and fibers $\varphi^{-1}(S)$.} \label{fig:expart1}
\end{figure}  \clearpage}

\afterpage{
\begin{figure}
\begin{center}
\begin{tabular}{ >{\centering\arraybackslash}m{3cm} | >{\centering\arraybackslash}m{2cm} | >{\centering\arraybackslash}m{2cm} | >{\centering\arraybackslash}m{3cm}}
$S \in \mathcal{S}(G)$ & $I(S)$ & $L(S)$ & $\varphi^{-1}(S) $\\
\hline
 \begin{tikzpicture}
	\SetFancyGraph
	\Vertex[NoLabel,x=-0.75,y=0]{v_1}
	\Vertex[NoLabel,x=0,y=1]{v_2}
	\Vertex[NoLabel,x=0.75,y=0]{v_3}
	\Edges[style={dashed}](v_1,v_2)
	\Edges[style={dashed}](v_1,v_3)
	\Edges[style={ultra thick}](v_2,v_3)
\end{tikzpicture}& $\{e_1\}$ & $\varnothing$ & \begin{tikzpicture}
	\node at (0,0.25) {};
	\Triangle[1,0,-1,1,-1]
	\Triangle[0,-1,0,1,-1]
	\Triangle[1,-1,-1,0,1]
	\Triangle[2,-1,-1,1,2]
	\Triangle[0,-2,0,0,1]
	\Triangle[1,-2,0,1,2]
	\Triangle[2,-2,-1,0,2]
	\Triangle[1,-3,0,0,2]
\end{tikzpicture} \\ \hline
\begin{tikzpicture}
	\SetFancyGraph
	\Vertex[NoLabel,x=-0.75,y=0]{v_1}
	\Vertex[NoLabel,x=0,y=1]{v_2}
	\Vertex[NoLabel,x=0.75,y=0]{v_3}
	\Edges[style={ultra thick}](v_1,v_2)
	\Edges[style={dashed}](v_1,v_3)
	\Edges[style={ultra thick}](v_2,v_3)
\end{tikzpicture}& $\{e_1\}$ & $\varnothing$ & \begin{tikzpicture}
	\node at (0,0.25) {};
	\Triangle[1,0,1,-1,1]
	\Triangle[0,-1,2,1,-1]
	\Triangle[1,-1,1,0,1]
	\Triangle[2,-1,1,-1,2]
	\Triangle[0,-2,2,0,1]
	\Triangle[1,-2,2,1,2]
	\Triangle[2,-2,1,0,2]
	\Triangle[1,-3,2,0,2]
\end{tikzpicture} \\ \hline
\begin{tikzpicture}
	\SetFancyGraph
	\Vertex[NoLabel,x=-0.75,y=0]{v_1}
	\Vertex[NoLabel,x=0,y=1]{v_2}
	\Vertex[NoLabel,x=0.75,y=0]{v_3}
	\Edges[style={dashed}](v_1,v_2)
	\Edges[style={ultra thick}](v_1,v_3)
	\Edges[style={ultra thick}](v_2,v_3)
\end{tikzpicture}& $\varnothing$ & $\{e_1\}$ & \begin{tikzpicture}
	\node at (0,0.25) {};
	\Triangle[1,0,1,1,-1]
	\Triangle[0,-1,0,-1,1]
	\Triangle[1,-1,1,2,-1]
	\Triangle[2,-1,1,1,2]
	\Triangle[0,-2,0,2,1]
	\Triangle[1,-2,0,-1,2]
	\Triangle[2,-2,1,2,2]
	\Triangle[1,-3,0,2,2]
\end{tikzpicture} \\ \hline
\begin{tikzpicture}
	\SetFancyGraph
	\Vertex[NoLabel,x=-0.75,y=0]{v_1}
	\Vertex[NoLabel,x=0,y=1]{v_2}
	\Vertex[NoLabel,x=0.75,y=0]{v_3}
	\Edges[style={ultra thick}](v_1,v_2)
	\Edges[style={ultra thick}](v_1,v_3)
	\Edges[style={ultra thick}](v_2,v_3)
\end{tikzpicture}& $\varnothing$ & $\{e_1\}$ & \begin{tikzpicture}
	\node at (0,0.25) {};
	\Triangle[1,0,-1,-1,1]
	\Triangle[0,-1,2,-1,1]
	\Triangle[1,-1,-1,2,1]
	\Triangle[2,-1,-1,-1,2]
	\Triangle[0,-2,2,2,1]
	\Triangle[1,-2,2,-1,2]
	\Triangle[2,-2,-1,2,2]
	\Triangle[1,-3,2,2,2]
\end{tikzpicture} \\ \hline
\end{tabular}
\end{center}
\caption{The second half of the $S \in \mathcal{S}(G)$ for Example~\ref{ex:main}, together with their activities and fibers $\varphi^{-1}(S)$.} \label{fig:expart2}
\end{figure} \clearpage}

\begin{remark} \label{rem:wheredoorientededgesgo}
There can be $S \in \mathcal{S}(G)$ with $O_1,O_2 \in \varphi^{-1}(S)$ but~$O_1^{+} \cap O_2^{-}\neq \varnothing$. For instance, with the setup of Example~\ref{ex:main}, for the following fourientations $O_1$ and $O_2$, we have~$O_1, O_2 \in \varphi^{-1}(\{e_3\})$ but $e_3 \in O_1^{-}$ while~$e_3 \in O_2^{+}$:
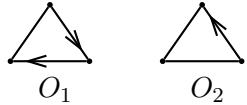
\begin{center}
\begin{tikzpicture}[scale=1.5]
\Triangle[0,0,0,1,-1]
\node at (0.2,-0.5) {$O_1$};
\end{tikzpicture} \qquad \begin{tikzpicture}[scale=1.5]
\Triangle[0,0,0,0,1]
\node at (0.2,-0.5) {$O_2$};
\end{tikzpicture}
\end{center}
Thus~$\varphi^{-1}(S)$ is not just a direct interpolation between a  fourientation in which every edge is oriented and a fourientation in which no edge is oriented. Put differently, it does not seem possible to deduce Theorem~\ref{thm:fouracttutteeval} from Theorems~\ref{thm:gordontraldi} and~\ref{thm:lasvergnas} alone. On the other hand, as suggested to us by~E.~Gioan, we can at least say the following: say the edges of $G$ are~$e_1 < \cdots < e_p$; then if $O_1,O_2 \in \varphi^{-1}(S)$ for some $S \in \mathcal{S}(G)$ and~$\{e_1,\ldots,e_i\} \cap O_1^{o} = \{e_1,\ldots,e_i\} \cap O_2^{o}$ for some~$i\in \{1,\ldots,p\}$, we in fact can conclude that~$O_1$ and~$O_2$ agree when restricted to~$\{e_1,\ldots,e_i\}$.
\end{remark}

\begin{remark} \label{rem:mysteryact}
It seems hard to give explicit descriptions of $I^{b}(O)$ and $L^{u}(O)$ using the intrinsic properties of a fourientation~$O$. One particularly intriguing fact is that whether an edge~$e \in E(G)$ belongs to~$I^{b}(O)$ depends on more than just the status in~$O$ of all edges~$f \in E(G)$ with~\mbox{$f \geq e$}. For instance, with the setup of Example~\ref{ex:main}, for the following fourientations we have~$e_2 \in I(O_1^b)$ but $e_2 \notin I(O_2^b)$ even though $O_1$ and $O_2$ look the same when restricted to~$\{e_2,e_3\}$:
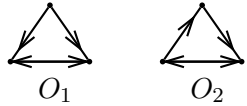
\begin{center}
\begin{tikzpicture}[scale=1.5]
    \Triangle[0,0,-1,2,-1]
    \node at (0.2,-0.5) {$O_1$};
\end{tikzpicture} \qquad \begin{tikzpicture}[scale=1.5]
    \Triangle[0,0,1,2,-1]
    \node at (0.2,-0.5) {$O_2$};
\end{tikzpicture}
\end{center}
This is in marked contrast to the situation for~$I^{o}(O)$, $L^{o}(O)$, $I^{u}(O)$, $L^{b}(O)$, or indeed the Gordon-Traldi or Las Vergnas activities where an edge's being active depends only on the status of edges greater than or equal to it in the total edge order. Somehow~$I^{u}(O)$ and~$L^{b}(O)$ are ``local'' (see the versatility of~$\sigma_u$ and~$\sigma_b$) whereas~$I^{b}(O)$ and~$L^{u}(O)$ must be ``global.''
\end{remark}

\begin{remark}
In the Las Vergnas orientation activities formula, Theorem~\ref{thm:lasvergnas}, the variables $x$ and $w$ play a symmetric role, as do the variables $y$ and $z$. This symmetry can be explained combinatorially as follows. For~$O \in \mathcal{O}(G)$ and~$S \in \mathcal{S}(G)$, let us use~$O^{\mathrm{rev}(S)}$ to denote the orientation obtained from $O$ by reversing the orientation of all edges in $S$; and let~$O^{\mathrm{rev}} := O^{\mathrm{rev}(E(G))}$. It is a well-known fact that if~$Cu(O)$ denotes all edges in $E(G)$ that belong to directed cuts of $O$ (the \emph{acyclic part} of~$O$) and~$Cy(O)$ denotes all edges in $E(G)$ that belong to directed cycles of $O$ (the \emph{cyclic part} of $O$) then~$\{Cu(O),Cy(O)\}$ is a partition of $E(G)$ for all $O \in \mathcal{O}(G)$. Thus the map~$O \mapsto O^{\mathrm{rev}(Cu(O))}$ is an involution on $\mathcal{O}(G)$ that swaps the roles of~$x$ and~$w$ in Theorem~\ref{thm:lasvergnas}. Similarly,~$O \mapsto O^{\mathrm{rev}(Cy(O))}$ swaps the roles of~$y$ and~$z$ in Theorem~\ref{thm:lasvergnas}. In particular,~$O \mapsto O^{\mathrm{rev}}$ preserves activities. When~$(k_1,k_2,l,m) := (1,1,1,1)$, the variables in~$\{x,w,\hat{x},\hat{w}\}$ and the variables in~$\{y,z,\hat{y},\hat{z}\}$ play symmetric roles in Theorem~\ref{thm:fouracttutteeval}. One might hope that there would be a way to see this symmetry combinatorially by reversing edge orientations. However, with the setup of Example~\ref{ex:main}, the four fourientations $O_1$, $O_2$, $O_3$, and~$O_4$ pictured below are all related through various kinds of ``fourientation reversals'' (which either swap unoriented and bioriented edges, or leave them fixed) and yet they exhibit all three different activity patterns that occur in~$K_3$:
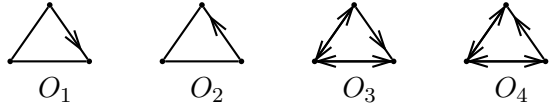
\begin{center}
\begin{tikzpicture}[scale=1.5]
    \Triangle[0,0,0,0,-1]
    \node at (0.2,-0.5) {$O_1$};
\end{tikzpicture} \qquad \begin{tikzpicture}[scale=1.5]
    \Triangle[0,0,0,0,1]
    \node at (0.2,-0.5) {$O_2$};
\end{tikzpicture} \qquad \begin{tikzpicture}[scale=1.5]
    \Triangle[0,0,2,2,-1]
    \node at (0.2,-0.5) {$O_3$};
\end{tikzpicture} \qquad \begin{tikzpicture}[scale=1.5]
    \Triangle[0,0,2,2,1]
    \node at (0.2,-0.5) {$O_4$};
\end{tikzpicture}
\end{center}
Nevertheless we can see the symmetry of~$\{x,w,\hat{x},\hat{w}\}$ combinatorially by using the map~$\varphi$. In the Gordon-Traldi formula, Theorem~\ref{thm:gordontraldi}, the variables $\hat{x}$ and $\hat{w}$ play a symmetric role, as do the variables $\hat{y}$ and $\hat{z}$. This symmetry can be seen by ``toggling'' rather than reversing orientations. Specifically, it follows from Lemma~\ref{lem:intervals} that the map~$S \mapsto S\Delta\hat{I}(S)$ (where~$\Delta$ denotes symmetric difference) is an involution on~$\mathcal{S}(G)$ that swaps the roles of~$\hat{x}$ and~$\hat{w}$ in Theorem~\ref{thm:gordontraldi}; and similarly~$S \mapsto S\Delta\hat{L}(S)$ is an involution on $\mathcal{S}(G)$ that swaps the roles of~$\hat{y}$ and~$\hat{z}$. We can push this activity toggling through~$\varphi$. Let $\widetilde{\varphi}\colon \mathcal{O}^4(G) \stackrel{\sim}\longrightarrow \mathcal{S}(G)\times\mathcal{S}(G)$ denote the bijection $\widetilde{\varphi}(O) := (\varphi(O),O^o)$. Then~\mbox{$O \mapsto \widetilde{\varphi}^{-1}(\varphi(O)\Delta(\hat{I}(\varphi(O))\setminus O^o),O^o)$} is an involution on $\mathcal{O}^{4}(G)$ that swaps the roles of~$\hat{x}$ and~$\hat{w}$ in Theorem~\ref{thm:fouracttutteeval} (when, as stated before, $(k_1,k_2,l,m) := (1,1,1,1)$). Similarly,~$O \mapsto \widetilde{\varphi}^{-1}(\varphi(O)\Delta(\hat{I}(\varphi(O))\cap O^o),O^o)$ swaps the roles of~$x$ and~$w$ in Theorem~\ref{thm:fouracttutteeval}. Finally,~$O \mapsto \widetilde{\varphi}^{-1}(\varphi(O),O^o\Delta(\hat{I}(\varphi(O))\cap \varphi(O)))$ swaps the roles of~$x$ and~$\hat{x}$ in Theorem~\ref{thm:fouracttutteeval}. The symmetry of~$\{y,z,\hat{y},\hat{z}\}$ can of course be explained similarly by toggling the edges in~$\hat{L}(\varphi(O))$ instead.
\end{remark}

\begin{remark} \label{rem:k4}
Considerations of space prohibit a complete analysis of $K_4$ like the analysis of $K_3$ presented above, but we take a moment to mention two interesting properties of the orientations $O_1$ and $^{e_{6}}O_{1}=O_{2}$ of $K_4$ that are pictured below (with edge order $e_1 < e_2 < \ldots < e_6$):
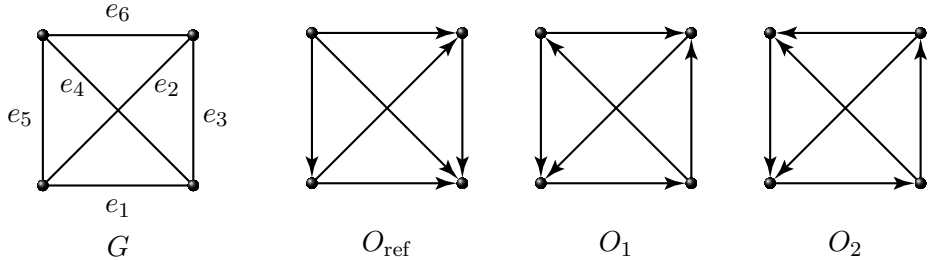
\begin{center}
\begin{tikzpicture}[scale=1]
	\SetFancyGraph
	\Vertex[NoLabel,x=-1,y=0]{v_1}
	\Vertex[NoLabel,x=-1,y=2]{v_2}
	\Vertex[NoLabel,x=1,y=0]{v_3}
    \Vertex[NoLabel,x=1,y=2]{v_4}
	\Edges[style={thick}](v_1,v_2)
	\Edges[style={thick}](v_1,v_3)
    \Edges[style={thick}](v_2,v_4)
	\Edges[style={thick}](v_2,v_3)
    \Edges[style={thick}](v_3,v_4)
    \Edges[style={thick}](v_1,v_4)
	\node at (1.3,.9) {$e_3$};
	\node at (0,-0.3) {$e_1$};
	\node at (-1.3,.9) {$e_5$};
    \node at (0,2.3) {$e_6$};
    \node at (.65,1.3) {$e_2$};
    \node at (-.6,1.3) {$e_4$};
	\node at (0,-0.8){$G$};
\end{tikzpicture} \qquad \begin{tikzpicture}[scale=1]
    \SetFancyGraph
	\Vertex[NoLabel,x=-1,y=0]{v_1}
	\Vertex[NoLabel,x=-1,y=2]{v_2}
	\Vertex[NoLabel,x=1,y=0]{v_3}
    \Vertex[NoLabel,x=1,y=2]{v_4}
	\Edges[style={thick,->,>=mytip}](v_2,v_1)
	\Edges[style={thick,->,>=mytip}](v_1,v_3)
    \Edges[style={thick,->,>=mytip}](v_2,v_4)
	\Edges[style={thick,->,>=mytip}](v_2,v_3)
    \Edges[style={thick,->,>=mytip}](v_4,v_3)
    \Edges[style={thick,->,>=mytip}](v_1,v_4)
	\node at (0,-0.8){$O_{\mathrm{ref}}$};
\end{tikzpicture} \qquad \begin{tikzpicture}[scale=1]
	\SetFancyGraph
	\SetFancyGraph
	\Vertex[NoLabel,x=-1,y=0]{v_1}
	\Vertex[NoLabel,x=-1,y=2]{v_2}
	\Vertex[NoLabel,x=1,y=0]{v_3}
    \Vertex[NoLabel,x=1,y=2]{v_4}
	\Edges[style={thick,->,>=mytip}](v_2,v_1)
	\Edges[style={thick,->,>=mytip}](v_1,v_3)
    \Edges[style={thick,->,>=mytip}](v_2,v_4)
	\Edges[style={thick,->,>=mytip}](v_3,v_2)
    \Edges[style={thick,->,>=mytip}](v_3,v_4)
    \Edges[style={thick,->,>=mytip}](v_4,v_1)
	\node at (0,-0.8){$O_{1}$};
\end{tikzpicture} \qquad \begin{tikzpicture}[scale=1]
	\SetFancyGraph
	\SetFancyGraph
	\Vertex[NoLabel,x=-1,y=0]{v_1}
	\Vertex[NoLabel,x=-1,y=2]{v_2}
	\Vertex[NoLabel,x=1,y=0]{v_3}
    \Vertex[NoLabel,x=1,y=2]{v_4}
	\Edges[style={thick,->,>=mytip}](v_2,v_1)
	\Edges[style={thick,->,>=mytip}](v_1,v_3)
    \Edges[style={thick,->,>=mytip}](v_4,v_2)
	\Edges[style={thick,->,>=mytip}](v_3,v_2)
    \Edges[style={thick,->,>=mytip}](v_3,v_4)
    \Edges[style={thick,->,>=mytip}](v_4,v_1)
	\node at (0,-0.8){$O_{2}$};
\end{tikzpicture}
\end{center}
The recursive definition of~$\varphi$ given in Section~\ref{sec:main} indicates that $\varphi(O_{1})=\{e_{6},e_{3},e_{2}\}$, $\varphi(O_{2})=\{e_{5},e_{3},e_{2}\}$, $\varphi(O_{1}^{\mathrm{rev}})=\{e_{5},e_{3},e_{2},e_{1}\}$ and~$\varphi(O_{2}^{\mathrm{rev}})=\{e_{6},e_{3},e_{2},e_{1}\}$. The first interesting feature of these orientations of~$K_4$ is that~$O_1$ and~$O_2$ both satisfy conditions~(1) and~(2) of Lemma~\ref{lem:key}, with 
\begin{gather*}
I(O_{1})=I(O_{2})=I(O_{1}/e_{6})=I(O_{2}/e_{6})=I(O_{1}-e_{6})=I(O_{2}-e_{6})=\varnothing; \\
L(O_{1})=L(O_{2})=L(O_{1}/e_{6})=L(O_{2}/e_{6})=L(O_{1}-e_{6})=L(O_{2}-e_{6})=\{e_{1}\}.
\end{gather*}
This first feature implies that our recursive definition of~$\varphi$ really does need some kind of~``tiebreaker'' clause to cover the case where conditions~(1) and~(2) of Lemma~\ref{lem:key} both hold for $O$. In the definition of~$\varphi$ given in Section~\ref{sec:main} we use whether $e_{\mathrm{max}}$ agrees or disagrees with~$O_{\mathrm{ref}}$ as a tiebreaker. But note that any time conditions~(1) and~(2) of Lemma~\ref{lem:key} both hold for $O$, the value $\varphi(O)$ is not uniquely determined by the properties required of it in Theorem~\ref{thm:fouracttutteeval}. Thus this example shows that $\varphi$ is very far from being unique.  In the special case of orientations, this feature of our work contrasts with ``the active bijection" of Gioan-Las Vergnas~\cite{gioan2009active}~\cite{gioanpreprint}~\cite{gioan2005activity} which fits into a deletion-contraction framework, but gives a canonical (up to edge order) bijection between certain activity classes of orientations and Crapo intervals of subgraphs.

Recall the equivalence relation $\sim$ given in Definition \ref{def:crapointervals}.  The second interesting feature of these orientations of $K_4$ is that for~$i\neq j\in \{1,2\}$ we have~$\varphi(O_{i}^{\mathrm{rev}})\sim \varphi(O_{j})$ but $\varphi(O_{i}^{\mathrm{rev}}) \not \sim \varphi(O_{i})$. This aspect is also different from the active bijection of Gioan-Las~Vergnas, where any orientation~$O$ and its reverse~$O^{\mathrm{rev}}$ are mapped to subgraphs in the same Crapo interval. In the terminology of Gioan-Las~Vergnas, $\varphi$ preserves activities but does not preserve the ``active partition.'' (We do not discuss this partition here.) This behavior is not peculiar to~$K_4$; whenever $e_{\mathrm{max}}$ is neither an isthmus nor a loop, $e_{\mathrm{max}} \in O^o$, and $O$ satisfies conditions~(1) and~(2) of Lemma~\ref{lem:key}, it will be the case that~$\varphi(O^{\mathrm{rev}}) \not \sim \varphi(O)$.
\end{remark}

\section{Extensions of our results} \label{sec:future}

In this section we discuss some possible extensions of our results and future directions.

\begin{remark}
We could take as additional input data two more arbitrary reference orientations of our graph, call them~$O_{\mathrm{ref}}^{Cu}$ and~$O_{\mathrm{ref}}^{Cy}$, and for each~$O\in\mathcal{O}^{4}(G)$ define
\begin{align*}
I^{+}(O) &:= I^o(O) \cap \{e\in O^o: \textrm{$e$ agrees with $O_{\mathrm{ref}}^{Cu}$} \};\\
I^{-}(O) &:= I^o(O) \cap \{e\in O^o: \textrm{$e$ disagrees with $O_{\mathrm{ref}}^{Cu}$} \};\\
L^{+}(O) &:= L^o(O) \cap \{e\in O^o: \textrm{$e$ agrees with $O_{\mathrm{ref}}^{Cy}$} \};\\
L^{-}(O) &:= L^o(O) \cap \{e\in O^o: \textrm{$e$ disagrees with $O_{\mathrm{ref}}^{Cy}$} \},
\end{align*}
instead of the current definitions of~$I^{\delta}(O)$ and~$L^{\delta}(O)$. Theorems~\ref{thm:main} and~\ref{thm:fouracttutteeval} would go through exactly and with the same proofs. However, we chose to present these theorems as stated because not much strength is gained by allowing the extra reference orientations~$O_{\mathrm{ref}}^{Cu}$ and~$O_{\mathrm{ref}}^{Cy}$.
\end{remark}

\begin{remark} \label{rem:k1andk2}
When $k_1 := k_2 := k$ in Theorem~\ref{thm:fouracttutteeval}, it is clear that the parameter~$k$ records the number of oriented edges in a fourientation. The effect of splitting the variable $k$ into $k_1$ and $k_2$ is not clear from the fourientation point of view because we do not know how to describe the sets~$O^{o}\cap \varphi(O)$ and~$O^{o}\setminus \varphi(O)$ without appealing to the recursive definition of~$\varphi$. Indeed, any simple, explicit description of~$O^{o}\cap \varphi(O)$ would immediately yield a simple, non-recursive activity-preserving bijection between~$\mathcal{O}(G)$ and~$\mathcal{S}(G)$. Such a simple bijection seems too much to hope for.  Finally, we mention the possibility that the mystery surrounding the variables $k_1$ and $k_2$ may be closely related to the mystery surrounding the activities $I^b(O)$ and $L^u(O)$.
\end{remark}

\begin{remark}
In~\cite{kung2010convolution}, Kung offers a ``convolution-multiplication'' formula for the Tutte polynomial that generalizes the convolution formula for the Tutte polynomial due to \'{E}tienne-Las~Vergnas~\cite{etienne1998external} and Kook-Reiner-Stanton~\cite{kook1999convolution}. Using the same notation as Kung~\cite{kung2010convolution}, this convolution-multiplication formula is
{\small \begin{equation} \label{eq:kung}
 T_G(\lambda \xi + 1, xy + 1) = \sum_{S \in \mathcal{S}(G)} \lambda^{\kappa(S) - \kappa} (-y)^{|S|+\kappa(S)-n} T_{G\rvert_S}(1-\lambda,1-x)T_{G/S}(1+\xi,1+y),
\end{equation}}where $G \rvert_S$ denotes the graph obtained from $G$ by restricting to $S$, and $G/S$ denotes the graph obtained from $G$ by contracting all edges in $S$. (Note that $G\rvert_S$ is the same as the spanning subgraph $S$.) One consequence of~(\ref{eq:kung}) is the following:
\begin{equation} \label{eq:lambda}
2^{g}T_G\left(2\lambda+1, \frac{1+y}{2}\right) = \sum_{S \in \mathcal{S}(G)} \lambda^{\kappa(S) - \kappa}T_{G\rvert_S}(\lambda+1,y).
\end{equation}
The special cases $\lambda = 1$ and $\lambda = \frac{1}{2}$ of~(\ref{eq:lambda}) have interpretations in terms of fourientation activities. For instance, with  $\lambda = 1$, equation~(\ref{eq:lambda}) becomes
\[ 2^{g}T_{G}\left(3,\frac{1+y}{2}\right) = \sum_{S \in \mathcal{S}(G)} T_{G\rvert_S}(2,y),\]
which can be seen as two different ways of writing the generating function
\begin{equation} \label{eq:partorcycleact}
 \sum_{O \in \mathcal{O}^{4}(G), \; O^b = \varnothing} \left(\frac{y}{2}\right)^{\lvert L^o(O)\rvert}.
\end{equation}
On the one hand, Theorem~\ref{thm:fouracttutteeval} can be applied directly to show that~(\ref{eq:partorcycleact}) is equal to~$2^{g}T_{G}\left(3,\frac{1+y}{2}\right)$. On the other hand, Definition~\ref{def:fouractivity} makes it clear that~(\ref{eq:partorcycleact}) is also
\begin{align*}
\sum_{O \in \mathcal{O}^{4}(G), \; O^b = \varnothing} \left(\frac{y}{2}\right)^{\lvert L^o(O) \rvert} &= \sum_{S \in \mathcal{S}(G)}  \sum_{\substack{O \in \mathcal{O}^{4}(G), \\ O^b = \varnothing, \; O^u = E \setminus S}}  \left(\frac{y}{2}\right)^{\vert L^o(O) \rvert}  \\
&=  \sum_{S \in \mathcal{S}(G)} \sum_{\substack{O \in \mathcal{O}^{4}(G\rvert_S), \\ O^u = O^b = \varnothing}} \left(\frac{y}{2}\right)^{\vert L^o(O) \rvert},
\end{align*}
and this last expression is~$ \sum_{S \in \mathcal{S}(G)} T_{G\rvert_S}(2,y)$, again thanks to Theorem~\ref{thm:fouracttutteeval}. Similarly, with $\lambda = \frac{1}{2}$, equation~(\ref{eq:lambda}) becomes
\[ 2^{g}T_{G}\left(2,\frac{1+y}{2}\right) = \sum_{S \in \mathcal{S}(G)} \left(\frac{1}{2}\right)^{\kappa(S) -\kappa} T_{G\rvert_S}\left (\frac{3}{2},y\right),\]
which can be seen by applying Theorem~\ref{thm:fouracttutteeval} to
\[\left(\frac{1}{2}\right)^{n-\kappa} \cdot \sum_{O \in \mathcal{O}^{4}(G)} \left(\frac{y}{3}\right)^{\vert L^o(O) \cup L^b(O) \rvert} = \left(\frac{1}{2}\right)^{n-\kappa} \cdot \sum_{S \in \mathcal{S}(G)} \sum_{O \in \mathcal{O}^{4}(G\rvert_S), \; O^u = \varnothing} \left(\frac{y}{3}\right)^{\vert L^o(O) \cup L^b(O) \rvert}. \]
It is also worth discussing the case $\lambda = 0$ of equation~(\ref{eq:lambda}). With $\lambda = 0$, equation~(\ref{eq:lambda}) becomes
\[ 2^{g}T_{G}\left(1,\frac{1+y}{2}\right) = \sum_{S \in \mathcal{S}(G), \; \kappa(S) = \kappa} T_{G\rvert_S}\left (1,y\right),\]
which would following by applying Theorem~\ref{thm:fouracttutteeval} to
\begin{equation} \label{eq:qconnected}
\sum_{\substack{O \in \mathcal{O}^{4}(G), \; O^b = \varnothing, \\ I^{u}(O) = I^{-}(O) = \varnothing}} \left(\frac{y}{2}\right)^{\vert L^o(O) \rvert} = \sum_{\substack{S \in \mathcal{S}(G), \\ \kappa(S) = \kappa}} \quad \sum_{\substack{O \in \mathcal{O}^{4}(G\rvert_S), \; O^b = O^u = \varnothing, \\ I^{u}(O) = I^{-}(O) = \varnothing}} \left(\frac{y}{2}\right)^{\vert L^o(O) \rvert}.
\end{equation}
if we could show directly that indeed equation~(\ref{eq:qconnected}) holds. Note that we can show directly that the cases $y=0$ and $y=2$ of equation~(\ref{eq:qconnected}) hold: this is thanks to the fact, mentioned above, that (when $G$ is connected) we can choose data so that the set of fourientations~$O \in \mathcal{O}^{4}(G)$ which satisfy~$I^{u}(O) = I^{-}(O) = \varnothing$ is exactly the set of $q$-connected fourientations for some choice of root $q \in V(G)$ (see also Backman-Hopkins~\cite[\S4.4]{backman2015fourientations}). Of course, there are also dual versions of all of these identities which exchange cut and cycle activities. It would be interesting to further explore the connection between our fourientation activities and convolution formulas.
\end{remark}

\begin{remark}
Several different parametrized or weighted versions of the Tutte polynomial have appeared in the literature~\cite{bollobas1999tutte}~\cite{ellismonaghan2006parametrized}~\cite{sokal2005multivariate}~\cite{traldi1989dichromatic}~\cite{zaslavsky1992strong}. Let us focus on the edge-weighted Tutte polynomial analog that Zaslavsky~\cite{zaslavsky1992strong} calls the \emph{normal function}. Given a commutative ring $R$ and any two functions $\alpha,\beta:E(G)\rightarrow R$, the normal function is
\[N_{G}(\alpha(e),\beta(e);u,v):=\sum_{S\in\mathcal{S}(G)}\left({\prod_{e\in S}}\alpha(e)\right)\left({\prod_{e\notin S}}\beta(e)\right)(u-1)^{\kappa(S)-\kappa}(v-1)^{\left\vert S\right\vert -n+\kappa(S)}.\]
Suppose for a moment that $\alpha(e)\equiv\alpha$ and $\beta(e)\equiv\beta$ are constants; we then adopt the notation~$\gamma:=\alpha+\beta(u-1)$ and~$\delta:=\alpha(v-1)+\beta$. The defining formula of the normal function becomes
\begin{gather} \label{eq:normaltotutte}
N_{G}(\alpha,\beta;u,v)=N_{G}\left(\alpha,\beta;\frac{\beta+\gamma-\alpha}{\beta},\frac{\alpha+\delta-\beta}{\alpha}\right)=\\
\sum_{S\in\mathcal{S}(G)}\alpha^{\left\vert S\right\vert }\beta^{\left\vert E(G)\setminus S\right\vert }\left(\frac{\gamma-\alpha}{\beta}\right)^{\kappa(S)-\kappa}\left(\frac{\delta-\beta}{\alpha}\right)^{\left\vert S\right\vert -n+\kappa(S)}=\nonumber\\
\sum_{S\in\mathcal{S}(G)}\alpha^{n-\kappa(S)}\beta^{\left\vert E(G)\setminus S\right\vert +\kappa-\kappa(S)}(\gamma-\alpha)^{\kappa(S)-\kappa}(\delta -\beta)^{\left\vert S\right\vert -n+\kappa(S)}=\nonumber\\
\alpha^{n-\kappa}\beta^{\left\vert E(G)\right\vert -n+\kappa}\cdot\sum_{S\in\mathcal{S}(G)}\left(\frac{\gamma}{\alpha}-1\right)^{\kappa(S)-\kappa}\left(\frac{\delta}{\beta}-1\right)^{\left\vert S\right\vert -n+\kappa(S)}=\nonumber\\
\alpha^{n-\kappa}\beta^{g} T_{G}\left(\frac{\gamma}{\alpha},\frac{\delta}{\beta}\right),\nonumber
\end{gather}
where $T_{G}$ is the standard unweighted Tutte polynomial of $G$. Consider the following assignment of variables, where $k_1$, $k_2$, $l$, $m$, $w$, $\hat{w}$, $x$, $\hat{x}$, $y$, $\hat{y}$, $z$, and~$\hat{z}$ are independent indeterminates:
\begin{align} \label{eq:assign}
\alpha &:=k_1+m;\\
\beta & :=k_2+l;\nonumber\\
\gamma &:=k_1x+k_2w+m\hat{x}+l\hat{w};\nonumber\\
\delta &:=k_2y+k_1z+l\hat{y}+m\hat{z}.\nonumber
\end{align}
With $\alpha$, $\beta$, $\gamma$, and~$\delta$ as in~(\ref{eq:assign}), the last line of~(\ref{eq:normaltotutte}) becomes
\[(k_1+m)^{n-\kappa}(k_2+l)^{g}T_G\left(\frac{k_1x+k_2w+m\hat{x}+l\hat{w}}{k_1+m},\frac{k_2y+k_1z+l\hat{y}+m\hat{z}}{k_2+l}\right).\]
Thus Theorem~\ref{thm:fouracttutteeval} implies
{\footnotesize
\begin{gather} \label{eq:mainnormal}
N_{G}\left(k_1+m,k_2+l;1+\frac{k_1x+k_2w+m\hat{x}+l\hat{w} - k_1 -m}{k_2+l},1+\frac{k_2y+k_1z+l\hat{y}+m\hat{z} - k_2 -l} {k_1+m}\right)= \\
\sum_{O \in \mathcal{O}^{4}(G)} \hspace{-0.3cm} k_1^{\left\vert O^{o} \cap \varphi(O) \right\vert} k_2^{\left\vert O^{o} \setminus \varphi(O) \right\vert}l^{\left\vert O^{u} \right\vert}m^{\left\vert O^{b}\right\vert}x^{\left\vert I^{+}(O) \right\vert} w^{\left\vert I^{-}(O) \right\vert} {\hat{x}}^{\left\vert I^{b}(O)\right\vert} {\hat{w}}^{\left\vert I^{u}(O) \right\vert} y^{\left\vert L^{+}( O) \right\vert} z^{\left\vert L^{-}(O) \right\vert}{\hat{y}}^{\left\vert L^{u}(O) \right\vert} {\hat{z}}^{\left\vert L^{b}(O) \right\vert}. \nonumber
\end{gather}}In general it seems hard to deform~(\ref{eq:mainnormal}) in a natural way into a fourientation activities expression involving nonconstant edge weights, because the left-hand side of~(\ref{eq:mainnormal}) mentions $k_1$, $k_2$, $l$ and $m$ in the formulas for $\alpha$ and $\beta$ (which would presumably vary with~$e$) and also in the formulas for $u$ and $v$ (which do not vary with $e$). However, we can offer an interesting edge-weighted deformation of a special case of~(\ref{eq:mainnormal}). Let us set~$x:=\hat{x}:=y:=\hat{y}:=1$ and $\hat{w}:=w$ and $\hat{z}:=z$. Then the appearances of $k_1$, $k_2$, $l$ and~$m$ in the formulas for $u$ and $v$ cancel out, and~(\ref{eq:mainnormal}) becomes
\begin{gather} \label{eq:normalspec}
N_{G}(k_1+m,k_2+l;1+w,1+z) = \\
\sum_{O \in \mathcal{O}^{4}(G)}  k_1^{\left\vert O^{o} \cap \varphi(O) \right\vert} k_2^{\left\vert O^{o} \setminus \varphi(O) \right\vert} l^{\left\vert O^u\right\vert} m^{\left\vert O^b\right\vert} w^{\left\vert I^{u}(O) \cup I^{-}(O) \right\vert}z^{\left\vert L^{b}(O)\cup L^{-}(O)\right\vert} \nonumber
\end{gather}
In fact, an edge-weighted version of~(\ref{eq:normalspec}) holds; namely,
\begin{gather} \label{eq:normalspecdeform} N_G(k_1(e)+m(e),k_2(e)+l(e);1+w,1+z) =\\
\sum_{O \in \mathcal{O}^{4}(G)} \prod_{e \in O^o \cap \varphi(O)} \hspace{-0.25cm} k_1(e) \hspace{-0.25cm} \prod_{e \in O^o \setminus \varphi(O)} \hspace{-0.25cm} k_2(e) \prod_{e \in O^u}l(e) \prod_{e \in O^b}m(e) \cdot w^{\left\vert I^{u}(O) \cup I^{-}(O) \right\vert}z^{\left\vert L^{b}(O)\cup L^{-}(O)\right\vert} \nonumber,
\end{gather}
where $k_1(e)$, $k_2(e)$, $l(e)$, and $m(e)$ are arbitrary functions $E(G) \to R$. Indeed, from Theorem~\ref{thm:fouracttutteeval} it follows that
{\small \begin{gather*} 
\sum_{O \in \mathcal{O}^{4}(G)} \prod_{e \in O^o \cap \varphi(O)} \hspace{-0.25cm} k_1(e) \hspace{-0.25cm} \prod_{e \in O^o \setminus \varphi(O)} \hspace{-0.25cm} k_2(e) \prod_{e \in O^u}l(e) \prod_{e \in O^b}m(e) \cdot w^{\left\vert I^{u}(O) \cup I^{-}(O) \right\vert}z^{\left\vert L^{b}(O)\cup L^{-}(O)\right\vert} = \\ 
\sum_{S \in \mathcal{S}(G)} \prod_{e \in \hat{I}(S)\setminus S} \hspace{-0.3cm} (wk_2(e)+wl(e)) \hspace{-0.3cm} \prod_{e \in \hat{L}(S)\cap S} \hspace{-0.3cm} (zk_1(e)+zm(e)) \hspace{-0.3cm} \prod_{e\in S \setminus \hat{L}(S)} \hspace{-0.3cm} (k_1(e)+m(e)) \hspace{-0.3cm} \prod_{e\notin S \cup \hat{L}(S)} \hspace{-0.3cm} (k_1(e)+l(e)) = \\
\sum_{S\in\mathcal{S}(G)} \left(\prod_{e \in S} (k_1(e)+m(e))\right) \left(\prod_{e \notin S}(k_2(e)+l(e))\right)w^{\left\vert\hat{I}(S)\setminus S\right\vert} z^{\left\vert\hat{L}(S)\cap S\right\vert} =\\ \sum_{S\in\mathcal{S}(G)} \left(\prod_{e \in S}( k_1(e)+m(e))\right) \left(\prod_{e \notin S} (k_2(e)+l(e))\right)w^{\kappa(S)-\kappa} z^{\left\vert S\right\vert -n+\kappa(S)} = \\
N_G(k_1(e)+m(e),k_2(e)+l(e);1+w,1+z),
\end{gather*}}where in the second-to-last line we applied Lemma~\ref{lem:rankact}. See~\cite[\S4.7, Remark~4.49]{backman2015fourientations} for discussion of some probabilistic results related to~(\ref{eq:normalspecdeform}).
\end{remark}

\begin{remark}
In~\cite[\S2]{backman2015fourientations}, Backman and Hopkins investigated ``Tutte properties'' of fourientations.  These are essentially the properties of potential cuts and cycles which determine classes of fourientations enumerated by Tutte polynomial evaluations.  A classification theorem was proven, showing that all Tutte properties are min-edge classes (which, as mentioned earlier, are enumerated by Theorem~\ref{thm:fouracttutteeval}) with the curious exception of the so-called ``cut (cycle) weird'' fourientations.  These are fourientations such that each potential cut (cycle) contains at least one oriented edge, and the minimum oriented edge in each potential cut (cycle) is oriented in agreement with its reference orientation.  What is indeed strange about this property, and thus inspired its name, is that it depends on the status of the minimum oriented edge in the potential cut (cycle), not the actual minimum edge.  Therefore, cut (cycle) weird fourientations do not fit into the framework of this paper. Nevertheless they are enumerated by the Tutte polynomial. It would be interesting to better understand cut (cycle) weird fourientations and how they relate to the approach presented here.
\end{remark}

\begin{remark}
Backman-Hopkins~\cite[\S4]{backman2015fourientations} showed how min-edge classes of fourientations appear in various algebraic settings. It would be extremely interesting to try to interpret the equation in Theorem~\ref{thm:fouracttutteeval} algebraically, e.g., as a formula for the Betti numbers of some polyhedral complex or the Hilbert series of some polynomial ideal.
\end{remark}

\begin{remark}
We have restricted our discussion to graphs in the present paper, and we leave open the problem of extending Theorem~\ref{thm:main} to the more general setting of Las~Vergnas's account of Theorem~\ref{thm:lasvergnas} in~\cite{las1984tutte} (namely, oriented matroids and their perspectives). Indeed, part of our intent in writing this paper is to provide a strictly graph-theoretic proof of~Theorem~\ref{thm:lasvergnas}, which does not require the oriented matroid machinery. Such a proof may be extracted from our discussion by restricting one's attention to the parts of the key lemma that involve $I^o$ and $L^o$.
\end{remark}

\bibliographystyle{plain}
\bibliography{four_activities}

\appendix

\section {Generalized activities: a brief account} \label{sec:genact}

In this appendix we provide some details regarding the generalized activities description of the Tutte polynomial introduced by Gordon and Traldi~\cite{gordon1990generalized}. In particular we prove Lemmas~\ref{lem:intervals} and~\ref{lem:rankact} and Theorem~\ref{thm:gordontraldi} stated in Section~\ref{sec:setup}. We will let $G$ be a graph given with a fixed edge order $e_1 < \ldots < e_p$. Recall the definitions of the Gordon-Traldi activities~$\hat{I}(S)$ and~$\hat{L}(S)$ for~$S \in \mathcal{S}(G)$ given in Section~\ref{sec:main}: $e \in \hat{I}(S)$ if and only if $e$ is the min edge in some cut of~$S \setminus \{e\}$, and $e \in \hat{L}(S)$ if and only if $e \in \hat{L}(S)$ if and only if $e$ is the min edge in some cycle of $S \cup \{e\}$. Recall that~$Cu$ being a cut of~$S$ means that~$E(Cu) \cap S = \varnothing$ and~$Cy$ being a cycle of~$S$ means that~$E(Cy) \subseteq S$. Here we present an alternative description of~$\hat{I}(S)$ and~$\hat{L}(S)$ using deletion-contraction trees that is more algorithmically transparent.

If $p$ is a positive integer then a \emph{perfect binary tree of depth} $p$ is a rooted tree with one level 0 vertex (the root), two level $1$ vertices, four level $2$ vertices, and so on. If~$1\leq i<p$ then each vertex at level $i$ has a parent at level $i-1$, a sibling at level $i$, and two children at level $i+1$. The vertices at level $p$ are \emph{leaves}.

\begin{definition}
Let $G$ be a graph given with a fixed edge order $e_1 < \ldots < e_p$. The corresponding \emph{perfect deletion-contraction tree} $DCT(G)$ is a perfect binary tree of depth~$p$. The root of $DCT(G)$ is labeled $G$, and for $i\in\{1,\ldots,p\}$ the vertices at level~$i$ are labeled with the $2^{i}$ graphs obtained from $G$ by using deletion and contraction to remove the edges $e_{p},\ldots,e_{p-i+1}$. If $i<p$ then the children of the level $i$ vertex labeled~$H$ are~$H-e_{p-i}$ and $H/e_{p-i}$.
\end{definition}

The tree $DCT(G)$ has $2^{p}$ leaves, labeled by the edgeless graphs obtained from~$G$ by removing all of its edges via deletion or contraction. That is, for every subset of edges~$S\in \mathcal{S}(G)$, $DCT(G)$ has a leaf labeled by the graph~$(G/S)-(E(G)\setminus S)$; to simplify notation we will also use~$S$ to label this leaf. If~$S\in \mathcal{S}(G)$ then there is a unique path in~$DCT(G)$ from the root to the leaf labeled $S$. We denote this path~$P(S)$, and call it the \emph{branch} of~$DCT(G)$ corresponding to $S$. The level $i$ vertex of~$P(S)$ is labeled by the graph
\[G_{i}(S):=(G/(\{e_{p-i+1},\ldots,e_{p}\}\cap S))-(\{e_{p-i+1},\ldots,e_{p}\}\setminus S).\]
We think of this vertex of $P(S)$ as corresponding to $e_{p-i}$.

\begin{prop} \label{prop:altgenactdef}
For $S \in \mathcal{S}(G)$ and~$0 \leq i \leq p-1$, we have $e_{p-i} \in \hat{I}(S)$ if and only if $e_{p-i}$ is an isthmus in~$G_i(S)$. Similarly, we have $e_{p-i} \in \hat{L}(S)$ if and only if $e_{p-i}$ is a loop in $G_i(S)$.
\end{prop}
\begin{proof}
If $e_{p-i}$ is the min edge in some cut $Cu$ of $S\setminus \{e_{p-i}\}$, then in $G_i(S)$ all the other edges in~$E(Cu)$ will have already been deleted and so $e_i$ will indeed be an isthmus. Conversely, if~$e_{p-i}$ is an isthmus of~$G_i(S)$ it determines a cut $Cu$ of $G$ where all edges~$e_j \in E(Cu)$ with~$e_j \neq e_{p-i}$ satisfy~$j > p-i$ and~$e_j \notin S$. Thus~$e_{p-i} \in \hat{I}(S)$ in this case. The claim about~$\hat{L}(S)$ is proved analogously.
\end{proof}

We now prove Lemma~\ref{lem:intervals}, which for convenience we restate.

\begin{customlem}{\ref{lem:intervals}}
Let $S, T \in \mathcal{S}(G)$ with~$S \setminus (\hat{I}(S) \cup \hat{L}(S)) \subseteq T \subseteq S \cup \hat{I}(S) \cup \hat{L}(S)$. Then~$\hat{I}(S) = \hat{I}(T)$ and~$\hat{L}(S) = \hat{L}(T)$.
\end{customlem}
\begin{proof}
The hypothesis implies that the only differences between the two sequences of graphs~$G=G_{0}(S),G_{1}(S),\ldots,G_{p}(S)$ and~$G=G_{0}(T),G_{1}(T),\ldots,G_{p}(T)$ are that some isthmuses or loops may be deleted in one sequence and contracted in the other. These differences will not affect the status of any edge; precisely the same edges will be loops and isthmuses in the two sequences. So from Proposition~\ref{prop:altgenactdef} we conclude~$\hat{I}(T)=\hat{I}(S)$ and~$\hat{L}(T)=\hat{L}(S)$.
\end{proof}

Next we prove Lemma~\ref{lem:rankact}, which again we restate.

\begin{customlem}{\ref{lem:rankact}}
We have $\left\vert\hat{I}(S) \setminus S\right\vert = \kappa(S) - \kappa$ and $\left\vert \hat{L}(S) \cap S \right\vert = \kappa(S) + \left\vert S \right\vert-n$ for all~$S \in \mathcal{S}(G)$.
\end{customlem}

\begin{proof}
Let $S \in \mathcal{S}(G)$ and set $T := (S \setminus \hat{L}(S))\cup\hat{I}(S)$. By Lemma~\ref{lem:intervals},~$\hat{I}(T) = \hat{I}(S)$ and~$\hat{L}(T) = \hat{L}(S)$. Note that $T$ is a maximal forest: it has no cycles since~$\hat{L}(T) \cap T = \varnothing$ and it has $\kappa(T) = \kappa$ since~$\hat{I}(T) \setminus T = \varnothing$. Thus the claimed formulas certainly hold when~$S$ is replaced by $T$ because both sides are equal to zero. Also, Lemma~\ref{lem:intervals} shows that these formula continue to hold as we add elements of~$\hat{I}(T)$ or remove elements of~$\hat{L}(T)$. So they hold for $S$ as well.
\end{proof}

Now recall the equivalence relation~$\sim$ on~$\mathcal{S}(G)$ defined in Definition~\ref{def:crapointervals}: $S \sim T$ if and only if~$S \setminus (\hat{I}(S) \cup \hat{L}(S)) \subseteq T \subseteq S \cup \hat{I}(S) \cup \hat{L}(S)$. Observe that for each $S \in \mathcal{S}$, we have $S \sim T$ for a unique maximal forest $T$ of $G$; namely,~$T := (S \setminus \hat{L}(S))\cup\hat{I}(S)$.

\begin{cor}
\label{activerank}
Let $S_{0}\in \mathcal{S}(G)$. Then
\begin{gather*}
\sum_{S_{0}\sim S\mathcal{\in S}(G)}\hat{x}^{\left\vert \hat{I}(S)\cap S\right\vert}\hat{w}^{\left\vert \hat{I}(S)\setminus S\right\vert }\hat{y}^{\left\vert\hat{L}(S)\setminus S\right\vert }\hat{z}^{\left\vert \hat{L}(S)\cap S\right\vert }=\\
\sum_{S_{0}\sim S\mathcal{\in S}(G)}(\hat{x}+\hat{w}-1)^{\kappa(S)-\kappa}(\hat{y}+\hat{z}-1)^{\left\vert S\right\vert -n+\kappa(S)}.
\end{gather*}
\end{cor}

\begin{proof}
For convenience we replace $S_{0}$ by the maximal forest $(S \setminus \hat{L}(S))\cup\hat{I}(S)$ equivalent to it. According to Lemma~\ref{lem:intervals}, the equivalence class of $S_{0}$ under $\sim$ includes all sets of the form $S=(S_0\cup S_{L})\setminus S_I$, where~$S_{I}$ is an arbitrary subset of $\hat{I}(S_{0})$ and $S_{L}$ is an arbitrary subset of $\hat{L}(S_{0})$. Consequently,
\begin{gather*}
\sum_{S_{0}\sim S\mathcal{\in S}(G)}\hat{x}^{\left\vert \hat{I}(S)\cap S\right\vert }\hat{w}^{\left\vert\hat{I}(S)\setminus S\right\vert}\hat{y}^{\left\vert \hat{L}(S)\setminus S\right\vert}\hat{z}^{\left\vert \hat{L}(S)\cap S\right\vert}=\\
\sum_{i=0}^{\left\vert \hat{I}(S_{0})\right\vert }\sum_{j=0}^{\left\vert\hat{L}(S_{0})\right\vert }\binom{\left\vert \hat{I}(S_{0})\right\vert }{i}\binom{\left\vert \hat{L}(S_{0})\right\vert }{j}\hat{x}^{i}\hat{w}^{\left\vert\hat{I}(S_{0})\right\vert -i}\hat{y}^j\hat{z}^{\left\vert \hat{L}(S_{0})\right\vert -j}=\\
(\hat{x}+\hat{w})^{\left\vert \hat{I}(S_{0})\right\vert }(\hat{y}+\hat{z})^{\left\vert \hat{L}(S_{0})\right\vert }.
\end{gather*}
Also, Lemma~\ref{lem:rankact} says that each~$S\sim S_{0}$ has~$\left\vert S_{I}\right\vert =\kappa(S)-\kappa$ and~$\left\vert S_{L}\right\vert=\left\vert S\right\vert -n+\kappa(S)$; so,
\begin{gather*}
(\hat{x}+\hat{w})^{\left\vert \hat{I}(S_{0})\right\vert }(\hat{y}+\hat{z})^{\left\vert \hat{L}(S_{0})\right\vert }=\\
(\hat{x}+\hat{w}-1+1)^{\left\vert \hat{I}(S_{0})\right\vert }(\hat{y}+\hat{z}-1+1)^{\left\vert \hat{L}(S_{0})\right\vert }=\\
\sum_{i=0}^{\left\vert \hat{I}(S_{0})\right\vert }\sum_{j=0}^{\left\vert\hat{L}(S_{0})\right\vert }\binom{\left\vert \hat{I}(S_{0})\right\vert }{i}\binom{\left\vert \hat{L}(S_{0})\right\vert }{j}(\hat{x}+\hat{w}-1)^{i}(\hat{y}+\hat{z}-1)^{j}=\\
\sum_{S_{0}\sim S\mathcal{\in S}(G)}(\hat{x}+\hat{w}-1)^{\kappa(S)-\kappa}(\hat{y}+\hat{z}-1)^{\left\vert S\right\vert -n+\kappa(S)}
\end{gather*}
as claimed.
\end{proof}

We can now prove Theorem~\ref{thm:gordontraldi}, which again we restate for convenience.

\begin{customthm}{\ref{thm:gordontraldi}}[Gordon-Traldi]
We have
\[T_G(\hat{x}+\hat{w},\hat{y}+\hat{z}) = \sum_{S \in \mathcal{S}(G)} {\hat{x}}^{\left\vert \hat{I}(S)\cap S\right\vert} {\hat{w}}^{\left\vert \hat{I}(S)\setminus S\right\vert} {\hat{y}}^{\left\vert \hat{L}(S)\setminus S\right\vert} {\hat{z}}^{\left\vert \hat{L}(S)\cap S\right\vert}.\]
\end{customthm}
\begin{proof}
By Corollary~\ref{activerank} we have
\begin{gather*}
\sum_{S \in \mathcal{S}(G)} {\hat{x}}^{\left\vert \hat{I}(S)\cap S\right\vert} {\hat{w}}^{\left\vert \hat{I}(S)\setminus S\right\vert} {\hat{y}}^{\left\vert \hat{L}(S)\setminus S\right\vert} {\hat{z}}^{\left\vert \hat{L}(S)\cap S\right\vert} = \\
\sum_{\substack{\textrm{$T$ maximal forest}\\\textrm{of $G$}}} \; \sum_{T \sim S\mathcal{\in S}(G)}\hat{x}^{\left\vert \hat{I}(S)\cap S\right\vert }\hat{w}^{\left\vert \hat{I}(S)\setminus S\right\vert}\hat{y}^{\left\vert \hat{L}(S)\setminus S\right\vert}\hat{z}^{\left\vert \hat{L}(S)\cap S\right\vert}=\\
\sum_{\substack{\textrm{$T$ maximal forest}\\\textrm{of $G$}}} \; \sum_{T \sim S\mathcal{\in S}(G)}(\hat{x}+\hat{w}-1)^{\kappa(S)-\kappa}(\hat{y}+\hat{z}-1)^{\left\vert S\right\vert -n+\kappa(S)}=\\
\sum_{S \in \mathcal{S}(G)}(\hat{x}+\hat{w}-1)^{\kappa(S)-\kappa}(\hat{y}+\hat{z}-1)^{\left\vert S\right\vert -n+\kappa(S)},
\end{gather*}
the corank-nullity expansion of~$T_G(\hat{x}+\hat{w},\hat{y}+\hat{z})$.
\end{proof}

\section{Proof of the key lemma} \label{sec:proofkeylem}

In this appendix we provide detailed arguments for all cases of the key lemma, Lemma~\ref{lem:key}. The usual proof of the Tutte polynomial's deletion-contraction recipe,
\[T_G=T_{G/e}+T_{G-e}\] 
for an ordinary edge $e$, involves the natural idea that a subset $S \subseteq E(G)$ may be associated with $G/e$ if $e \in S$, and with $G-e$ if $e \notin S$. The key lemma provides a similar dichotomy for fourientations: if $E(G)$ is ordered as $e_1 < \ldots < e_p$ then we may decide, for each fourientation $O$ of $G$, whether $O$ should be associated with $G/e_p$ or~$G-e_p$, and moreover these decisions can be made in such a way that $^{e_p}O$ is associated with the other of $G/e_p$, $G-e_p$. Regrettably, the similarity between the two dichotomies does not extend to their proofs: the deletion-contraction recipe takes about ten lines to prove, whereas the key lemma requires many more than that.

Throughout our proof of Lemma~\ref{lem:key} our setup will be as in the statement of that lemma: $G$ is a graph (which comes with the extra data of $<$, $O_{\mathrm{ref}}$, $\sigma_u$, and $\sigma_b$), its maximum edge~$e := e_{\mathrm{max}}$ is neither an isthmus nor loop, and~$O \in \mathcal{O}^{4}(G)$ is some fourientation. For convenience let~$e_1 < \ldots < e_p$ be the elements of~$E(G)$ (so~$e=e_p$). The facts that condition~(1) of the lemma holds when~$e$ is bioriented in $O$, and that condition~(2) of the lemma holds when~$e$ is unoriented in $O$, are easy to see because when~$e$ is bioriented there is an obvious one-to-one correspondence between potential cycles of~$O$ and potential cycles of~$O/e$, and when~$e$ is unoriented there is an obvious one-to-one correspondence between potential cuts of~$O$ and of~$O-e$. So we may assume that~$e$ is oriented: say~$e = \{u,v\}$ and~$e$ is oriented from~$u$ to~$v$ in~$O$. Our aim is then to show that at least one of conditions~(1) or~(2) holds, although we cannot say which one. Each of~(1) and~(2) consists of eight assertions; we will now show systematically that if any one of the eight assertions in~(2) fails then all eight of the assertions in~(1) hold. But first let us prove a useful lemma about potential paths.

\begin{definition}
Let~$O\in \mathcal{O}^{4}(G)$. Then a \emph{potential path} of~$O$ is a directed path of $G$ such that each edge of the path is either bioriented in $O$ or is oriented in $O$ consistently with the direction of the path.
\end{definition}

\begin{lemma}
\label{lem:preliminary}
Let $O\in \mathcal{O}^{4}(G)$. Suppose $u\neq v\in V(G)$, $P_{1}$ is a potential path of~$O$ from~$v$ to~$u$, and $P_{2}$ is a potential path of $O$ from $u$ to $v$. Let $P_{1}P_{2}$ be the closed walk obtained by concatenating $P_{1}$ and $P_{2}$. Then for each oriented edge $e$ that appears in~$P_{1}P_{2}$, there is a potential cycle $Cy$ of $O$ with $e\in Cy\subseteq P_{1}P_{2}$. Moreover, if a bioriented edge $e^{\prime }$ appears in $Cy$ then the orientation of $e^{\prime }$ that is consistent with the direction of~$Cy$ is also consistent with the direction of at least one of $P_{1},P_{2}$.
\end{lemma}

\begin{proof}
If $P_{1}$ and $P_{2}$ are both of length 1, then either there is no oriented edge in $P_{1}P_{2}$ or $P_{1}P_{2}$ consists of two parallel edges between $u$ and $v$. The lemma is satisfied in either case. So we proceed using induction on the length of the shorter of $P_{1}$ or $P_{2}$.

If $P_{1}$ and $P_{2}$ are internally edge- and vertex-disjoint then $P_{1}P_{2}$ is a circuit of $G$, so~$Cy=P_{1}P_{2}$ satisfies the lemma. Otherwise, there is some vertex $x\notin \{u,v\}$, which is visited by both $P_{1}$ and $P_{2}$. Let $x$ be the last such vertex on $P_{1}$, and let $P_{1}=P_{1}^{\prime }xP_{1}^{\prime \prime }$ and $P_{2}=P_{2}^{\prime}xP_{2}^{\prime \prime }$. No vertex appears twice in the closed subwalk $P_{1}^{\prime \prime }P_{2}^{\prime }$, so $P_{1}^{\prime \prime}P_{2}^{\prime }$ is a potential cycle of $O$; if $e$ appears on this potential cycle the claim is satisfied. If not, then $e$ appears on the closed subwalk $P_{1}^{\prime }P_{2}^{\prime \prime }$. As $P_{1}^{\prime }$ is a potential path of $O$ from~$v$ to~$x$, $P_{2}^{\prime \prime }$ is a potential path of $O$ from~$x$ to~$v$, and their lengths are strictly less than those of $P_{1}$ and $P_{2}$ (respectively), the inductive hypothesis applies.
\end{proof}

We should emphasize that the conclusion of Lemma~\ref{lem:preliminary} does not hold for bioriented edges. For example, if $e$ is a bioriented edge between $u$ and $v$ then~$P_{1}=\{e\}$ and~$P_{2}=\{e\}$\ satisfy the hypothesis of the lemma, but $P_{1}P_{2}$ does not contain a potential cycle, because it does not contain a circuit.

Now we proceed to work out the consequences of the negation of each of the assertions that make up condition~(2) of Lemma~\ref{lem:key}. Again, the setup is always that the maximum edge~$e$ of~$G$ is neither an isthmus nor a loop and is oriented from~$u$ to~$v$ in~$O$.

\subsection{Consequences of \texorpdfstring{$L^{o}(O)\neq L^{o}(O-e)$}{}}

\begin{lemma} $L^{o}(O)\neq L^{o}(O-e)\Rightarrow L^{o}(^{e}O)=L^{o}(O-e)$.
\end{lemma}

\begin{proof}
Every potential cycle of $O-e$ is also a potential cycle of both $O$ and $^{e}O$, so $L^{o}(O-e)\subseteq L^{o}(O)\cap L^{o}(^{e}O)$. Consequently if the lemma fails then $G$ has edges $e_{i}\in L^{o}(O)$ and $e_{j}\in L^{o}(^{e}O)$, neither of which is included in $L^{o}(O-e)$. Interchanging $O$ and $^{e}O$ if necessary, we may presume that $i\leq j$.

There are then a potential cycle $e_{i_{1}},\ldots,e_{i_{s}}$ of $O$ and a potential cycle $e_{j_{1}}^{\prime},\ldots,e_{j_{t}}^{\prime }$ of $^{e}O$, with $i=\min \{i_{1},\ldots,i_{s}\}$ and $j=\min \{j_{1},\ldots,j_{t}\}$. As neither $e_{i}$ nor $e_{j}$ is included in $L^{o}(O-e)$, both potential cycles include $e$. We may presume that $e_{i_{1}}=e_{j_{1}}^{\prime }=e$. Then $P_{1}=e_{i_{2}},\ldots,e_{i_{s}}$ is a potential path of $O-e$ from $v$ to $u$ and $P_{2}=e_{j_{2}}^{\prime},\ldots,e_{j_{t}}^{\prime }$ is a potential path of $O-e$ from $u$ to $v$. Lemma~\ref{lem:preliminary} tells us that there is a potential cycle $Cy$ of $O-e$, such that $e_{i}\in Cy\subseteq P_{1}P_{2}$. But this is impossible, as $e_{i}\notin L^{o}(O-e)$.
\end{proof}

\begin{lemma}\label{cyclesb}
$L^{o}(O)\neq L^{o}(O-e)\Rightarrow L^{b}(^{e}O)=L^{b}(O-e)$.
\end{lemma}

\begin{proof}
If $L^{o}(O)\neq L^{o}(O-e)$ then $G$ has\ an edge $e_{i}\in L^{o}(O)$, which is excluded from $L^{o}(O-e)$. That is, there is a potential cycle $e_{i_{1}},\ldots,e_{i_{s}}$ of $O$ with $i=\min \{i_{1},\ldots,i_{s}\}$, and every such potential cycle includes $e$. If the lemma fails then $G$ has an edge $e_{j}\in L^{b}(^{e}O)$, which is excluded from $L^{b}(O-e)$. That is, if $^{e}O^{\prime }=$ $^{e}O\setminus \{e_{j}^{\sigma _{b}(e_{j})}\}$ then there is a potential cycle $e_{j_{1}}^{\prime },\ldots,e_{j_{t}}^{\prime }$ of $^{e}O^{\prime }$ with $j=\min \{j_{1},\ldots,j_{t}\}$, and every such potential cycle includes $e$. Notice that $e_{j_{1}}^{\prime },\ldots,e_{j_{t}}^{\prime }$ is also a potential cycle of~$^{e}O$.

After cyclic permutations of the two cycles we may presume that $e_{i_{1}}=e_{j_{1}}^{\prime }=e$. Then $P_{1}=e_{i_{2}},\ldots,e_{i_{s}}$ is a potential path of $O-e$ from $v$ to $u$ and $P_{2}=e_{j_{2}}^{\prime},\ldots,e_{j_{t}}^{\prime }$ is a potential path of $O-e$ from $u$ to $v$. Lemma \ref{lem:preliminary} tells us that there is a potential cycle $Cy$ of $O-e$, such that $e_{i}\in Cy\subseteq P_{1}P_{2}$. As $e_{i}\notin L^{o}(O-e)$, $e_{i} $ cannot be the least-indexed edge of $Cy$; necessarily then $j<i$.

Lemma \ref{lem:preliminary} also tells us that there is a potential cycle $Cy^{\prime }$ of $O-e$, such that $e_{j}\in Cy^{\prime }\subseteq P_{1}P_{2}$. Moreover, the direction of $e_{j}$ that is consistent with the preferred orientation of $Cy^{\prime }$ is the same as the direction of $e_{j}$ that is consistent with the preferred orientation of $P_{2}$. (As $j<i$, $e_{j}$ does not appear on $P_{1}$.) Consequently $Cy^{\prime }$ is a potential cycle of $^{e}O^{\prime }-e$. As $e\notin Cy^{\prime }$, this contradicts the hypothesis that $e_{j}$ is excluded from $L^{b}(O-e)$.
\end{proof}

\begin{lemma}
\label{cuts}
$L^{o}(O)\neq L^{o}(O-e)\Rightarrow I^{o}(^{e}O)=I^{o}(O-e)$.
\end{lemma}

\begin{proof}
If the lemma fails then $I^{o}(^{e}O)\not=I^{o}(O-e)$. Each potential cut $Cu$ of $^{e}O$ has the property that $Cu\setminus \{e\}$ is a union of disjoint potential cuts of $O-e$, so $I^{o}(^{e}O)\subseteq I^{o}(O-e)$. It follows that $G$ has\ an edge $e_{j}\in I^{o}(O-e)$, which is excluded from $I^{o}(^{e}O)$. Then there is a potential cut $Cu$ of $O-e$ whose least-indexed element is $e_{j}$, and no such $Cu$ has either $Cu$ or $Cu\cup \{e\}$ a potential cut of~$^{e}O$.

Let $Cu$ be a potential cut of $O-e$, which includes $e_{j}$. That is, there is a subset $U\subseteq V(G)$ such that every edge of $G-e$ with precisely one end-vertex in $U$ is either unoriented in $O$ or directed away from $U$, and $Cu$ is the set of edges of $G-e$ with precisely one end-vertex in $U$. As neither $Cu$ nor $Cu\cup \{e\}$ is a potential cut of $^{e}O$, and $e$ is oriented from $v$ to $u$ in $^{e}O$, it follows that $u\in U$ and $v\notin U$.

If $L^{o}(O)\neq L^{o}(O-e)$ then $G$ has an edge $e_{i}\in L^{o}(O)$, such that $e$ is included in every potential cycle $e_{i_{1}},\ldots,e_{i_{s}}$ of $O$ with $i=\min \{i_{1},\ldots,i_{s}\}$. Suppose $e$ is included in a potential cycle $e_{i_{1}},\ldots,e_{i_{s}}$ of $O$. We may presume that $e=e_{i_{1}}$; then $e_{i_{2}},\ldots,e_{i_{s}}$ is a potential path of $O-e$ from $v$ to $u$. As $v\notin U$ and $u\in U$, some edge of this path must have precisely one end-vertex in $U$, and be either bioriented or directed into $U$. The existence of such an edge contradicts the description of $Cu$ in the preceding paragraph.
\end{proof}

\begin{lemma}
\label{cutsb}
$L^{o}(O)\neq L^{o}(O-e)\Rightarrow I^{u}(^{e}O)=I^{u}(O-e)$.
\end{lemma}

\begin{proof}
If $j<p$ then each potential cut $Cu$ of $^{e}O\cup \{e_{j}^{\sigma_{u}(e_{j})}\}$ has the property that $Cu\setminus \{e\}$ is a union of disjoint potential cuts of $(O-e)\cup \{e_{j}^{\sigma _{u}(e_{j})}\}$, so $I^{u}(^{e}O)\subseteq I^{u}(O-e)$. Consequently, if the lemma fails then $G $ has\ an edge $e_{j}\in I^{u}(O-e)$, which is excluded from $I^{u}(^{e}O)$. Then there is a potential cut $Cu$ of $(O-e)\cup\{e_{j}^{\sigma _{u}(e_{j})}\}$ whose least-indexed element is $e_{j}$, and no such $Cu$ has either $Cu$ or $Cu\cup \{e\}$ a potential cut of $^{e}O\cup\{e_{j}^{\sigma _{u}(e_{j})}\}$.

There is a subset $U\subseteq V(G)$ such that $Cu$ is the set of edges of $G-e$ with precisely one end-vertex in $U$, and every element of $Cu$ is either unoriented in $(O-e)\cup \{e_{j}^{\sigma _{u}(e_{j})}\}$ or directed away from $U$.\ As neither $Cu$ nor $Cu\cup \{e\}$ is a potential cut of $^{e}O\cup \{e_{j}^{\sigma _{u}(e_{j})}\}$, and $e$ is directed from $v$ to $u $ in $^{e}O$, it follows that $u\in U$ and $v\notin U$.

Recall that, as in Lemma~\ref{cuts}, $e$ is included in a potential cycle $e_{i_{1}},\ldots,e_{i_{s}}$ of $O$. We may presume that $e=e_{i_{1}}$; then $e_{i_{2}},\ldots,e_{i_{s}}$ is a potential path of $O-e$ from $v$ to $u$. As $v\notin U$ and $u\in U$, some edge of this path must have precisely one end-vertex in $U$, and be either bioriented or directed toward this vertex. But no such edge exists, according to the description of $Cu$ in the preceding paragraph.
\end{proof}

\begin{lemma}
\label{cyclesc}
$L^{o}(O)\neq L^{o}(O-e)\Rightarrow L^{o}(O)=L^{o}(O/e)$.
\end{lemma}

\begin{proof}
Every potential cycle $Cy$ of $O$ has the property that $Cy\setminus \{e\}$ is a union of disjoint potential cycles of $O/e$, so $L^{o}(O)\subseteq L^{o}(O/e)$. If $L^{o}(O)\not=L^{o}(O/e)$ then $L^{o}(O/e)$ has an element $e_{j}$ that is not included in $L^{o}(O)$. Then there is a potential cycle $Cy^{\prime }$ of $O/e$, such that $j$ is the smallest index appearing in $Cy^{\prime }$ and neither $Cy^{\prime }$ nor $Cy^{\prime }\cup \{e\}$ is a potential cycle of $O$. It follows that $Cy^{\prime }$ yields a cycle $Cy^{\prime \prime }$ of $G$, of the form $e=e_{j_{1}}^{\prime},e_{j_{2}}^{\prime },\ldots,e_{j_{t}}^{\prime }$, such that $e_{j_{2}}^{\prime},\ldots,e_{j_{t}}^{\prime }$ is an $O$-path but $Cy^{\prime \prime }$ is not a potential cycle of~$O$.

As we have seen in other lemmas, there must be a potential cycle $e_{i_{1}},\ldots,e_{i_{s}}$ of $O$ that contains $e$. Permuting indices if necessary we may presume that $e=e_{i_{1}}$, so that $e_{i_{2}},\ldots,e_{i_{s}}$ is a potential path of $O$ from $v$ to $u$. Both $e_{i}$ and $e_{j}$ are oriented in $O$, so Lemma \ref{lem:preliminary} tells us that the closed walk $e_{j_{2}}^{\prime},\ldots,e_{j_{t}}^{\prime },e_{i_{2}},\ldots,e_{i_{s}}$ contains a potential cycle of $O$ that contains $e_{i}$, and also contains a potential cycle of $O $ that contains $e_{j}$. (They may be the same potential cycle.) If $i\leq j$ then we have a contradiction, as the potential cycle that includes $e_{i}$ does not also include $e$. We conclude that $j<i$; but this also yields a contradiction, as the existence of a potential cycle of $O$ that includes $e_{j}$ would imply that $e_{j}\in L^{o}(O)$.
\end{proof}

\begin{lemma}
\label{cyclescb}
$L^{o}(O)\neq L^{o}(O-e)\Rightarrow L^{b}(O)=L^{b}(O/e)$.
\end{lemma}

\begin{proof}
The proof is only a little more complicated than the corresponding argument for Lemma~\ref{cyclesc}. If $j<p$ then each potential cycle $Cy$ of $O\setminus \{e_{j}^{\sigma _{b}(e_{j})}\}$ has the property
that $Cy\setminus \{e\}$ is a union of disjoint potential cycles of $(O\setminus \{e_{j}^{\sigma _{b}(e_{j})}\})/e$, so $L^{b}(O)\subseteq L^{b}(O/e)$. If $L^{b}(O)\neq L^{b}(O/e)$ then $L^{b}(O/e)$ has an element $e_{j}\notin L^{b}(O)$. Then there is a potential cycle $Cy^{\prime} $ of $(O\setminus \{e_{j}^{\sigma _{b}(e_{j})}\})/e$ such that $j$ is the smallest index appearing in $Cy^{\prime }$ and neither $Cy^{\prime }$ nor $Cy^{\prime }\cup \{e\}$ is a potential cycle of $O\setminus \{e_{j}^{\sigma_{b}(e_{j})}\}$. Consequently there is a cycle $Cy^{\prime \prime }$ of $G$, of the form $e=e_{j_{1}}^{\prime },e_{j_{2}}^{\prime },\ldots,e_{j_{t}}^{\prime}$, such that $e_{j_{2}}^{\prime },\ldots,e_{j_{t}}^{\prime }$ is a potential path of $O\setminus \{e_{j}^{\sigma _{b}(e_{j})}\}$ but $Cy^{\prime \prime }$ is not a potential cycle of $O\setminus \{e_{j}^{\sigma _{b}(e_{j})}\}$. Uay $P_{1}=e_{j_{2}}^{\prime},\ldots,e_{j_{t}}^{\prime }$ is a potential path of $O\setminus \{e_{j}^{\sigma _{b}(e_{j})}\}$ from $u$ to $v$. Of course $P_{1}$ is also a potential path of~$O$.

Let $e_{i_{1}},\ldots,e_{i_{s}}$ be a potential cycle of $O$ which includes $e$, which as we have seen in other lemmas must exist. Permuting indices if necessary we may presume that $e=e_{i_{1}}$, so that $P_{2}=e_{i_{2}},\ldots,e_{i_{s}}$ is a potential path of $O$ from $v$ to $u$. Lemma \ref{lem:preliminary} tells us that the concatenation $P_{1}P_{2}$ contains a potential cycle $Cy_{1}$ of $O$ that contains $e_{i}$, and also a potential cycle $Cy_{2}$ of $O$ that contains $e_{j}$. (It is possible that $Cy_{1}=Cy_{2}$.) If $i\leq j$ the existence of $Cy_{1}$ contradicts the fact that $e_{i}\notin L^{o}(O-e)$, so it must be that $j<i$. But then $e_{j}$ does not appear in $P_{2}$, so Lemma \ref{lem:preliminary} tells us that the direction of $e_{j}$ consistent with the preferred orientation of $Cy_{2}$ is the same as the direction of $e_{j}$ in $P_{1}$. It follows that $Cy_{2}$ is a potential cycle of $O\setminus \{e_{j}^{\sigma _{b}(e_{j})}\}$, contradicting the hypothesis that $e_{j}\notin L^{b}(O)$.
\end{proof}

\begin{lemma}
$L^{o}(O)\neq L^{o}(O-e)\Rightarrow I^{o}(O)=I^{o}(O/e)$.
\end{lemma}

\begin{proof}
Again, if $L^{o}(O)\neq L^{o}(O-e)$ then $G$ has\ an edge $e_{i}\in L^{o}(O)$, such that\ $e$ is included in every potential cycle $e_{i_{1}},\ldots,e_{i_{s}}$ of $O$ with $i=\min \{i_{1},\ldots,i_{s}\}$. Each potential cut of $O/e$ is a potential cut of $O$, so $I^{o}(O/e)\subseteq I^{o}(O)$. If $I^{o}(O)\not=I^{o}(O/e)$ then there is an $e_{j}\in I^{o}(O)\setminus I^{o}(O/e)$. Then every potential cut of $O$ with minimum element $e_{j}$ also contains $e$. As $e$ cannot appear in both a potential cycle of $O$ and a potential cut of $O$, the preceding sentence contradicts the first sentence of the proof.
\end{proof}

\begin{lemma}
$L^{o}(O)\neq L^{o}(O-e)\Rightarrow I^{u}(O)=I^{u}(O/e)$.
\end{lemma}

\begin{proof}
Again, if $L^{o}(O)\neq L^{o}(O-e)$ then $G$ has\ an edge $e_{i}\in L^{o}(O)$, such that\ $e$ is included in every potential cycle $e_{i_{1}},\ldots,e_{i_{s}}$ of $O$ with $i=\min \{i_{1},\ldots,i_{s}\}$. If $j<p$ then each potential cut of $(O\cup \{e_{j}^{\sigma _{u}(e_{j})}\})/e$ is a potential cut of $O\cup \{e_{j}^{\sigma _{u}(e_{j})}\}$, so $I^{u}(O/e)\subseteq I^{u}(O)$. If $I^{u}(O)\not=I^{u}(O/e)$ then there is an $e_{j}\in I^{u}(O)\setminus I^{u}(O/e)$. Then every potential cut of $O\cup \{e_{j}^{\sigma _{u}(e_{j})}\}$ with minimum element $e_{j}$ also contains $e$.

According to the first sentence of the proof, $e$ appears in a potential cycle of $O$; of course a potential cycle of $O$ is also a potential cycle of $O\cup \{e_{j}^{\sigma_{u}(e_{j})}\}$. As no edge of $G$ can appear in both a potential cut of $O\cup \{e_{j}^{\sigma_{u}(e_{j})}\}$ and a potential cycle of $O\cup \{e_{j}^{\sigma_{u}(e_{j})}\}$, we have a contradiction.
\end{proof}

\subsection{Consequences of \texorpdfstring{$L^{b}(O)\neq L^{b}(O-e)$}{}}

\begin{lemma}
\label{cyclesb0}
$L^{b}(O)\neq L^{b}(O-e)\Rightarrow L^{o}(^{e}O)=L^{o}(O-e)$.
\end{lemma}

\begin{proof}
This is the contrapositive of Lemma \ref{cyclesb}, applied to $^{e}O$ rather than $O$.
\end{proof}

\begin{lemma}
\label{cyclesb1}
$L^{b}(O)\neq L^{b}(O-e)\Rightarrow L^{b}(^{e}O)=L^{b}(O-e)$.
\end{lemma}

\begin{proof}
If $L^{b}(O)\neq L^{b}(O-e)$ then $G$ has\ an edge $e_{i}\in L^{b}(O)$, which is excluded from $L^{b}(O-e)$. That is, there is a potential cycle $e_{i_{1}},\ldots,e_{i_{s}}$ of $O\setminus \{e_{i}^{\sigma _{b}(e_{i})}\}$ with  $i=\min \{i_{1},\ldots,i_{s}\}$, and every such potential cycle includes $e$. If the lemma fails then $G$ has\ an edge $e_{j}\in L^{b}(^{e}O)$, which is excluded from $L^{b}(O-e)$. That is, there is a potential cycle $e_{j_{1}}^{\prime },\ldots,e_{j_{t}}^{\prime }$ of $^{e}O\setminus \{e_{j}^{\sigma _{b}(e_{j})}\}$ with $j=\min \{j_{1},\ldots,j_{t}\}$, and every such potential cycle includes $e$.

Interchanging $O$ and $^{e}O$ if necessary, we may presume that $i\leq j$. Also, after cyclic permutations of the two cycles we may presume that $e_{i_{1}}=e_{j_{1}}^{\prime }=e$. Then $P_{1}=e_{i_{2}},\ldots,e_{i_{s}}$ is a potential path of $(O\setminus \{e_{i}^{\sigma _{b}(e_{i})}\})-e$ from $v$ to $u$ and $P_{2}=e_{j_{2}}^{\prime },\ldots,e_{j_{t}}^{\prime }$ is a potential path of $(O\setminus \{e_{j}^{\sigma _{b}(e_{j})}\})-e$ from $u$ to $v$. If $i=j$, then of course $\sigma _{b}(e_{i})=\sigma _{b}(e_{j})$; and if $i<j$ then $e_{j}$ is bioriented in $(O\setminus \{e_{i}^{\sigma_{b}(e_{i})}\})$. Either way, we conclude that $P_{1}$ and $P_{2}$ are both potential paths of $(O\setminus \{e_{i}^{\sigma _{b}(e_{i})}\})-e$, so Lemma~\ref{lem:preliminary} tells us that there is a potential cycle $Cy$ of $(O\setminus \{e_{i}^{\sigma _{b}(e_{i})}\})-e$, such that $e_{i}\in Cy\subseteq P_{1}P_{2}$. The existence of $Cy$ contradicts the hypothesis that $e_{i}\notin L^{b}(O-e)$.
\end{proof}

\begin{lemma}
\label{cyclesb2}
$L^{b}(O)\neq L^{b}(O-e)\Rightarrow L^{o}(O)=L^{o}(O/e)$.
\end{lemma}

\begin{proof}
Every potential cycle $Cy$ of $O$ has the property that $Cy\setminus \{e\}$ is a union of disjoint potential cycles of $O/e$, so $L^{o}(O)\subseteq L^{o}(O/e)$. If $L^{o}(O)\not=L^{o}(O/e)$ then $L^{o}(O/e)$ has an element $e_{j}$ that is not included in $L^{o}(O)$. Then there is a potential cycle $Cy^{\prime }$ of $O/e$, such that $j$ is the smallest index appearing in $Cy^{\prime }$ and neither $Cy^{\prime }$ nor $Cy^{\prime }\cup \{e\}$ is a potential cycle of $O$. It follows that $Cy^{\prime }$ yields a cycle $Cy^{\prime \prime }$ of $G$, of the form $e=e_{j_{1}}^{\prime},e_{j_{2}}^{\prime },\ldots,e_{j_{t}}^{\prime }$, such that $e_{j_{2}}^{\prime},\ldots,e_{j_{t}}^{\prime }$ is a potential path of $O$ but $Cy^{\prime \prime}$ is not a potential cycle of $O$. Then $e_{j_{2}}^{\prime },\ldots,e_{j_{t}}^{\prime }$ is a potential
path of $O$ from $u$ to $v$.

If $L^{b}(O)\neq L^{b}(O-e)$ then $G$ has\ an edge $e_{i}\in L^{b}(O)$, such that\ $e$ is included in every potential cycle $e_{i_{1}},\ldots,e_{i_{s}}$ of $O\setminus \{e_{i}^{\sigma _{b}(e_{i})}\}$ with $i=\min \{i_{1},\ldots,i_{s}\}$. Consider a potential cycle $e_{i_{1}},\ldots,e_{i_{s}}$ of $O$ that contains $e$. Permuting indices if necessary we may presume that $e=e_{i_{1}}$, so that $e_{i_{2}},\ldots,e_{i_{s}}$ is an $O$-path from $v$ to $u$. Both $e_{i}$ and $e_{j}$ are oriented in $O$, so Lemma~\ref{lem:preliminary} tells us that the closed walk $e_{j_{2}}^{\prime},\ldots,e_{j_{t}}^{\prime },e_{i_{2}},\ldots,e_{i_{s}}$ contains a potential cycle of $O$ that contains $e_{i}$, and also contains a potential cycle of $O $ that contains $e_{j}$. (They may be the same potential cycle.) If $i<j$ then Lemma~\ref{lem:preliminary} tells us that the direction of $e_{i}$ that is consistent with the preferred orientation of the potential cycle that includes $e_{i}$ is the same as the direction of $e_{i}$ in the walk $e_{i_{2}},\ldots,e_{i_{s}}$. But then this potential cycle is a potential cycle of $O\setminus \{e_{i}^{\sigma _{b}(e_{i})}\}$, an impossibility as the cycle excludes $e$. Of course $i=j$ is impossible, as $e_{i}$ is bioriented in $O$ and $e_{j}$ is oriented in $O$. We conclude that $j<i$; but this also yields a contradiction, as the existence of a potential cycle of $O$ that includes $e_{j}$ would imply that $e_{j}\in L^{o}(O)$.
\end{proof}

\begin{lemma}
$L^{b}(O)\neq L^{b}(O-e)\Rightarrow L^{b}(O)=L^{b}(O/e)$.
\end{lemma}

\begin{proof}

If $j<p$ then each potential cycle $Cy$ of $O\setminus \{e_{j}^{\sigma_{b}(e_{j})}\}$ has the property that $Cy\setminus \{e\}$ is a union of disjoint potential cycles of $(O\setminus \{e_{j}^{\sigma _{b}(e_{j})}\})/e$, so $L^{b}(O)\subseteq L^{b}(O/e)$. If $L^{b}(O)\neq L^{b}(O/e)$ then $L^{b}(O/e)$ has an element $e_{j}\notin L^{b}(O)$. Then there is a potential cycle $Cy^{\prime }$ of $(O\setminus \{e_{j}^{\sigma_{b}(e_{j})}\})/e$ such that $j$ is the smallest index appearing in $Cy^{\prime }$ and neither $Cy^{\prime }$ nor $Cy^{\prime }\cup \{e\}$ is a potential cycle of $O\setminus \{e_{j}^{\sigma _{b}(e_{j})}\}$. Consequently there is a cycle $Cy^{\prime \prime }$ of $G$, of the form $e=e_{j_{1}}^{\prime },e_{j_{2}}^{\prime },\ldots,e_{j_{t}}^{\prime }$, such that $e_{j_{2}}^{\prime },\ldots,e_{j_{t}}^{\prime }$ is a potential path of $O\setminus \{e_{j}^{\sigma _{b}(e_{j})}\}$ but $Cy^{\prime \prime }$ is not a potential cycle of $O\setminus \{e_{j}^{\sigma _{b}(e_{j})}\}$. Then $P_{1}=e_{j_{2}}^{\prime},\ldots,e_{j_{t}}^{\prime }$ is a potential path of $O\setminus \{e_{j}^{\sigma _{b}(e_{j})}\}$ from $u$ to $v$. Of course $P_{1}$ is also a potential path of~$O$.

If $L^{b}(O)\neq L^{b}(O-e)$ then $G$ has\ an edge $e_{i}\in L^{b}(O)$, such that\ $e$ is included in every potential cycle $e_{i_{1}},\ldots,e_{i_{s}}$ of $O\setminus \{e_{i}^{\sigma _{b}(e_{i})}\}$ with $i=\min\{i_{1},\ldots,i_{s}\}$. Let $e_{i_{1}},\ldots,e_{i_{s}}$ be a potential cycle of $O$ which includes $e$. Permuting indices if necessary we may presume that $e=e_{i_{1}}$, so that $P_{2}=e_{i_{2}},\ldots,e_{i_{s}}$ is a potential path of $O$ from $v$ to $u$. Lemma \ref{lem:preliminary} tells us that the concatenation $P_{1}P_{2}$ contains a potential cycle $Cy_{1}$ of $O$ that contains $e_{i}$, and also a potential cycle $Cy_{2}$ of $O$ that contains $e_{j}$. (It is possible that $Cy_{1}=Cy_{2}$.) If $i<j$ then Lemma \ref{lem:preliminary} tells us that the direction of $e_{i}$ in $P_{2}$ is consistent with the preferred orientation of $Cy_{1}$, i.e., $Cy_{1}$ is a potential cycle of $O\setminus\{e_{i}^{\sigma _{b}(e_{i})}\}$. As $e\notin Cy_{1}$, this contradicts the first sentence of the proof. If $i=j$ we derive the same contradiction after noting that $e_{i}=e_{j}$ has the same direction in $P_{1}$ and $P_{2}$, namely the direction different from $\sigma _{b}(e_{i})$. If $i>j$ then Lemma \ref{lem:preliminary} tells us that the direction of $e_{j}$ in $P_{2}$\ is consistent with the preferred orientation of $Cy_{2}$. It follows that $Cy_{2}$ is a potential cycle of $O\setminus \{e_{j}^{\sigma _{b}(e_{j})}\}$, contradicting the hypothesis that $e_{j}\notin L^{b}(O)$.
\end{proof}

\begin{lemma}
$L^{b}(O)\neq L^{b}(O-e)\Rightarrow I^{o}(^{e}O)=I^{o}(O-e)$.
\end{lemma}

\begin{proof}
 The proof is the same argument that was given in the proof of Lemma~\ref{cuts}; we repeat it for the reader's convenience.

If the lemma fails then $I^{o}(^{e}O)\not=I^{o}(O-e)$. Each potential cut $Cu$ of $^{e}O$ yields a potential cut $Cu\setminus \{e\}$ of $O-e$, so $I^{o}(^{e}O)\subseteq I^{o}(O-e)$. It follows that $G$ has\ an edge $e_{j}\in I^{o}(O-e)$, which is excluded from $I^{o}(^{e}O)$. Then there is a potential cut $Cu$ of $O-e$ whose least-indexed element is $e_{j}$, and no such $Cu$ has either $Cu$ or $Cu\cup \{e\}$ a potential cut of $^{e}O$. 

Let $Cu$ be a potential cut of $O-e$, which includes $e_{j}$. That is, there is a subset $U\subseteq V(G)$ such that every edge of $G-e$ with precisely one end-vertex in $U$ is either unoriented in $O$ or directed away from $U$, and $Cu$ is the set of edges of $G-e$ with precisely one end-vertex in $U$. As neither $Cu$ nor $Cu\cup \{e\}$ is a potential cut of $^{e}O$, and $e$ is oriented from $v$ to $u$ in $^{e}O$, it follows that $u\in U$ and $v\notin U$.

If $L^{b}(O)\neq L^{b}(O-e)$ then $G$ has\ an edge $e_{i}\in L^{b}(O)$, such that\ $e$ is included in every potential cycle $e_{i_{1}},\ldots,e_{i_{s}}$ of $O\setminus \{e_{i}^{\sigma _{b}(e_{i})}\}$ with $i=\min\{i_{1},\ldots,i_{s}\}$. Suppose $e$ is included in a potential cycle $e_{i_{1}},\ldots,e_{i_{s}}$ of $O$. We may presume that $e=e_{i_{1}}$; then $e_{i_{2}},\ldots,e_{i_{s}}$ is a potential path of $O-e$ from $v$ to $u$. As $v\notin U$ and $u\in U$, some edge of this path must be bioriented or directed into $U$. The existence of such an edge contradicts the description of $Cu$ in the preceding paragraph.
\end{proof}

\begin{lemma}
$L^{b}(O)\neq L^{b}(O-e)\Rightarrow I^{u}(^{e}O)=I^{u}(O-e)$.
\end{lemma}

\begin{proof}
The argument follows the proof of Lemma~\ref{cutsb} closely.

If $j<p$ then each potential cut $Cu$ of $^{e}O\cup \{e_{j}^{\sigma_{u}(e_{j})}\}$ yields a potential cut $Cu\setminus \{e\}$ of $(O-e)\cup \{e_{j}^{\sigma _{u}(e_{j})}\}$, so $I^{u}(^{e}O)\subseteq I^{u}(O-e)$. Consequently, if the lemma fails then $G$ has\ an edge $e_{j}\in I^{u}(O-e) $, which is excluded from $I^{u}(^{e}O)$. Then there is a potential cut $Cu$ of $(O-e)\cup \{e_{j}^{\sigma _{u}(e_{j})}\}$ whose least-indexed element is $e_{j}$, and no such $Cu$ has either $Cu$ or $Cu\cup \{e\}$ a potential cut of $^{e}O\cup \{e_{j}^{\sigma _{u}(e_{j})}\}$.

There is a subset $U\subseteq V(G)$ such that $Cu$ is the set of edges of $G-e$ with precisely one end-vertex in $U$, and every element of $Cu$ is either unoriented in $(O-e)\cup \{e_{j}^{\sigma _{u}(e_{j})}\}$ or directed away from $U$.\ As neither $Cu$ nor $Cu\cup \{e\}$ is a potential cut of $^{e}O\cup \{e_{j}^{\sigma _{u}(e_{j})}\}$, and $e$ is directed from $v$ to $u $ in $^{e}O$, it follows that $u\in U$ and $v\notin U$.

Recall that $e$ is included in a potential cycle $e_{i_{1}},\ldots,e_{i_{s}}$ of $O$. We may presume that $e=e_{i_{1}}$; then $e_{i_{2}},\ldots,e_{i_{s}}$ is a potential path of $O-e$ from $v$ to $u$. As $v\notin U$ and $u\in U$, some edge of this path must have precisely one end-vertex in $U$, and be either bioriented or directed toward this vertex. But no such edge exists, according to the description of $Cu$ in the preceding paragraph.
\end{proof}

\begin{lemma}
$L^{b}(O)\neq L^{b}(O-e)\Rightarrow I^{o}(O)=I^{o}(O/e)$.
\end{lemma}

\begin{proof}
Again, if $L^{b}(O)\neq L^{b}(O-e)$ then $G$ has an edge $e_{i}\in L^{b}(O) $, such that $e$ is included in every potential cycle $e_{i_{1}},\ldots,e_{i_{s}}$ of $O\setminus \{e_{i}^{\sigma _{b}(e_{i})}\}$ with $i=\min \{i_{1},\ldots,i_{s}\}$. A potential cycle of $O\setminus\{e_{i}^{\sigma _{b}(e_{i})}\}$ is also a potential cycle of $O$, of course.

Each potential cut of $O/e$ is a potential cut of $O$, so $I^{o}(O/e)\subseteq I^{o}(O)$. If $I^{o}(O)\not=I^{o}(O/e)$ then there is an $e_{j}\in I^{o}(O)\setminus I^{o}(O/e)$. Then every potential cut of $O$ with minimum element $e_{j}$ also contains $e$. As $e$ cannot appear in both a potential cycle of $O$ and a potential cut of $O$, we have a contradiction.
\end{proof}

\begin{lemma}
$L^{b}(O)\neq L^{b}(O-e)\Rightarrow I^{u}(O)=I^{u}(O/e)$.
\end{lemma}

\begin{proof}
Again, if $L^{b}(O)\neq L^{b}(O-e)$ then $G$ has an edge $e_{i}\in L^{b}(O) $, such that $e$ is included in every potential cycle $e_{i_{1}},\ldots,e_{i_{s}}$ of $O\setminus \{e_{i}^{\sigma _{b}(e_{i})}\}$ with $i=\min \{i_{1},\ldots,i_{s}\}$. A potential cycle of $O\setminus\{e_{i}^{\sigma _{b}(e_{i})}\}$ is also a potential cycle of $O$, of course.

If $j<p$ then each potential cut of $(O\cup \{e_{j}^{\sigma_{u}(e_{j})}\})/e $ is a potential cut of $O\cup \{e_{j}^{\sigma_{u}(e_{j})}\}$, so $I^{u}(O/e)\subseteq I^{u}(O)$. If $I^{u}(O)\not=I^{u}(O/e)$ then there is an $e_{j}\in I^{u}(O)\setminus I^{u}(O/e)$. Then $e$ is contained in every potential cut of $O\cup \{e_{j}^{\sigma _{u}(e_{j})}\}$ with minimum element $e_{j}$. According to the first sentence of the proof, $e$ appears in a potential cycle of $O\setminus \{e_{i}^{\sigma _{b}(e_{i})}\}$. Every potential cycle of $O\setminus \{e_{i}^{\sigma _{b}(e_{i})}\}$ is also a potential cycle of $O\cup \{e_{j}^{\sigma _{u}(e_{j})}\}$, no matter how $i$ and $j$ are related. As no edge of $G$ can appear in both a potential cut of $O\cup\{e_{j}^{\sigma _{u}(e_{j})}\}$ and a potential cycle of $O\cup\{e_{j}^{\sigma _{u}(e_{j})}\}$, we have a contradiction.
\end{proof}

\subsection{Consequences of \texorpdfstring{$L^{o}(O)\neq L^{o}(O/e)$}{}}

\begin{lemma}
\label{cyclesc0}
$L^{o}(O)\neq L^{o}(O/e)\Rightarrow L^{o}(O)=L^{o}(O-e)$ and~$L^{b}(O)=L^{b}(O-e)$.
\end{lemma}

\begin{proof}
These follow from\ Lemmas~\ref{cyclesc} and~\ref{cyclesb2}, applied to $^{e}O$ rather than~$O$.
\end{proof}

\begin{lemma}
$L^{o}(O)\neq L^{o}(O/e)\Rightarrow L^{o}(^{e}O)=L^{o}(O/e)$.
\end{lemma}

\begin{proof}
Every potential cycle $Cy$ of $O$ has the property that $Cy\setminus \{e\}$ is a union of disjoint potential cycles of $O/e$, so $L^{o}(O)\subseteq L^{o}(O/e)$. If $L^{o}(O)\not=L^{o}(O/e)$ then $L^{o}(O/e)$ has an element $e_{i}$ that is not included in $L^{o}(O)$. Then there is a potential cycle $Cy$ of $O/e$, such that $i$ is the smallest index appearing in $Cy$ and neither $Cy$ nor $Cy\cup \{e\}$ is a potential cycle of $O$. It follows that $Cy$ yields a cycle of $G$, of the form $e=e_{i_{1}},e_{i_{2}},\ldots,e_{i_{s}}$, such that $e_{i_{2}},\ldots,e_{i_{s}}$ is a potential path of~$O$ but $e_{i_{1}},e_{i_{2}},\ldots,e_{i_{s}}$ is not a potential cycle of~$O$.

Similarly, if $L^{o}(^{e}O)\not=L^{o}(O/e)$ then $L^{o}(O/e)$ has an element $e_{j}$ that is not included in $L^{o}(^{e}O)$. Then there is a potential cycle $Cy^{\prime }$ of $O/e$, such that $j$ is the smallest index appearing in $Cy^{\prime }$ and neither $Cy^{\prime }$ nor $Cy^{\prime }\cup \{e\}$ is a potential cycle of $^{e}O$. It follows that $Cy^{\prime }$ yields a cycle of $G$, of the form $e=e_{j_{1}}^{\prime },e_{j_{2}}^{\prime},\ldots,e_{j_{t}}^{\prime }$, such that $e_{j_{2}}^{\prime},\ldots,e_{j_{t}}^{\prime }$ is a potential path of $^{e}O$ but $e_{j_{1}}^{\prime },e_{j_{2}}^{\prime },\ldots,e_{j_{t}}^{\prime }$ is not a potential cycle of $^{e}O$. Then $e_{j_{2}}^{\prime },\ldots,e_{j_{t}}^{\prime }$ is a potential path of $^{e}O$ from $v$ to $u$, and $e_{i_{2}},\ldots,e_{i_{s}}$ is a potential path of $O$ from $u$ to $v$.

As neither of these potential paths contains $e$, Lemma~\ref{lem:preliminary} applies to them in both $O$ and $^{e}O$. We conclude that $G$ has a cycle that is a potential cycle of both $O$ and $^{e}O$, which contains $e_{\min\{i,j\}}$. This contradicts at least one of the hypotheses $e_{i}\notin L^{o}(O)$, $e_{j}\notin L^{o}(^{e}O)$.
\end{proof}

\begin{lemma}
\label{cyclesc2}
$L^{o}(O)\neq L^{o}(O/e)\Rightarrow L^{b}(^{e}O)=L^{b}(O/e)$.
\end{lemma}

\begin{proof}
Every potential cycle $Cy$ of $O$ has the property that $Cy\setminus \{e\}$ is a union of disjoint potential cycles of $O/e$, so $L^{o}(O)\subseteq L^{o}(O/e)$. If $L^{o}(O)\not=L^{o}(O/e)$ then $L^{o}(O/e)$ has an element $e_{i}$ that is not included in $L^{o}(O)$. Then there is a potential cycle $Cy$ of $O/e$, such that $i$ is the smallest index appearing in $Cy$ and neither $Cy$ nor $Cy\cup \{e\}$ is a potential cycle of $O$. It follows that $Cy$ yields a cycle of $G$, of the form $e=e_{i_{1}},e_{i_{2}},\ldots,e_{i_{s}}$, such that $e_{i_{2}},\ldots,e_{i_{s}}$ is a potential path of $O$ but $e_{i_{1}},e_{i_{2}},\ldots,e_{i_{s}}$ is not a potential cycle of~$O$.

Similarly, if $j<p$ then each potential cycle $Cy$ of $^{e}O\setminus \{e_{j}^{\sigma _{b}(e_{j})}\}$ has the property that $Cy\setminus \{e\}$ is a union of disjoint potential cycles of $(^{e}O\setminus \{e_{j}^{\sigma_{b}(e_{j})}\})/e$, so $L^{b}(^{e}O)\subseteq L^{b}(O/e)$. If $L^{b}(^{e}O)\neq L^{b}(O/e)$ then $L^{b}(O/e)$ has an element $e_{j}\notin L^{b}(^{e}O)$. Then there is a potential cycle $Cy^{\prime }$ of  $(^{e}O\setminus \{e_{j}^{\sigma _{b}(e_{j})}\})/e$ such that $j$ is the smallest index appearing in $Cy^{\prime }$ and neither $Cy^{\prime }$ nor $Cy^{\prime }\cup \{e\}$ is a potential cycle of $^{e}O\setminus \{e_{j}^{\sigma _{b}(e_{j})}\}$. Consequently there is a cycle of $G$, of the form $e=e_{j_{1}}^{\prime },e_{j_{2}}^{\prime },\ldots,e_{j_{t}}^{\prime }$, such that $e_{j_{2}}^{\prime },\ldots,e_{j_{t}}^{\prime }$ is a potential path of $^{e}O\setminus \{e_{j}^{\sigma _{b}(e_{j})}\}$ but $e_{j_{1}}^{\prime },\ldots,e_{j_{t}}^{\prime }$ is not a potential cycle of $^{e}O\setminus \{e_{j}^{\sigma _{b}(e_{j})}\}$. Then $P_{1}=e_{i_{2}},\ldots,e_{i_{s}}$ is a potential path of $O$ from $u$ to $v$ and $P_{2}=e_{j_{2}}^{\prime},\ldots,e_{j_{t}}^{\prime }$ is a potential path of $^{e}O\setminus \{e_{j}^{\sigma _{b}(e_{j})}\}$ from $v$ to $u$. Of course $P_{2}$ is also a potential path of $O$.

Lemma~\ref{lem:preliminary} tells us that the concatenation $P_{1}P_{2}$ contains a potential cycle $Cy_{1}$ of $O$ that contains $e_{i}$, and also a potential cycle $Cy_{2}$ of $O$ that contains $e_{j}$. (It is possible that $Cy_{1}=Cy_{2}$.) If $i<j$ then Lemma \ref{lem:preliminary} tells us that the direction of $e_{i}$ in $P_{1}$ is consistent with the preferred orientation of $Cy_{1}$, i.e., $Cy_{1}$ is a potential cycle of $O$. This contradicts the hypothesis that $e_{i}\notin L^{o}(O)$. As $e_{i}$ is oriented in $O$ and $e_{j}$ is not, $i\neq j$. But if $i>j$ then Lemma \ref{lem:preliminary} tells us that the direction of $e_{j}$ in $P_{2}$\ is consistent with the preferred orientation of $Cy_{2}$. It follows that $Cy_{2}$ is a potential cycle of $O\setminus \{e_{j}^{\sigma _{b}(e_{j})}\}$, contradicting the hypothesis that $e_{j}\notin L^{b}(^{e}O)$.
\end{proof}

\begin{lemma}
\label{cyclesc3}
$L^{o}(O)\neq L^{o}(O/e)\Rightarrow I^{o}(O)=I^{o}(O-e)$.
\end{lemma}

\begin{proof}
If $Cy$ is a potential cycle of $O$ then $Cy\setminus \{e\}$ is a union of disjoint potential cycles of $O/e$, so $L^{o}(O)\subseteq L^{o}(O/e)$. Consequently, if $L^{o}(O)\neq L^{o}(O/e)$ then there is an $e_{i}\in L^{o}(O/e)$ that is excluded from $L^{o}(O)$. That is, there is a cycle $e_{i_{1}},\ldots,e_{i_{s}}$ of $G$ which is not a potential cycle of $O$, such that $e_{i_{1}}=e$, $i=\min \{i_{1},\ldots,i_{s}\}$ and $e_{i_{2}},\ldots,e_{i_{s}} $ is a potential path of $O$.

Suppose $I^{o}(O)\not=I^{o}(O-e)$. Each potential cut $Cu$ of $O$ has the property that $Cu\setminus \{e\}$ is a union of disjoint potential cuts of $O-e$, so $I^{o}(O)\subseteq I^{o}(O-e)$. It follows that $G$ has\ an edge $e_{j}\in I^{o}(O-e)$, which is excluded from $I^{o}(O)$. Then there is a potential cut $Cu$ of $O-e$ whose least-indexed element is $e_{j}$, and no such $Cu$ has either $Cu$ or $Cu\cup \{e\}$ a potential cut of $O$.

Let $Cu$ be a potential cut of $O-e$, as described. Then there is a subset $U\subseteq V(G)$ such that $Cu$ includes every edge of $G-e$ with precisely one vertex in $U$, and every element of $Cu$ is unoriented in $O$ or directed away from $U$. As neither $Cu$ nor $Cu\cup \{e\}$ is a potential cut of $O$, $v\in U$ and $u\notin U$. The first paragraph of the proof implies that $e_{i_{2}},\ldots,e_{i_{s}}$ is a potential path of $O$ from $u$ to $v$, which does not include $e$. The first edge of this potential path with a vertex in $U$ is bioriented or directed into $U$, however, so it contradicts the description of $Cu$.
\end{proof}

\begin{lemma}
\label{cyclesc4}
$L^{o}(O)\neq L^{o}(O/e)\Rightarrow I^{u}(O)=I^{u}(O-e)$.
\end{lemma}

\begin{proof}
If $Cy$ is a potential cycle of $O$ then $Cy\setminus \{e\}$ is a union of disjoint potential cycles of $O/e$, so $L^{o}(O)\subseteq L^{o}(O/e)$. Consequently, if $L^{o}(O)\neq L^{o}(O/e)$ then there is an $e_{i}\in L^{o}(O/e)$ that is excluded from $L^{o}(O)$. That is, there is a cycle $e_{i_{1}},\ldots,e_{i_{s}}$ of $G$ which is not a potential cycle of $O$, such that $e_{i_{1}}=e$, $i=\min \{i_{1},\ldots,i_{s}\}$ and $e_{i_{2}},\ldots,e_{i_{s}} $ is a potential path of $O$.

If $j<p$ then each potential cut $Cu$ of $O\cup \{e_{j}^{\sigma_{u}(e_{j})}\}$ has the property that $Cu\setminus \{e\}$ is a union of disjoint potential cuts of $O-e$, so $I^{u}(O)\subseteq I^{u}(O-e)$. If $I^{u}(O)\neq I^{u}(O-e)$ it follows that $G$ has\ an edge $e_{j}\in I^{u}(O-e)$, which is excluded from $I^{u}(O)$. Then there is a potential cut $Cu$ of $(O\cup \{e_{j}^{\sigma _{u}(e_{j})}\})-e$ whose least-indexed element is $e_{j}$, and no such $Cu$ has either $Cu$ or $Cu\cup \{e\}$ a potential cut of $O\cup \{e_{j}^{\sigma _{u}(e_{j})}\}$.

Let $Cu$ be a potential cut of $(O\cup \{e_{j}^{\sigma _{u}(e_{j})}\})-e$, as described. Then there is a subset $U\subseteq V(G)$ such that $Cu$ includes every edge of $G-e$ with precisely one vertex in $U$, and every element of $Cu$ is unoriented in $O\cup \{e_{j}^{\sigma _{u}(e_{j})}\}$ or directed away from $U$. As neither $Cu$ nor $Cu\cup \{e\}$ is a potential cut of $O\cup \{e_{j}^{\sigma _{u}(e_{j})}\}$, $v\in U$ and $u\notin U$. The first paragraph of the proof implies that $e_{i_{2}},\ldots,e_{i_{s}}$ is a potential path of $O\cup \{e_{j}^{\sigma _{u}(e_{j})}\}$ from $u$ to $v$, which does not include $e$. The first edge of this potential path with a vertex in $U$ is bioriented or directed into $U$, however, so it contradicts the description of $Cu$.
\end{proof}

\begin{lemma}
\label{cyclesc5}
$L^{o}(O)\neq L^{o}(O/e)\Rightarrow I^{o}(^{e}O)=I^{o}(O/e)$.
\end{lemma}

\begin{proof}
If $Cy$ is a potential cycle of $O$ then $Cy\setminus \{e\}$ is a union of disjoint potential cycles of $O/e$, so $L^{o}(O)\subseteq L^{o}(O/e)$. Consequently, if $L^{o}(O)\neq L^{o}(O/e)$ then there is an $e_{i}\in L^{o}(O/e)$ that is excluded from $L^{o}(O)$. That is, there is a cycle $e_{i_{1}},\ldots,e_{i_{s}}$ of $G$ which is not a potential cycle of $O$, such that $e_{i_{1}}=e$, $i=\min \{i_{1},\ldots,i_{s}\}$ and $e_{i_{2}},\ldots,e_{i_{s}} $ is a potential path of $O$.

Every potential cut of $O/e$ is also a potential cut of $^{e}O$, so $I^{o}(O/e)\subseteq I^{o}(^{e}O)$. If $I^{o}(O/e)\neq I^{o}(^{e}O)$ it follows that $G$ has\ an edge $e_{j}\in I^{o}(^{e}O)\setminus I^{o}(O/e)$. Then there is a potential cut $Cu$ of $^{e}O$ whose least-indexed element is $e_{j}$, and every such $Cu$ includes~$e$.

Let $Cu$ be a potential cut of $^{e}O$, as described. Then there is a subset $U\subseteq V(G)$ such that $Cu$ includes every edge of $G$ with precisely one vertex in $U$, and every element of $Cu$ is unoriented in $^{e}O$ or directed away from $U$. As $e\in Cu$, $v\in U$ and $u\notin U$. The first paragraph of the proof implies that $e_{i_{2}},\ldots,e_{i_{s}}$ is a potential path of $^{e}O$ from $u$ to $v$, which does not include $e$. The first edge of this potential path with a vertex in $U$ is bioriented or directed into $U $, however, so it contradicts the description of $Cu$.
\end{proof}

\begin{lemma}
\label{cyclesc6}
$L^{o}(O)\neq L^{o}(O/e)\Rightarrow I^{u}(^{e}O)=I^{u}(O/e)$.
\end{lemma}

\begin{proof}
If $Cy$ is a potential cycle of $O$ then $Cy\setminus \{e\}$ is a union of disjoint potential cycles of $O/e$, so $L^{o}(O)\subseteq L^{o}(O/e)$. Consequently, if $L^{o}(O)\neq L^{o}(O/e)$ then there is an $e_{i}\in L^{o}(O/e)$ that is excluded from $L^{o}(O)$. That is, there is a cycle $e_{i_{1}},\ldots,e_{i_{s}}$ of $G$ which is not a potential cycle of $O$, such that $e_{i_{1}}=e$, $i=\min \{i_{1},\ldots,i_{s}\}$ and $e_{i_{2}},\ldots,e_{i_{s}} $ is a potential path of $O$.

If $j<p$ then every potential cut of $(O\cup \{e_{j}^{\sigma_{u}(e_{j})}\})/e$ is also a potential cut of $^{e}O\cup \{e_{j}^{\sigma_{u}(e_{j})}\}$, so $I^{u}(O/e)\subseteq I^{u}(^{e}O)$. If $I^{u}(O/e)\neq I^{u}(^{e}O)$ it follows that $G$ has\ an edge $e_{j}\in I^{u}(^{e}O)\setminus I^{u}(O/e)$. Then there is a potential cut $Cu$ of $^{e}O\cup \{e_{j}^{\sigma _{u}(e_{j})}\}$ whose least-indexed element is $e_{j}$, and every such $Cu$ includes $e$.

Let $Cu$ be a potential cut of $^{e}O\cup \{e_{j}^{\sigma _{u}(e_{j})}\}$, as described. Then there is a subset $U\subseteq V(G)$ such that $Cu$ includes every edge of $G$ with precisely one vertex in $U$, and every element of $Cu$ is unoriented in $^{e}O\cup \{e_{j}^{\sigma _{u}(e_{j})}\}$ or directed away from $U$. As $e\in Cu$, $v\in U$ and $u\notin U$. The first paragraph of the proof implies that $e_{i_{2}},\ldots,e_{i_{s}}$ is a potential path of $^{e}O\cup \{e_{j}^{\sigma _{u}(e_{j})}\}$ from $u$ to $v$, which does not include $e$. The first edge of this potential path with a vertex in $U$ is bioriented or directed into $U$, however, so it contradicts the description of $Cu$.
\end{proof}

\subsection{Consequences of \texorpdfstring{$L^{b}(O)\not=L^{b}(O/e)$}{}}

\begin{lemma}
$L^{b}(O)\neq L^{b}(O/e)\Rightarrow L^{o}(O)=L^{o}(O-e)$ and~$L^{b}(O)=L^{b}(O-e)$ and~$L^{o}(^{e}O)=L^{o}(O/e)$.
\end{lemma}

\begin{proof}
These implications follow from\ earlier results, applied to $^{e}O$ rather than $O$.
\end{proof}

\begin{lemma}
$L^{b}(O)\neq L^{b}(O/e)\Rightarrow L^{b}(^{e}O)=L^{b}(O/e)$.
\end{lemma}

\begin{proof}
If $i<p$ and $Cy$ is a potential cycle of $O\setminus \{e_{i}^{\sigma_{b}(e_{i})}\}$ then $Cy\setminus \{e\}$ is a union of disjoint potential cycles of $(O\setminus \{e_{i}^{\sigma _{b}(e_{i})}\})/e$, so $L^{b}(O)\subseteq L^{b}(O/e)$. Consequently, if $L^{b}(O)\neq L^{b}(O/e)$ then there is an $e_{i}\in L^{b}(O/e)$ that is excluded from $L^{b}(O)$. That is, there is a cycle $e_{i_{1}},\ldots,e_{i_{s}}$ of $G$ which is not a potential cycle of $O\setminus \{e_{i}^{\sigma _{b}(e_{i})}\}$, such that $e_{i_{1}}=e$, $i=\min \{i_{1},\ldots,i_{s}\}$ and $e_{i_{2}},\ldots,e_{i_{s}}$ is a potential path of $O\setminus \{e_{i}^{\sigma _{b}(e_{i})}\}$. 

Similarly, if $L^{b}(^{e}O)\neq L^{b}(O/e)$ then there is an $e_{j}\in L^{b}(O/e)$ that is excluded from $L^{b}(^{e}O)$. That is, there is a cycle $e_{j_{1}},\ldots,e_{j_{t}}$ of $G$ which is not a potential cycle of $O\setminus \{e_{j}^{\sigma _{b}(e_{j})}\}$, such that $e_{j_{1}}=e$, $j=\min\{j_{1},\ldots,j_{t}\}$ and $e_{j_{2}},\ldots,e_{j_{t}}$ is a potential path of $^{e}O\setminus \{e_{j}^{\sigma _{b}(e_{j})}\}$.

Interchanging $O$ and $^{e}O$ if necessary, we may presume that $i\leq j$. Lemma \ref{lem:preliminary} tells us that the closed walk $e_{i_{2}},\ldots,e_{i_{s}}$, $e_{j_{2}},\ldots,e_{j_{t}}$ contains a potential cycle of $O\setminus \{e_{i}^{\sigma _{b}(e_{i})}\}$, which contains $e_{i}$. The existence of such a potential cycle contradicts the hypothesis that $e_{i}\notin L^{b}(O)$.
\end{proof}

\begin{lemma}
$L^{b}(O)\neq L^{b}(O/e)\Rightarrow I^{o}(O)=I^{o}(O-e)$.
\end{lemma}

\begin{proof}
If $i<p$ and $Cy$ is a potential cycle of $O\setminus \{e_{i}^{\sigma_{b}(e_{i})}\}$ then $Cy\setminus \{e\}$ is a union of disjoint potential cycles of $(O\setminus \{e_{i}^{\sigma _{b}(e_{i})}\})/e$, so $L^{b}(O)\subseteq L^{b}(O/e)$. Consequently, if $L^{b}(O)\neq L^{b}(O/e)$ then there is an $e_{i}\in L^{b}(O/e)$ that is excluded from $L^{b}(O)$. That is, there is a cycle $e_{i_{1}},\ldots,e_{i_{s}}$ of $G$ which is not a potential cycle of $O\setminus \{e_{i}^{\sigma _{b}(e_{i})}\}$, such that $e_{i_{1}}=e$, $i=\min \{i_{1},\ldots,i_{s}\}$ and $e_{i_{2}},\ldots,e_{i_{s}}$ is a potential path of $O\setminus \{e_{i}^{\sigma _{b}(e_{i})}\}$.

Suppose $I^{o}(O)\not=I^{o}(O-e)$. Each potential cut $Cu$ of $O$ has the property that $Cu\setminus \{e\}$ is a union of disjoint potential cuts of $O-e$, so $I^{o}(O)\subseteq I^{o}(O-e)$. It follows that $G$ has\ an edge $e_{j}\in I^{o}(O-e)$, which is excluded from $I^{o}(O)$. Then there is a potential cut $Cu$ of $O-e$ whose least-indexed element is $e_{j}$, and no such $Cu$ has either $Cu$ or $Cu\cup \{e\}$ a potential cut of $O$.

Let $Cu$ be a potential cut of $O-e$, as described. Then there is a subset $U\subseteq V(G)$ such that $Cu$ includes every edge of $G-e$ with precisely one vertex in $U$, and every element of $Cu$ is unoriented in $O$ or directed away from $U$. As neither $Cu$ nor $Cu\cup \{e\}$ is a potential cut of $O$, $v\in U$ and $u\notin U$. The first paragraph of the proof implies that $e_{i_{2}},\ldots,e_{i_{s}}$ is a potential path of $O$ from $u$ to $v$, which does not include $e$. The first edge of this potential path with a vertex in $U$ is bioriented or directed into $U$, however, so it contradicts the description of $Cu$.
\end{proof}

\begin{lemma}
$L^{b}(O)\neq L^{b}(O/e)\Rightarrow I^{u}(O)=I^{u}(O-e)$.
\end{lemma}

\begin{proof}
If $i<p$ and $Cy$ is a potential cycle of $O\setminus \{e_{i}^{\sigma_{b}(e_{i})}\}$ then $Cy\setminus \{e\}$ is a union of disjoint potential cycles of $(O\setminus \{e_{i}^{\sigma _{b}(e_{i})}\})/e$, so $L^{b}(O)\subseteq L^{b}(O/e)$. Consequently, if $L^{b}(O)\neq L^{b}(O/e)$ then there is an $e_{i}\in L^{b}(O/e)$ that is excluded from $L^{b}(O)$. That is, there is a cycle $e_{i_{1}},\ldots,e_{i_{s}}$ of $G$ which is not a potential cycle of $O\setminus \{e_{i}^{\sigma _{b}(e_{i})}\}$, such that $e_{i_{1}}=e$, $i=\min \{i_{1},\ldots,i_{s}\}$ and $e_{i_{2}},\ldots,e_{i_{s}}$ is a potential path of $O\setminus \{e_{i}^{\sigma _{b}(e_{i})}\}$.

If $j<p$ then each potential cut $Cu$ of $O\cup \{e_{j}^{\sigma_{u}(e_{j})}\}$ has the property that $Cu\setminus \{e\}$ is a union of disjoint potential cuts of $O-e$, so $I^{u}(O)\subseteq I^{u}(O-e)$. If $I^{u}(O)\neq I^{u}(O-e)$ it follows that $G$ has\ an edge $e_{j}\in I^{u}(O-e)$, which is excluded from $I^{u}(O)$. Then there is a potential cut $Cu$ of $(O\cup \{e_{j}^{\sigma _{u}(e_{j})}\})-e$ whose least-indexed element is $e_{j}$, and no such $Cu$ has either $Cu$ or $Cu\cup \{e\}$ a potential cut of $O\cup \{e_{j}^{\sigma _{u}(e_{j})}\}$.

Let $Cu$ be a potential cut of $(O\cup \{e_{j}^{\sigma _{u}(e_{j})}\})-e$, as described. Then there is a subset $U\subseteq V(G)$ such that $Cu$ includes every edge of $G-e$ with precisely one vertex in $U$, and every element of $Cu$ is unoriented in $O\cup \{e_{j}^{\sigma _{u}(e_{j})}\}$ or directed away from $U$. As neither $Cu$ nor $Cu\cup \{e\}$ is a potential cut of $O\cup \{e_{j}^{\sigma _{u}(e_{j})}\}$, $v\in U$ and $u\notin U$. The first paragraph of the proof implies that $e_{i_{2}},\ldots,e_{i_{s}}$ is a potential path of $O\cup \{e_{j}^{\sigma _{u}(e_{j})}\}$ from $u$ to $v$, which does not include $e$. The first edge of this potential path with a vertex in $U$ is bioriented or directed into $U$, however, so it contradicts the description of $Cu$.
\end{proof}

\begin{lemma}
$L^{b}(O)\neq L^{b}(O/e)\Rightarrow I^{o}(^{e}O)=I^{o}(O/e)$.
\end{lemma}

\begin{proof}
As before, if $L^{b}(O)\neq L^{b}(O/e)$ then there is an $e_{i}\in L^{b}(O/e)$ that is excluded from $L^{b}(O)$, so there is a cycle $e_{i_{1}},\ldots,e_{i_{s}}$ of $G$ which is not a potential cycle of $O\setminus \{e_{i}^{\sigma _{b}(e_{i})}\}$, such that $e_{i_{1}}=e$, $i=\min \{i_{1},\ldots,i_{s}\}$ and $e_{i_{2}},\ldots,e_{i_{s}}$ is a potential path of $O\setminus \{e_{i}^{\sigma _{b}(e_{i})}\}$.

Also $I^{o}(O/e)\subseteq I^{o}(^{e}O)$, so if $I^{o}(^{e}O)\not=I^{o}(O/e)$ then $G$ has\ an edge $e_{j}\in I^{o}(^{e}O)\setminus I^{o}(O/e)$. Then there is a potential cut $Cu$ of $^{e}O$ whose least-indexed element is $e_{j}$, and every such $Cu$ includes $e$.

Let $Cu$ be a potential cut of $^{e}O$, as described. Then there is a subset $U\subseteq V(G)$ such that $Cu$ includes every edge of $G$ with precisely one vertex in $U$, and every element of $Cu$ is unoriented in $^{e}O$ or directed away from $U$. As $e\in Cu$, $v\in U$ and $u\notin U$. The first paragraph of the proof implies that $e_{i_{2}},\ldots,e_{i_{s}}$ is a potential path of $^{e}O$ from $u$ to $v$, which does not include $e$. The first edge of this potential path with a vertex in $U$ is bioriented or directed into $U $, however, so it contradicts the description of $Cu$.
\end{proof}

\begin{lemma}
$L^{b}(O)\neq L^{b}(O/e)\Rightarrow I^{u}(^{e}O)=I^{u}(O/e)$.
\end{lemma}

\begin{proof}
Again, if $L^{b}(O)\neq L^{b}(O/e)$ then there is an $e_{i}\in L^{b}(O/e) $ that is excluded from $L^{b}(O)$, so there is a cycle $e_{i_{1}},\ldots,e_{i_{s}}$ of $G$ which is not a potential cycle of $O\setminus \{e_{i}^{\sigma _{b}(e_{i})}\}$, such that $e_{i_{1}}=e$, $i=\min \{i_{1},\ldots,i_{s}\}$ and $e_{i_{2}},\ldots,e_{i_{s}}$ is a potential path of $O\setminus \{e_{i}^{\sigma _{b}(e_{i})}\}$.

If $j<p$ then every potential cut of $(O\cup \{e_{j}^{\sigma_{u}(e_{j})}\})/e$ is also a potential cut of $^{e}O\cup \{e_{j}^{\sigma_{u}(e_{j})}\}$, so $I^{u}(O/e)\subseteq I^{u}(^{e}O)$. If $I^{u}(O/e)\neq I^{u}(^{e}O)$ it follows that $G$ has\ an edge $e_{j}\in I^{u}(^{e}O)\setminus I^{u}(O/e)$. Then there is a potential cut $Cu$ of $^{e}O\cup \{e_{j}^{\sigma _{u}(e_{j})}\}$ whose least-indexed element is $e_{j}$, and every such $Cu$ includes $e$.

Let $Cu$ be a potential cut of $^{e}O\cup \{e_{j}^{\sigma _{u}(e_{j})}\}$, as described. Then there is a subset $U\subseteq V(G)$ such that $Cu$ includes every edge of $G$ with precisely one vertex in $U$, and every element of $Cu$ is unoriented in $^{e}O\cup \{e_{j}^{\sigma _{u}(e_{j})}\}$ or directed away from $U$. As $e\in Cu$, $v\in U$ and $u\notin U$. The first paragraph of the proof implies that $e_{i_{2}},\ldots,e_{i_{s}}$ is a potential path of $^{e}O\cup \{e_{j}^{\sigma _{u}(e_{j})}\}$ from $u$ to $v$, which does not include $e$. The first edge of this potential path with a vertex in $U$ is bioriented or directed into $U$, however, so it contradicts the description of $Cu$.
\end{proof}

\subsection{Consequences of \texorpdfstring{$I^{o}(O)\not=I^{o}(O-e)$}{}}

\begin{lemma}
$I^{o}(O)\not=I^{o}(O-e)\Rightarrow L^{o}(^{e}O)=L^{o}(O-e)$ and~$L^{b}(^{e}O)=L^{b}(O-e)$ and~$L^{o}(O)=L^{o}(O/e)$ and~$L^{b}(O)=L^{b}(O/e)$.
\end{lemma}

\begin{proof}
These implications follow from\ earlier results, applied to $^{e}O$ rather than $O$.
\end{proof}

\begin{lemma}
$I^{o}(O)\not=I^{o}(O-e)\Rightarrow I^{o}(^{e}O)=I^{o}(O-e)$.
\end{lemma}

\begin{proof}
Suppose $I^{o}(O)\not=I^{o}(O-e)$. Each potential cut $Cu$ of $O$ has the property that $Cu\setminus \{e\}$ is a union of disjoint potential cuts of $O-e$, so $I^{o}(O)\subseteq I^{o}(O-e)$. It follows that $G$ has\ an edge $e_{i}\in I^{o}(O-e)$, which is excluded from $I^{o}(O)$. Then there is a potential cut $Cu$ of $O-e$ whose least-indexed element is $e_{i}$, and no such $Cu$ has either $Cu$ or $Cu\cup \{e\}$ a potential cut of $O$. Similarly, if $I^{o}(^{e}O)\not=I^{o}(O-e)$ then $G$ has\ an edge $e_{j}\in I^{o}(O-e)\setminus I^{o}(^{e}O)$. Then there is a potential cut $Cu^{\prime }$ of $O-e$ whose least-indexed element is $e_{j}$, and no such $Cu^{\prime }$ has either $Cu^{\prime }$ or $Cu^{\prime }\cup \{e\}$ a potential cut of $^{e}O$.

Suppose $Cu$ and $Cu^{\prime }$ are as described. Then there are subsets $U,U^{\prime }\subseteq V(G)$ such that $Cu$ is the set of edges of $G-e$ with precisely one vertex in $U$, $Cu^{\prime }$ is the set of edges of $G-e$ with precisely one vertex in $U^{\prime }$, every edge of $Cu$ is either unoriented or directed away from $U$, and every edge of $Cu^{\prime }$ is either unoriented or directed away from $U^{\prime }$. As neither $Cu$ nor $Cu\cup \{e\}$ is a potential cut of $O$, it must be that $v\in U$ and $u\notin U$; similarly, $u\in U^{\prime }$ and $v\notin U^{\prime }$.

Interchanging $O$ and $^{e}O$ if necessary, we may presume that $i\leq j$. Let $e_{i}$ have vertices $x$ and $y$, with $x\in U$ and $y\notin U$. Then $e_{i}$ is directed from $x$ to $y$.

Suppose $x\in U^{\prime }$. Let $Cu^{\prime \prime }$ be the set of edges of $G$ with precisely one vertex in $U\cap U^{\prime }$; then $Cu^{\prime \prime }$ is a union of disjoint cuts of $G$. As $e_{i}\in Cu^{\prime \prime}$, $Cu^{\prime \prime }\neq \varnothing $. As $Cu^{\prime \prime }\subseteq Cu\cup Cu^{\prime }$, every element of $Cu^{\prime \prime }$ is either unoriented or directed away from $U\cap U^{\prime }$. In addition, $e_{i}$ is the least-indexed element of $Cu^{\prime \prime }$; but then $e_{i}$ is the least-indexed element of a potential cut of $O$ contained in $Cu^{\prime \prime }$, and this contradicts the hypothesis that $e_{i}\notin I^{o}(O)$.

If $x\notin U^{\prime }$ then $y\notin U^{\prime }$ too, as every edge with precisely one vertex in $U^{\prime }$ is directed away from $U^{\prime }$. Let $Cu^{\prime \prime }$ be the set of edges of $G$ with precisely one vertex in $U\cup U^{\prime }$. Again, $Cu^{\prime \prime }\neq \varnothing $ as $e_{i}\in Cu^{\prime \prime }$, and $Cu^{\prime \prime }\subseteq Cu\cup Cu^{\prime }$, so every element of $Cu^{\prime \prime }$ is either unoriented or directed away from $U\cup U^{\prime }$. In addition, $e_{i}$ is the least-indexed element of $Cu^{\prime \prime }$; but again, this contradicts the hypothesis that $e_{i}\notin I^{o}(O)$.
\end{proof}

\begin{lemma}
$I^{o}(O)\not=I^{o}(O-e)\Rightarrow I^{u}(^{e}O)=I^{u}(O-e)$.
\end{lemma}

\begin{proof}
Suppose $I^{o}(O)\not=I^{o}(O-e)$. Each potential cut $Cu$ of $O$ has the property that $Cu\setminus \{e\}$ is a union of disjoint potential cuts of $O-e$, so $I^{o}(O)\subseteq I^{o}(O-e)$. It follows that $G$ has\ an edge $e_{i}\in I^{o}(O-e)$, which is excluded from $I^{o}(O)$. Then there is a potential cut $Cu$ of $O-e$ whose least-indexed element is $e_{i}$, and no such $Cu$ has either $Cu$ or $Cu\cup \{e\}$ a potential cut of $O$. Similarly, if $I^{u}(^{e}O)\not=I^{u}(O-e)$ then $G$ has\ an edge $e_{j}\in I^{u}(O-e)\setminus I^{u}(^{e}O)$. Then there is a potential cut $Cu^{\prime }$ of $(O\cup \{e_{j}^{\sigma _{u}(e_{j})}\})-e$ whose least-indexed element is $e_{j}$, and no such $Cu^{\prime }$ has either $Cu^{\prime }$ or $Cu^{\prime }\cup \{e\}$ a potential cut of $^{e}O\cup \{e_{j}^{\sigma _{u}(e_{j})}\}$.

Suppose $Cu$ and $Cu^{\prime }$ are as described. Then there are subsets $U,U^{\prime }\subseteq V(G)$ such that $Cu$ is the set of edges of $G-e$ with precisely one vertex in $U$, $Cu^{\prime }$ is the set of edges of $G-e$ with precisely one vertex in $U^{\prime }$, every edge of $Cu$ is either unoriented or directed away from $U$ in $O$, and every edge of $Cu^{\prime }$ is either unoriented or directed away from $U^{\prime }$ in $^{e}O\cup \{e_{j}^{\sigma _{u}(e_{j})}\}$.

Suppose $i<j$. Let $e_{i}$ have vertices $x$ and $y$, with $x\in U$ and $y\notin U$. Then $e_{i}$ is directed from $x$ to $y$ in $O$.

If $x\in U^{\prime }$ then $y\in U^{\prime }$ too, as $i<j$. Let $Cu^{\prime \prime }$ be the set of edges of $G$ with precisely one vertex in $U\cap U^{\prime }$; then $Cu^{\prime \prime }$ is a union of disjoint cuts of $G$, and $Cu^{\prime \prime }\neq \varnothing $ as $e_{i}\in Cu^{\prime \prime }$. As $Cu^{\prime \prime }\subseteq Cu\cup Cu^{\prime }$, every element of $Cu^{\prime \prime }$ is either unoriented in $O$ or directed away from $U\cap U^{\prime }$. (The direction of $e_{j}$ in $^{e}O\cup \{e_{j}^{\sigma_{u}(e_{j})}\}$ is irrelevant to this assertion, as $e_{j}$ is unoriented in $O$.) In addition, $e_{i}$ is the least-indexed element of $Cu^{\prime \prime }$; but this contradicts the hypothesis that $e_{i}\notin I^{o}(O)$.

If $x\notin U^{\prime }$ then $y\notin U^{\prime }$ too, as $i<j$. Let $Cu^{\prime \prime }$ be the set of edges of $G$ with precisely one vertex in $U\cup U^{\prime }$. Again, $Cu^{\prime \prime }\neq \varnothing $ as $e_{i}\in Cu^{\prime \prime }$, and $Cu^{\prime \prime }\subseteq Cu\cup Cu^{\prime }$, so every element of $Cu^{\prime \prime }$ is either unoriented or directed away from $U\cup U^{\prime }$. In addition, $e_{i}$ is the least-indexed element of $Cu^{\prime \prime }$; but again, this contradicts the hypothesis that $e_{i}\notin I^{o}(O)$.

We conclude that $i>j$. (N.b. $e_{i}\neq e_{j}$ as $e_{i}$ is oriented in $O$ and $e_{j}$ is not.) Let $e_{j}$ have vertices $x$ and $y$, with $x\in U^{\prime }$ and $y\notin U^{\prime }$. Then $e_{j}$ is directed from $x$ to $y$ in $^{e}O\cup \{e_{j}^{\sigma _{u}(e_{j})}\}$.

Suppose $x\in U$. As $j<i$, it follows that $y\in U$ too. Let $Cu^{\prime \prime }$ be the set of edges of $G$ with precisely one vertex in $U\cap U^{\prime }$; $Cu^{\prime \prime }\neq \varnothing $ as $e_{j}\in Cu^{\prime \prime }$. As $Cu^{\prime \prime }\subseteq Cu\cup Cu^{\prime }$, every element of $Cu^{\prime \prime }$ is either unoriented in $^{e}O\cup \{e_{j}^{\sigma _{u}(e_{j})}\}$ or directed away from $U\cap U^{\prime }$. In addition, $e_{j}$ is the least-indexed element of $Cu^{\prime \prime }$; but this contradicts the hypothesis that $e_{j}\notin I^{u}(^{e}O)$.

Suppose now that $x\not\in U$; then $y\not\in U$ too. Let $Cu^{\prime \prime}$ be the set of edges of $G$ with precisely one vertex in $U\cup U^{\prime} $; $Cu^{\prime \prime }\neq \varnothing $ as $e_{j}\in Cu^{\prime \prime }$. As $Cu^{\prime \prime }\subseteq Cu\cup Cu^{\prime }$, every element of $Cu^{\prime \prime }$ is either unoriented in $^{e}O\cup \{e_{j}^{\sigma_{u}(e_{j})}\}$ or directed away from $U\cup U^{\prime }$. In addition, $e_{j}$ is the least-indexed element of $Cu^{\prime \prime }$; but again, this contradicts the hypothesis that $e_{j}\notin I^{u}(^{e}O)$.
\end{proof}

\begin{lemma}
$I^{o}(O)\not=I^{o}(O-e)\Rightarrow I^{o}(O)=I^{o}(O/e)$.
\end{lemma}

\begin{proof}
Suppose $I^{o}(O)\not=I^{o}(O-e)$. Each potential cut $Cu$ of $O$ has the property that $Cu\setminus \{e\}$ is a union of disjoint potential cuts of $O-e$, so $I^{o}(O)\subseteq I^{o}(O-e)$. It follows that $G$ has\ an edge $e_{i}\in I^{o}(O-e)$, which is excluded from $I^{o}(O)$. Then there is a potential cut $Cu$ of $O-e$ whose least-indexed element is $e_{i}$, and no such $Cu$ has either $Cu$ or $Cu\cup \{e\}$ a potential cut of $O$.

Every potential cut of $O/e$ is a potential cut of $O$, so $I^{o}(O/e)\subseteq I^{o}(O)$. If $I^{o}(O)\not=I^{o}(O/e)$, then, $G$ has an edge $e_{j}\in I^{o}(O)\setminus I^{o}(O/e)$. Then there is a potential cut $Cu^{\prime }$ of $O$ whose least-indexed element is $e_{j}$, and every such $Cu^{\prime }$ includes $e$.

Suppose $Cu$ and $Cu^{\prime }$ are as described. Then there are subsets $U,U^{\prime }\subseteq V(G)$ such that $Cu$ is the set of edges of $G-e$ with precisely one vertex in $U$, $Cu^{\prime }$ is the set of edges of $G$ with precisely one vertex in $U^{\prime }$, every edge of $Cu$ is either unoriented or directed away from $U$ in $O$, and every edge of $Cu^{\prime }$ is either unoriented or directed away from $U^{\prime }$ in $O$. As neither $Cu$ nor $Cu\cup \{e\}$ is a potential cut of $O$, $v\in U$ and $u\notin U$; as $e\in Cu^{\prime }$, $u\in U^{\prime }$ and $v\notin U^{\prime }$.

Suppose $i<j$. Let $e_{i}$ have vertices $x$ and $y$, with $x\in U$ and $y\notin U$. Then $e_{i}$ is directed from $x$ to $y$ in $O$.

If $x\in U^{\prime }$ then $y\in U^{\prime }$ too, as $i<j$. Let $Cu^{\prime \prime }$ be the set of edges of $G$ with precisely one vertex in $U\cap U^{\prime }$; then $Cu^{\prime \prime }$ is a union of disjoint cuts of $G$, and $Cu^{\prime \prime }\neq \varnothing $ as $e_{i}\in Cu^{\prime \prime }$. As $Cu^{\prime \prime }\subseteq Cu\cup Cu^{\prime }$, every element of $Cu^{\prime \prime }$ is either unoriented in $O$ or directed away from $U\cap U^{\prime }$. In addition, $e_{i}$ is the least-indexed element of $Cu^{\prime \prime }$; but this contradicts the hypothesis that $e_{i}\notin I^{o}(O)$.

If $x\notin U^{\prime }$ then $y\notin U^{\prime }$ too, as $i<j$. Let $Cu^{\prime \prime }$ be the set of edges of $G$ with precisely one vertex in  $U\cup U^{\prime }$. Again, $Cu^{\prime \prime }\neq \varnothing $ as $e_{i}\in Cu^{\prime \prime }$, and $Cu^{\prime \prime }\subseteq Cu\cup Cu^{\prime }$, so every element of $Cu^{\prime \prime }$ is either unoriented or directed away from $U\cup U^{\prime }$. In addition, $e_{i}$ is the least-indexed element of $Cu^{\prime \prime }$; but again, this contradicts the hypothesis that $e_{i}\notin I^{o}(O)$.

Suppose now that $i\geq j$. Let $e_{j}$ have vertices $x$ and $y$, with $x\in U^{\prime }$ and $y\notin U^{\prime }$. Then $e_{j}$ is directed from $x $ to $y$ in $O$.

Suppose $x\in U$. Let $Cu^{\prime \prime }$ be the set of edges of $G$ with precisely one vertex in $U\cap U^{\prime }$; $Cu^{\prime \prime }\neq \varnothing $ as $e_{j}\in Cu^{\prime \prime }$. As $Cu^{\prime \prime}\subseteq Cu\cup Cu^{\prime }$, every element of $Cu^{\prime \prime }$ is either unoriented in $O$ or directed away from $U\cap U^{\prime }$. Moreover, the fact that $u,v\notin U\cap U^{\prime }$ implies that $Cu^{\prime \prime }$ provides a potential cut of $O/e$, which contains $e_{j}$. The hypothesis $e_{j}\notin I^{o}(O/e)$ is contradicted because $e_{j}$ is the least-indexed element of~$Cu^{\prime \prime }$.

If $x\not\in U$, instead, then let $Cu^{\prime \prime }$ be the set of edges of $G$ with precisely one vertex in $U\cup U^{\prime }$. As $e_{j}$ is directed from $x$ to $y$ in $O$, it must be that $y\not\in U$ too. Then $e_{j}\in Cu^{\prime \prime }$, so $Cu^{\prime \prime }\neq \varnothing $. As $Cu^{\prime \prime }\subseteq Cu\cup Cu^{\prime }$, every element of $Cu^{\prime \prime }$ is either unoriented in $O$ or directed\ away from $U\cup U^{\prime }$. Moreover, the fact that $u,v\in U\cup U^{\prime }$ implies that $Cu^{\prime \prime }$ provides a potential cut of $O/e$, which contains $e_{j}$. The hypothesis $e_{j}\notin I^{o}(O/e)$ is contradicted because $e_{j}$ is the least-indexed element of~$Cu^{\prime \prime }$.
\end{proof}

\begin{lemma}
$I^{o}(O)\not=I^{o}(O-e)\Rightarrow I^{u}(O)=I^{u}(O/e)$.
\end{lemma}

\begin{proof}
Suppose $I^{o}(O)\not=I^{o}(O-e)$. Each potential cut $Cu$ of $O$ has the property that $Cu\setminus \{e\}$ is a union of disjoint potential cuts of $O-e$, so $I^{o}(O)\subseteq I^{o}(O-e)$. It follows that $G$ has\ an edge $e_{i}\in I^{o}(O-e)$, which is excluded from $I^{o}(O)$. Then there is a potential cut $Cu$ of $O-e$ whose least-indexed element is $e_{i}$, and no such $Cu$ has either $Cu$ or $Cu\cup \{e\}$ a potential cut of $O$.

If $j<p$ then every potential cut of $(O\cup \{e_{j}^{\sigma_{u}(e_{j})}\})/e$ is a potential cut of $O$, so $I^{u}(O/e)\subseteq I^{u}(O)$. If $I^{u}(O)\not=I^{u}(O/e)$, $G$ has\ an edge $e_{j}\in I^{u}(O)\setminus I^{u}(O/e)$. Then there is a potential cut $Cu^{\prime }$ of $O\cup \{e_{j}^{\sigma _{u}(e_{j})}\}$ whose least-indexed element is $e_{j}$, and every such $Cu^{\prime }$ includes $e$.

Suppose $Cu$ and $Cu^{\prime }$ are as described. Then there are subsets $U,U^{\prime }\subseteq V(G)$ such that $Cu$ is the set of edges of $G-e$ with precisely one vertex in $U$, $Cu^{\prime }$ is the set of edges of $G$ with precisely one vertex in $U^{\prime }$, every edge of $Cu$ is either unoriented or directed away from $U$ in $O$, and every edge of $Cu^{\prime }$ is either unoriented or directed away from $U^{\prime }$ in $O\cup \{e_{j}^{\sigma _{u}(e_{j})}\}$. As neither $Cu$ nor $Cu\cup \{e\}$ is a potential cut of $O$, $v\in U$ and $u\notin U$; as $e\in Cu^{\prime }$, $u\in U^{\prime }$ and $v\notin U^{\prime }$.

Suppose $i<j$. Let $e_{i}$ have vertices $x$ and $y$, with $x\in U$ and $y\notin U$. Then $e_{i}$ is directed from $x$ to $y$ in $O$.

If $x\in U^{\prime }$ then $y\in U^{\prime }$ too, as $i<j$. Let $Cu^{\prime \prime }$ be the set of edges of $G$ with precisely one vertex in $U\cap U^{\prime }$; then $Cu^{\prime \prime }$ is a union of disjoint cuts of $G$, and $Cu^{\prime \prime }\neq \varnothing $ as $e_{i}\in Cu^{\prime \prime }$. As $Cu^{\prime \prime }\subseteq Cu\cup Cu^{\prime }$, every element of $Cu^{\prime \prime }$ is either unoriented in $O$ or directed away from $U\cap U^{\prime }$. In addition, $e_{i}$ is the least-indexed element of $Cu^{\prime \prime }$; but this contradicts the hypothesis that $e_{i}\notin I^{o}(O)$.

If $x\notin U^{\prime }$ then $y\notin U^{\prime }$ too, as $i<j$. Let $Cu^{\prime \prime }$ be the set of edges of $G$ with precisely one vertex in $U\cup U^{\prime }$. Again, $Cu^{\prime \prime }\neq \varnothing $ as $e_{i}\in Cu^{\prime \prime }$, and $Cu^{\prime \prime }\subseteq Cu\cup Cu^{\prime }$, so every element of $Cu^{\prime \prime }$ is either unoriented in $O$ or directed away from $U\cup U^{\prime }$. In addition, $e_{i}$ is the least-indexed element of $Cu^{\prime \prime }$; but again, this contradicts the hypothesis that $e_{i}\notin I^{o}(O)$.

As $e_{i}$ is oriented in $O$ and $e_{j}$ is not, $i\neq j$; we conclude that $i>j$. Let $e_{j}$ have vertices $x$ and $y$, with $x\in U^{\prime }$ and $y\notin U^{\prime }$. Then $e_{j}$ is directed from $x$ to $y$ in $O\cup \{e_{j}^{\sigma _{u}(e_{j})}\}$.

Suppose $x\in U$. Let $Cu^{\prime \prime }$ be the set of edges of $G$ with precisely one vertex in $U\cap U^{\prime }$; $Cu^{\prime \prime }\neq \varnothing $ as $e_{j}\in Cu^{\prime \prime }$. As $Cu^{\prime \prime}\subseteq Cu\cup Cu^{\prime }$, every element of $Cu^{\prime \prime }$ is either unoriented in $O\cup \{e_{j}^{\sigma _{u}(e_{j})}\}$ or directed away from $U\cap U^{\prime }$. Moreover, the fact that $u,v\notin U\cap U^{\prime}$ implies that $Cu^{\prime \prime }$ provides a potential cut of $(O\cup \{e_{j}^{\sigma _{u}(e_{j})}\})/e$, which contains $e_{j}$. The hypothesis $e_{j}\notin I^{u}(O/e)$ is contradicted because $e_{j}$ is the least-indexed element of~$Cu^{\prime \prime }$.

If $x\not\in U$ then as $j<i$, $y\not\in U$ too. Let $Cu^{\prime \prime }$ be the set of edges of $G$ with precisely one vertex in $U\cup U^{\prime }$. Then $e_{j}\in Cu^{\prime \prime }$, so $Cu^{\prime \prime }\neq \varnothing $. As $Cu^{\prime \prime }\subseteq Cu\cup Cu^{\prime }$, every element of $Cu^{\prime \prime }$ is either unoriented in $O\cup \{e_{j}^{\sigma _{u}(e_{j})}\}$ or directed\ away from $U\cup U^{\prime }$. Moreover, the fact that $u,v\in U\cup U^{\prime }$ implies that $Cu^{\prime \prime }$ provides a potential cut of $(O\cup \{e_{j}^{\sigma _{u}(e_{j})}\})/e$, which contains $e_{j}$. The hypothesis $e_{j}\notin I^{u}(O/e)$ is contradicted because $e_{j}$ is the least-indexed element of~$Cu^{\prime \prime }$.
\end{proof}

\subsection{Consequences of \texorpdfstring{$I^{u}(O)\not=I^{u}(O-e)$}{}}

\begin{lemma}
$I^{u}(O)\not=I^{u}(O-e)\Rightarrow L^{o}(^{e}O)=L^{o}(O-e)$ and~$L^{b}(^{e}O)=L^{b}(O-e)$ and~$L^{o}(O)=L^{o}(O/e)$ and~$L^{b}(O)=L^{b}(O/e)$ and~$I^{o}(O)=I^{o}(O/e)$.
\end{lemma}

\begin{proof}
These implications follow from\ earlier results, applied to $^{e}O$ rather than $O$.
\end{proof}

\begin{lemma}
$I^{u}(O)\not=I^{u}(O-e)\Rightarrow I^{u}(^{e}O)=I^{u}(O-e)$.
\end{lemma}

\begin{proof}
Suppose $I^{u}(O)\not=I^{u}(O-e)$. Each potential cut $Cu$ of $O$ has the property that $Cu\setminus \{e\}$ is a union of disjoint potential cuts of $O-e$, so $I^{u}(O)\subseteq I^{u}(O-e)$. It follows that $G$ has\ an edge $e_{i}\in I^{u}(O-e)$, which is excluded from $I^{u}(O)$. Then there is a potential cut $Cu$ of $(O\cup \{e_{i}^{\sigma _{u}(e_{i})}\})-e$ whose least-indexed element is $e_{i}$, and no such $Cu$ has either $Cu$ or $Cu\cup \{e\}$ a potential cut of $O\cup \{e_{i}^{\sigma _{u}(e_{i})}\}$. Similarly, if $I^{u}(^{e}O)\not=I^{u}(O-e)$ then $G$ has\ an edge $e_{j}\in I^{u}(O-e)\setminus I^{u}(^{e}O)$. Then there is a potential cut $Cu^{\prime }$ of $(O\cup \{e_{j}^{\sigma _{u}(e_{j})}\})-e$ whose least-indexed element is $e_{j}$, and no such $Cu^{\prime }$ has either $Cu^{\prime }$ or $Cu^{\prime }\cup \{e\}$ a potential cut of $^{e}O\cup\{e_{j}^{\sigma _{u}(e_{j})}\}$.

Suppose $Cu$ and $Cu^{\prime }$ are as described. Then there are subsets $U,U^{\prime }\subseteq V(G)$ such that $Cu$ is the set of edges of $G-e$ with precisely one vertex in $U$, $Cu^{\prime }$ is the set of edges of $G-e$ with precisely one vertex in $U^{\prime }$, every edge of $Cu$ is either unoriented or directed away from $U$ in $O\cup \{e_{i}^{\sigma_{u}(e_{i})}\} $, and every edge of $Cu^{\prime }$ is either unoriented or directed away from $U^{\prime }$ in $^{e}O\cup \{e_{j}^{\sigma_{u}(e_{j})}\} $.

Interchanging $O$ and $^{e}O$ if necessary, we may presume that $i\leq j$. Let $e_{i}$ have vertices $x$ and $y$, with $x\in U$ and $y\notin U$. Then $e_{i}$ is directed from $x$ to $y$ in $O\cup \{e_{i}^{\sigma _{u}(e_{i})}\}$.

Suppose $x\in U^{\prime }$. Let $Cu^{\prime \prime }$ be the set of edges of $G$ with precisely one vertex in $U\cap U^{\prime }$; then $Cu^{\prime \prime }$ is a union of disjoint cuts of $G$. As $e_{i}\in Cu^{\prime \prime}$, $Cu^{\prime \prime }\neq \varnothing $. As $Cu^{\prime \prime }\subseteq Cu\cup Cu^{\prime }$, every element of $Cu^{\prime \prime }$ is either unoriented or directed away from $U\cap U^{\prime }$ in $O\cup \{e_{i}^{\sigma _{u}(e_{i})}\}$. (If $i<j$ then the direction of $e_{j}$ in $^{e}O\cup \{e_{j}^{\sigma _{u}(e_{j})}\}$ is irrelevant to this assertion, as $e_{j}$ is unoriented in $O\cup \{e_{i}^{\sigma _{u}(e_{i})}\}$.) In addition, $e_{i}$ is the least-indexed element of $Cu^{\prime \prime }$; but then $e_{i}$ is the least-indexed element of a potential cut of $O\cup \{e_{i}^{\sigma _{u}(e_{i})}\}$ contained in $Cu^{\prime \prime }$, and this contradicts the hypothesis that $e_{i}\notin I^{u}(O)$.

If $x\notin U^{\prime }$ and $y\notin U^{\prime }$ then let $Cu^{\prime \prime }$ be the set of edges of $G$ with precisely one vertex in $U\cup U^{\prime }$. Again, $Cu^{\prime \prime }\neq \varnothing $ as $e_{i}\in Cu^{\prime \prime }$, and $Cu^{\prime \prime }\subseteq Cu\cup Cu^{\prime }$, so every element of $Cu^{\prime \prime }$ is either unoriented or directed away from $U\cup U^{\prime }$ in $O\cup \{e_{i}^{\sigma _{u}(e_{i})}\}$. In addition, $e_{i}$ is the least-indexed element of $Cu^{\prime \prime }$; but again, this contradicts the hypothesis that $e_{i}\notin I^{u}(O)$.

If $x\notin U^{\prime }$ and $y\in U^{\prime }$ then $e_{i}\in Cu^{\prime \prime }$, so it must be that $i=j$. But this is impossible, as $e_{i}$ is directed from $x$ to $y$ in $O\cup \{e_{i}^{\sigma _{u}(e_{i})}\}$.
\end{proof}

\begin{lemma}
$I^{u}(O)\not=I^{u}(O-e)\Rightarrow I^{o}(O)=I^{o}(O/e)$.
\end{lemma}

\begin{proof}
Suppose $I^{u}(O)\not=I^{u}(O-e)$. If $i<p$ then each potential cut $Cu$ of $O\cup \{e_{i}^{\sigma _{u}(e_{i})}\}$ has the property that $Cu\setminus \{e\}$ is a union of disjoint potential cuts of $(O\cup \{e_{i}^{\sigma_{u}(e_{i})}\})-e$, so $I^{u}(O)\subseteq I^{u}(O-e)$. It follows that $G$ has an edge $e_{i}\in I^{u}(O-e)$, which is excluded from $I^{u}(O)$. Then there is a potential cut $Cu$ of $(O\cup \{e_{i}^{\sigma_{u}(e_{i})}\})-e$ whose least-indexed element is $e_{i}$, and no such $Cu$ has either $Cu$ or $Cu\cup \{e\}$ a potential cut of $O\cup \{e_{i}^{\sigma_{u}(e_{i})}\}$.

Every potential cut of $O/e$ is a potential cut of $O$, so $I^{o}(O/e)\subseteq I^{o}(O)$. If $I^{o}(O)\not=I^{o}(O/e)$, then, $G$ has an edge $e_{j}\in I^{o}(O)\setminus I^{o}(O/e)$. Then there is a potential cut $Cu^{\prime }$ of $O$ whose least-indexed element is $e_{j}$, and every such $Cu^{\prime }$ includes $e$.

Suppose $Cu$ and $Cu^{\prime }$ are as described. Then there are subsets $U,U^{\prime }\subseteq V(G)$ such that $Cu$ is the set of edges of $G-e$ with precisely one vertex in $U$, $Cu^{\prime }$ is the set of edges of $G$ with precisely one vertex in $U^{\prime }$, every edge of $Cu$ is either unoriented or directed away from $U$ in $O\cup \{e_{i}^{\sigma_{u}(e_{i})}\} $, and every edge of $Cu^{\prime }$ is either unoriented or directed away from $U^{\prime }$ in $O$. As neither $Cu$ nor $Cu\cup \{e\}$ is a potential cut of $O\cup \{e_{i}^{\sigma _{u}(e_{i})}\}$, $v\in U$ and $u\notin U$; as $e\in Cu^{\prime }$, $u\in U^{\prime }$ and $v\notin U^{\prime }$.

Suppose $i<j$. Let $e_{i}$ have vertices $x$ and $y$, with $x\in U$ and $y\notin U$. Then $e_{i}$ is directed from $x$ to $y$ in $O\cup \{e_{i}^{\sigma _{u}(e_{i})}\}$.

If $x\in U^{\prime }$ then $y\in U^{\prime }$ too, as $i<j$. Let $Cu^{\prime \prime }$ be the set of edges of $G$ with precisely one vertex in $U\cap U^{\prime }$; then $Cu^{\prime \prime }$ is a union of disjoint cuts of $G$, and $Cu^{\prime \prime }\neq \varnothing $ as $e_{i}\in Cu^{\prime \prime }$. As $Cu^{\prime \prime }\subseteq Cu\cup Cu^{\prime }$, every element of $Cu^{\prime \prime }$ is either unoriented in $O\cup \{e_{i}^{\sigma_{u}(e_{i})}\}$ or directed away from $U\cap U^{\prime }$. In addition, $e_{i}$ is the least-indexed element of $Cu^{\prime \prime }$; but this contradicts the hypothesis that $e_{i}\notin I^{u}(O)$.

If $x\notin U^{\prime }$ then $y\notin U^{\prime }$ too, as $i<j$. Let $Cu^{\prime \prime }$ be the set of edges of $G$ with precisely one vertex in  $U\cup U^{\prime }$. Again, $Cu^{\prime \prime }\neq \varnothing $ as $e_{i}\in Cu^{\prime \prime }$, and $Cu^{\prime \prime }\subseteq Cu\cup Cu^{\prime }$, so every element of $Cu^{\prime \prime }$ is either unoriented in $O\cup \{e_{i}^{\sigma _{u}(e_{i})}\}$ or directed away from $U\cup U^{\prime }$. In addition, $e_{i}$ is the least-indexed element of $Cu^{\prime \prime }$; but again, this contradicts the hypothesis that $e_{i}\notin I^{u}(O)$.

We conclude that $i>j$. (N.b. $e_{i}\neq e_{j}$ as $e_{i}$ is unoriented in $O$, and $e_{j}$ is oriented.) Let $e_{j}$ have vertices $x$ and $y$, with $x\in U^{\prime }$ and $y\notin U^{\prime }$. Then $e_{j}$ is directed from $x $ to $y$ in $O$.

Suppose $x\in U$. As $j<i$, it follows that $y\in U$ too. Let $Cu^{\prime \prime }$ be the set of edges of $G$ with precisely one vertex in $U\cap U^{\prime }$; $Cu^{\prime \prime }\neq \varnothing $ as $e_{j}\in Cu^{\prime \prime }$. As $Cu^{\prime \prime }\subseteq Cu\cup Cu^{\prime }$, every element of $Cu^{\prime \prime }$ is either unoriented in $O$ or directed away from $U\cap U^{\prime }$. Moreover, the fact that $u,v\notin U\cap U^{\prime }$ implies that $Cu^{\prime \prime }$ provides a potential cut of $O/e$, which contains $e_{j}$. The hypothesis $e_{j}\notin I^{o}(O/e)$ is contradicted because $e_{j}$ is the least-indexed element of~$Cu^{\prime \prime }$.

If $x\not\in U$ then $y\not\in U$ too, as $j>i$. Let $Cu^{\prime \prime }$ be the set of edges of $G$ with precisely one vertex in $U\cup U^{\prime }$; $Cu^{\prime \prime }\neq \varnothing $ as $e_{j}\in Cu^{\prime \prime }$. As $Cu^{\prime \prime }\subseteq Cu\cup Cu^{\prime }$, every element of $Cu^{\prime \prime }$ is either unoriented in $O$ or directed\ away from $U\cup U^{\prime }$. Moreover, the fact that $u,v\in U\cup U^{\prime }$ implies that $Cu^{\prime \prime }$ provides a potential cut of $O/e$, which contains $e_{j}$. The hypothesis $e_{j}\notin I^{o}(O/e)$ is contradicted because $e_{j}$ is the least-indexed element of~$Cu^{\prime \prime }$.
\end{proof}

\begin{lemma}
$I^{u}(O)\not=I^{u}(O-e)\Rightarrow I^{u}(O)=I^{u}(O/e)$.
\end{lemma}

\begin{proof}
Suppose $I^{u}(O)\not=I^{u}(O-e)$. If $i<p$ then each potential cut $Cu$ of $O\cup \{e_{i}^{\sigma _{u}(e_{i})}\}$ has the property that $Cu\setminus \{e\}$ is a union of disjoint potential cuts of $(O\cup \{e_{i}^{\sigma_{u}(e_{i})}\})-e$, so $I^{u}(O)\subseteq I^{u}(O-e)$. It follows that $G$ has an edge $e_{i}\in I^{u}(O-e)$, which is excluded from $I^{u}(O)$. Then there is a potential cut $Cu$ of $(O\cup \{e_{i}^{\sigma_{u}(e_{i})}\})-e$ whose least-indexed element is $e_{i}$, and no such $Cu$ has either $Cu$ or $Cu\cup \{e\}$ a potential cut of $O\cup \{e_{i}^{\sigma_{u}(e_{i})}\}$.

If $j<p$ then every potential cut of $(O\cup \{e_{j}^{\sigma_{u}(e_{j})}\})/e$ is a potential cut of $O\cup \{e_{j}^{\sigma_{u}(e_{j})}\}$, so $I^{u}(O/e)\subseteq I^{u}(O)$. If $I^{u}(O)\not=I^{u}(O/e)$, $G$ has\ an edge $e_{j}\in I^{u}(O)\setminus I^{u}(O/e)$. Then there is a potential cut $Cu^{\prime }$ of $O\cup \{e_{j}^{\sigma _{u}(e_{j})}\}$ whose least-indexed element is $e_{j}$, and every such $Cu^{\prime }$ includes $e$.

Suppose $Cu$ and $Cu^{\prime }$ are as described. Then there are subsets $U,U^{\prime }\subseteq V(G)$ such that $Cu$ is the set of edges of $G-e$ with precisely one vertex in $U$, $Cu^{\prime }$ is the set of edges of $G$ with precisely one vertex in $U^{\prime }$, every edge of $Cu$ is either unoriented or directed away from $U$ in $O\cup \{e_{i}^{\sigma_{u}(e_{i})}\} $, and every edge of $Cu^{\prime }$ is either unoriented or directed away from $U^{\prime }$ in $O\cup \{e_{j}^{\sigma _{u}(e_{j})}\}$. As neither $Cu$ nor $Cu\cup \{e\}$ is a potential cut of $O\cup \{e_{i}^{\sigma _{u}(e_{i})}\}$, $v\in U$ and $u\notin U$; as $e\in Cu^{\prime }$, $u\in U^{\prime }$ and $v\notin U^{\prime }$.

Suppose $i<j$. Let $e_{i}$ have vertices $x$ and $y$, with $x\in U$ and $y\notin U$. Then $e_{i}$ is directed from $x$ to $y$ in $O\cup \{e_{i}^{\sigma _{u}(e_{i})}\}$.

If $x\in U^{\prime }$ then $y\in U^{\prime }$ too, as $i<j$. Let $Cu^{\prime \prime }$ be the set of edges of $G$ with precisely one vertex in $U\cap U^{\prime }$; then $Cu^{\prime \prime }$ is a union of disjoint cuts of $G$, and $Cu^{\prime \prime }\neq \varnothing $ as $e_{i}\in Cu^{\prime \prime }$. As $Cu^{\prime \prime }\subseteq Cu\cup Cu^{\prime }$, every element of $Cu^{\prime \prime }$ is either unoriented in $O\cup \{e_{i}^{\sigma_{u}(e_{i})}\}$ or directed away from $U\cap U^{\prime }$. In addition, $e_{i}$ is the least-indexed element of $Cu^{\prime \prime }$; but this contradicts the hypothesis that $e_{i}\notin I^{u}(O)$.

If $x\notin U^{\prime }$ then $y\notin U^{\prime }$ too, as $i<j$. Let $Cu^{\prime \prime }$ be the set of edges of $G$ with precisely one vertex in $U\cup U^{\prime }$. Again, $Cu^{\prime \prime }\neq \varnothing $ as $e_{i}\in Cu^{\prime \prime }$, and $Cu^{\prime \prime }\subseteq Cu\cup Cu^{\prime }$, so every element of $Cu^{\prime \prime }$ is either unoriented in $O\cup \{e_{i}^{\sigma _{u}(e_{i})}\}$ or directed away from $U\cup U^{\prime }$. In addition, $e_{i}$ is the least-indexed element of $Cu^{\prime \prime }$; but again, this contradicts the hypothesis that $e_{i}\notin I^{u}(O)$.

We conclude that $i\geq j$. Let $e_{j}$ have vertices $x$ and $y$, with $x\in U^{\prime }$ and $y\notin U^{\prime }$. Then $e_{j}$ is directed from $x $ to $y$ in $O\cup \{e_{j}^{\sigma _{u}(e_{j})}\}$.

Suppose $x\in U$. Let $Cu^{\prime \prime }$ be the set of edges of $G$ with precisely one vertex in $U\cap U^{\prime }$; $Cu^{\prime \prime }\neq \varnothing $ as $e_{j}\in Cu^{\prime \prime }$. As $Cu^{\prime \prime}\subseteq Cu\cup Cu^{\prime }$, every element of $Cu^{\prime \prime }$ is either unoriented in $O\cup \{e_{j}^{\sigma _{u}(e_{j})}\}$ or directed away from $U\cap U^{\prime }$. Moreover, the fact that $u,v\notin U\cap U^{\prime}$ implies that $Cu^{\prime \prime }$ provides a potential cut of $(O\cup \{e_{j}^{\sigma _{u}(e_{j})}\})/e$, which contains $e_{j}$. The hypothesis $e_{j}\notin I^{u}(O/e)$ is contradicted because $e_{j}$ is the least-indexed element of~$Cu^{\prime \prime }$.

If $x\not\in U$, instead, then let $Cu^{\prime \prime }$ be the set of edges of $G$ with precisely one vertex in $U\cup U^{\prime }$. Note that if $i>j$ then $y\notin U$ because $e_{j}\notin Cu$, and if $i=j$ then $y\notin U$ because no edge is directed into $U$ in $O\cup \{e_{i}^{\sigma _{u}(e_{i})}\}$. It follows that $e_{j}\in Cu^{\prime \prime }$, so $Cu^{\prime \prime}\neq \varnothing $. As $Cu^{\prime \prime }\subseteq Cu\cup Cu^{\prime }$, every element of $Cu^{\prime \prime }$ is either unoriented in $O\cup \{e_{j}^{\sigma _{u}(e_{j})}\}$ or directed\ away from $U\cup U^{\prime }$. Moreover, the fact that $u,v\in U\cup U^{\prime }$ implies that $Cu^{\prime \prime }$ provides a potential cut of $(O\cup \{e_{j}^{\sigma_{u}(e_{j})}\})/e$, which contains $e_{j}$. The hypothesis $e_{j}\notin I^{u}(O/e)$ is contradicted because $e_{j}$ is the least-indexed element of $Cu^{\prime \prime }$.
\end{proof}

\subsection{Consequences of \texorpdfstring{$I^{o}(O)\not=I^{o}(O/e)$}{}}

\begin{lemma}
$I^{o}(O)\not=I^{o}(O/e)\Rightarrow L^{o}(O)=L^{o}(O-e)$ and~$L^{b}(O)=L^{b}(O-e)$ and~$L^{o}(^{e}O)=L^{o}(O/e)$ and~$L^{b}(^{e}O)=L^{b}(O/e)$ and~$I^{o}(^{e}O)=I^{o}(O/e)$ and~$I^{u}(O)=I^{u}(O-e)$.
\end{lemma}

\begin{proof}
These implications follow from earlier results, applied to $^{e}O$ rather than $O$.
\end{proof}

\begin{lemma}
$I^{o}(O)\not=I^{o}(O/e)\Rightarrow I^{o}(^{e}O)=I^{o}(O/e)$.
\end{lemma}

\begin{proof}
Suppose $I^{o}(O)\not=I^{o}(O/e)$. Every potential cut of $O/e$ is also a potential cut of $O$, so $I^{o}(O/e)\subseteq I^{o}(O)$. It follows that $G $ has\ an edge $e_{i}\in I^{o}(O)$, which is excluded from $I^{o}(O/e)$. Then there is a potential cut $Cu$ of $O$ whose least-indexed element is $e_{i}$, and every such $Cu$ includes $e$. Similarly, if $I^{o}(^{e}O)\not=I^{o}(O/e)$ then $G$ has an edge $e_{j}\in I^{o}(^{e}O)$, which is excluded from $I^{o}(O/e)$. Then there is a potential cut $Cu^{\prime }$ of $^{e}O$ whose least-indexed element is $e_{j}$, and every such $Cu^{\prime }$ includes $e$.

Interchanging $O$ and $^{e}O$ if necessary, we may presume that $i\leq j$.

Suppose $Cu$ and $Cu^{\prime }$ are as described. Then there are subsets $U,U^{\prime }\subseteq V(G)$ such that $Cu$ is the set of edges of $G$ with precisely one vertex in $U$, $Cu^{\prime }$ is the set of edges of $G$ with precisely one vertex in $U^{\prime }$, every edge of $Cu$ is either unoriented or directed away from $U$ in $O$, and every edge of $Cu^{\prime }$ is either unoriented or directed away from $U^{\prime }$ in $^{e}O$. As $e\in Cu\cap Cu^{\prime }$, $u\in U\setminus U^{\prime }$ and $v\in U^{\prime}\setminus U$.

Let $e_{i}$ have vertices $x$ and $y$, with $x\in U$ and $y\notin U$. Then $e_{i}$ is directed from $x$ to $y$ in $O$.

If $x\in U^{\prime }$, let $Cu^{\prime \prime }$ be the set of edges of $G$ with precisely one vertex in $U\cap U^{\prime }$. Then $Cu^{\prime \prime }$ is a union of disjoint cuts of $G$, and $Cu^{\prime \prime }\neq \varnothing$ as $e_{i}\in Cu^{\prime \prime }$. As $Cu^{\prime \prime }\subseteq Cu\cup Cu^{\prime }$, every element of $Cu^{\prime \prime }$ is either unoriented in $O$ or directed away from $U\cap U^{\prime }$. In addition, $e_{i}$ is the least-indexed element of $Cu^{\prime \prime }$. But $e\notin Cu^{\prime \prime }$, and this contradicts the hypothesis that $e_{i}\notin I^{o}(O/e) $.

If $x\notin U^{\prime }$, then $y\notin U^{\prime }$ as no edge is directed into $U^{\prime }$. Let $Cu^{\prime \prime }$ be the set of edges of $G$ with precisely one vertex in $U\cup U^{\prime }$. Again, $Cu^{\prime \prime}\neq \varnothing $ as $e_{i}\in Cu^{\prime \prime }$, and $Cu^{\prime \prime }\subseteq Cu\cup Cu^{\prime }$, so every element of $Cu^{\prime \prime }$ is either unoriented or directed away from $U\cup U^{\prime }$. In addition, $e_{i}$ is the least-indexed element of $Cu^{\prime \prime }$; but again, the fact that $e\notin Cu^{\prime \prime }$ contradicts the hypothesis that $e_{i}\notin I^{o}(O/e)$.
\end{proof}

\begin{lemma}
$I^{o}(O)\not=I^{o}(O/e)\Rightarrow I^{u}(^{e}O)=I^{u}(O/e)$.
\end{lemma}

\begin{proof}
Suppose $I^{o}(O)\not=I^{o}(O/e)$. Every potential cut of $O/e$ is also a potential cut of $O$, so $I^{o}(O/e)\subseteq I^{o}(O)$. It follows that $G $ has\ an edge $e_{i}\in I^{o}(O)$, which is excluded from $I^{o}(O/e)$. Then there is a potential cut $Cu$ of $O$ whose least-indexed element is $e_{i}$, and every such $Cu$ includes $e$. Similarly, if $I^{u}(^{e}O)\not=I^{u}(O/e)$ then $G$ has an edge $e_{j}\in I^{u}(^{e}O)$, which is excluded from $I^{u}(O/e)$. Then there is a potential cut $Cu^{\prime }$ of $^{e}O\cup \{e_{j}^{\sigma_{u}(e_{j})}\}$ whose least-indexed element is $e_{j}$, and every such $Cu^{\prime }$ includes $e$.

Suppose $Cu$ and $Cu^{\prime }$ are as described. Then there are subsets $U,U^{\prime }\subseteq V(G)$ such that $Cu$ is the set of edges of $G$ with precisely one vertex in $U$, $Cu^{\prime }$ is the set of edges of $G$ with precisely one vertex in $U^{\prime }$, every edge of $Cu$ is either unoriented or directed away from $U$ in $O$, and every edge of $Cu^{\prime }$ is either unoriented or directed away from $U^{\prime }$ in $^{e}O\cup \{e_{j}^{\sigma _{u}(e_{j})}\}$. As $e\in Cu\cap Cu^{\prime }$, $u\in U\setminus U^{\prime }$ and $v\in U^{\prime }\setminus U$.

Suppose $i<j$, and let $e_{i}$ have vertices $x\in U$ and $y\notin U$. Then $e_{i}$ is directed from $x$ to $y$ in $O$. Notice that $i<j$ implies that either $x,y\in U^{\prime }$ or $x,y\notin U^{\prime }$.

If $x,y\in U^{\prime }$, then let $Cu^{\prime \prime }$ be the set of edges of $G$ with precisely one vertex in $U\cap U^{\prime }$. Then $Cu^{\prime \prime }$ is a union of disjoint cuts of $G$, and $Cu^{\prime \prime }\neq \varnothing $ as $e_{i}\in Cu^{\prime \prime }$. As $Cu^{\prime \prime }\subseteq Cu\cup Cu^{\prime }$, every element of $Cu^{\prime \prime }$ is either unoriented in $O$ or directed away from $U\cap U^{\prime }$. In addition, $e_{i}$ is the least-indexed element of $Cu^{\prime \prime }$. But  $e\notin Cu^{\prime \prime }$, so we have contradicted the hypothesis that $e_{i}\notin I^{o}(O/e)$.

If $x,y\notin U^{\prime }$, let $Cu^{\prime \prime }$ be the set of edges of  $G$ with precisely one vertex in $U\cup U^{\prime }$. Again, $Cu^{\prime \prime }\neq \varnothing $ as $e_{i}\in Cu^{\prime \prime }$, and $Cu^{\prime \prime }\subseteq Cu\cup Cu^{\prime }$, so every element of $Cu^{\prime \prime }$ is either unoriented in $O$ or directed away from $U\cup U^{\prime }$. In addition, $e_{i}$ is the least-indexed element of $Cu^{\prime \prime }$; but again, $e\notin Cu^{\prime \prime }$ contradicts the hypothesis that $e_{i}\notin I^{o}(O/e)$.

As $i<j$ leads to a contradiction, and $i=j$ contradicts the fact that $e_{i} $ is oriented in $O$ and $e_{j}$ is not, we conclude that $i>j$. Let $e_{j}$ have vertices $x\in U^{\prime }$ and $y\notin U^{\prime }$; then $e_{j}$ is directed from $x$ to $y$ in $^{e}O\cup \{e_{j}^{\sigma_{u}(e_{j})}\}$. Notice that $j<i$ implies either $x,y\in U$ or $x,y\notin U$. We derive contradictions just as before, using $U\cap U^{\prime }$ if $x,y\in U$ and $U\cup U^{\prime }$ if $x,y\notin U$.
\end{proof}

\subsection{Consequences of \texorpdfstring{$I^{u}(O)\not=I^{u}(O/e)$}{}}

\begin{lemma}
$I^{u}(O)\not=I^{u}(O/e)\Rightarrow L^{o}(O)=L^{o}(O-e)$ and~$L^{b}(O)=L^{b}(O-e)$ and~$L^{o}(^{e}O)=L^{o}(O/e)$ and~$L^{b}(^{e}O)=L^{b}(O/e)$ and~$I^{o}(^{e}O)=I^{o}(O/e)$ and~$I^{u}(O)=I^{u}(O-e)$ and~$I^{o}(^{e}O)=I^{o}(O/e)$.
\end{lemma}

\begin{proof}
These implications follow from\ earlier results, applied to $^{e}O$ rather than $O$.
\end{proof}

\begin{lemma}
$I^{u}(O)\not=I^{u}(O/e)\Rightarrow I^{u}(^{e}O)=I^{u}(O/e)$.
\end{lemma}

\begin{proof}
Suppose $I^{u}(O)\not=I^{u}(O/e)$. If $i<p$ then every potential cut of $(O\cup \{e_{i}^{\sigma _{u}(e_{i})}\})/e$ is also a potential cut of $O\cup \{e_{i}^{\sigma _{u}(e_{i})}\}$, so $I^{u}(O/e)\subseteq I^{u}(O)$. It follows that $G$ has\ an edge $e_{i}\in I^{u}(O)$, which is excluded from $I^{u}(O/e)$. Then there is a potential cut $Cu$ of $O\cup \{e_{i}^{\sigma_{u}(e_{i})}\}$ whose least-indexed element is $e_{i}$, and every such $Cu$ includes $e$. Similarly, if $I^{u}(^{e}O)\not=I^{u}(O/e)$ then $G$ has\ an edge $e_{j}\in I^{u}(^{e}O)$, which is excluded from $I^{u}(O/e)$. Then there is a potential cut $Cu^{\prime }$ of $^{e}O\cup \{e_{j}^{\sigma _{u}(e_{j})}\}$ whose least-indexed element is $e_{j}$, and every such $Cu^{\prime }$ includes $e$.

Suppose $Cu$ and $Cu^{\prime }$ are as described. Then there are subsets $U,U^{\prime }\subseteq V(G)$ such that $Cu$ is the set of edges of $G$ with precisely one vertex in $U$, $Cu^{\prime }$ is the set of edges of $G$ with precisely one vertex in $U^{\prime }$, every edge of $Cu$ is either unoriented or directed away from $U$ in $O\cup \{e_{i}^{\sigma_{u}(e_{i})}\} $, and every edge of $Cu^{\prime }$ is either unoriented or directed away from $U^{\prime }$ in $^{e}O\cup \{e_{j}^{\sigma_{u}(e_{j})}\} $. As $e\in Cu\cap Cu^{\prime }$, $u\in U\setminus U^{\prime} $ and $v\in U^{\prime }\setminus U$.

Interchanging $O$ and $^{e}O$ if necessary, we may presume that $i\leq j$. Let $e_{i}$ have vertices $x\in U$ and $y\notin U$. Then $e_{i}$ is directed from $x$ to $y$ in $O\cup \{e_{i}^{\sigma _{u}(e_{i})}\}$.

If $x\in U^{\prime }$, let $Cu^{\prime \prime }$ be the set of edges of $G$ with precisely one vertex in $U\cap U^{\prime }$. Then $Cu^{\prime \prime }$ is a union of disjoint cuts of $G$, and $Cu^{\prime \prime }\neq \varnothing $ as $e_{i}\in Cu^{\prime \prime }$. As $Cu^{\prime \prime }\subseteq Cu\cup Cu^{\prime }$, every element of $Cu^{\prime \prime }$ is either unoriented in $O\cup \{e_{i}^{\sigma _{u}(e_{i})}\}$ or directed away from $U\cap U^{\prime }$. In addition, $e_{i}$ is the least-indexed element of $Cu^{\prime \prime }$. But $e\notin Cu^{\prime \prime }$, and this contradicts the hypothesis that $e_{i}\notin I^{u}(O/e)$.

Suppose $x\notin U^{\prime }$. If $i<j$ then $y\notin U^{\prime }$, as $e_{i}\notin Cu^{\prime }$. If $i=j$ then $y\notin U^{\prime }$, as no edge is directed into $U^{\prime }$ in $^{e}O\cup \{e_{j}^{\sigma _{u}(e_{j})}\}$. Either way, $y\notin U^{\prime }$. Let $Cu^{\prime \prime }$ be the set of edges of $G$ with precisely one vertex in $U\cup U^{\prime }$. Again, $Cu^{\prime \prime }\neq \varnothing $ as $e_{i}\in Cu^{\prime \prime }$, and $Cu^{\prime \prime }\subseteq Cu\cup Cu^{\prime }$, so every element of $Cu^{\prime \prime }$ is either unoriented in $O\cup \{e_{i}^{\sigma_{u}(e_{i})}\}$ or directed away from $U\cup U^{\prime }$. In addition, $e_{i}$ is the least-indexed element of $Cu^{\prime \prime }$; but again, the fact that $e\notin C^{\prime \prime }$ contradicts the hypothesis that $e_{i}\notin I^{u}(O/e)$.
\end{proof}

\end{document}